\documentclass[12pt]{amsart}
\usepackage{amssymb,amsmath,tabularx,mathrsfs,mathbbol,yfonts,upgreek}
\usepackage{amsthm,verbatim,comment}
\usepackage[bookmarks=true]{hyperref}
\usepackage{pstricks,pst-node,pst-plot}
\usepackage{geometry}
\usepackage{stmaryrd}
\usepackage{paralist}
\usepackage[all]{xy}
\usepackage{array}
\usepackage{url}
\usepackage{longtable}
\usepackage{bbm}
\usepackage{dutchcal}
\numberwithin{equation}{section}

\DeclareSymbolFontAlphabet{\mathbb}{AMSb}
\DeclareSymbolFontAlphabet{\mathbbl}{bbold}

\newtheorem{thm}{Theorem}[section]

\newtheorem{lem}[thm]{Lemma}

\newtheorem{cor}[thm]{Corollary}

\theoremstyle{definition}

\newtheorem{nota}[thm]{Notation}
\newtheorem{eg}[thm]{Example}
\newtheorem{rem}[thm]{Remark}

\newtheorem*{rem*}{Remarks}

\newtheoremstyle{case}{}{}{}{}{}{:}{ }{}
\theoremstyle{case}

\newcommand{\F}{\mathbb{F}}

\title[Terwilliger $\F$-algebras of direct products of group divisible schemes]{On Terwilliger $\F$-algebras of direct products of group divisible association schemes}
\begin{document}
\author{Yu Jiang}
\address[Y. Jiang]{School of Mathematical Sciences, Anhui University (Qingyuan Campus), No. 111, Jiulong Road, Hefei, 230601, China}
\email[Y. Jiang]{jiangyu@ahu.edu.cn}
\begin{abstract}
The Terwilliger algebras of association schemes over an arbitrary field $\F$ were briefly called the Terwilliger $\F$-algebras of association schemes in \cite{J2}. In this paper, the Terwilliger $\F$-algebras of direct products of group divisible association schemes are studied. The centers, the semisimplicity, the Jacobson radicals and their nilpotent indices, the Wedderburn-Artin decompositions of the Terwilliger $\F$-algebras of direct products of group divisible association schemes are obtained.
\end{abstract}
\maketitle
\vspace{-1.5em}
\section{Introduction}
Association schemes, briefly called schemes, are extensively studied as important objects in algebraic combinatorics. In particular, they are known to have connections with many different mathematical objects, such as groups, graphs, codes, designs, and so on. Conversely, many different tools have already been used to study schemes.

The subconstituent algebras of commutative schemes, introduced by Terwilliger in \cite{T1}, are new tools for studying schemes. They are finite-dimensional semisimple associative $\mathbb{C}$-algebras and are also known as the Terwilliger algebras of commutative schemes. In \cite{Han}, Hanaki defined the Terwilliger algebras for an arbitrary scheme and an arbitrary commutative unital ring. In \cite{J2}, we briefly called the Terwilliger algebras of schemes over an arbitrary field $\F$ the Terwilliger $\F$-algebras of schemes. So the Terwilliger algebras of commutative schemes are exactly their Terwilliger $\mathbb{C}$-algebras.

The Terwilliger $\mathbb{C}$-algebras of many commutative schemes have been investigated (for example, see \cite{BST, CD, LMP, LM, LMW, MA, T1, T2, T3, TY}). However, the investigation of the Terwilliger $\F$-algebras of schemes is almost completely open (see \cite{Her}). In \cite{CX,J2}, the authors presented some ring-theoretical properties of the Terwilliger $\F$-algebras of quasi-thin schemes. In this paper, we study the Terwilliger $\F$-algebras of direct products of group divisible schemes from the viewpoint of ring theory. As the main results of this paper, we obtain the centers, the semisimplicity, the Jacobson radicals and their nilpotent indices, the Wedderburn-Artin decompositions of these algebras (see Theorems \ref{T;Center}, \ref{T;Semisimplicity}, \ref{T;Jacobson}, \ref{T;Decomposition}, respectively). These main results of this paper contribute to understanding the Terwilliger $\F$-algebras of direct products of schemes.

The organization of this paper is as follows: In Section 2, we list the basic notation and preliminaries. In Section 3, we present two $\F$-bases for a Terwilliger $\F$-algebra of a direct product of group divisible schemes. Theorems \ref{T;Center}, \ref{T;Semisimplicity}, \ref{T;Jacobson} are proved in Sections 4, 5, 6, respectively. Sections 7 and 8 contain the proof of Theorem \ref{T;Decomposition}.
\section{Basic notation and preliminaries}
For a general background on association schemes, the reader may refer to \cite{B,Z}.
\subsection{Conventions}
Let $\mathbb{N}$ be the set of all natural numbers and $\mathbb{N}_0=\mathbb{N}\cup\{0\}$. If
$g, h\in\mathbb{N}_0$, let $[g, h]=\{a: a\in\mathbb{N}_0, g\leq a\leq h\}$. If $g\in\mathbb{N}$ and $\mathbb{U}_1, \mathbb{U}_2, \ldots, \mathbb{U}_g$ are sets, let $\prod_{h=1}^g\mathbb{U}_h$ be the cartesian product $\mathbb{U}_1\times\mathbb{U}_2\times\cdots\times\mathbb{U}_g$ and $\mathbf{i}_j$ be the $\mathbb{U}_j$-component of a $g$-tuple $\mathbf{i}$ in $\prod_{h=1}^g\mathbb{U}_h$ for any $j\in[1, g]$. Call an association scheme a scheme. Fix a field $\F$ of characteristic $p$. Let $\delta_{g, h}$ be the Kronecker delta of the symbols $g$ and $h$ whose values are in $\F$.
The addition, the multiplication, and the scalar multiplication of $\F$-matrices in this paper are the usual matrix operations.
Let $\F^{\times}$ be the set of all invertible elements in $\F$. Let $\F_p$ be the prime subfield of $\F$. If $\mathbb{Z}$ is the ring of integers and $g\in\mathbb{Z}$, let $\overline{g}$ be the image of $g$ under the unital ring homomorphism from $\mathbb{Z}$ to $\F_p$. Let $\langle \mathbb{U}\rangle_\mathbb{V}$ be the $\F$-linear subspace of an $\F$-linear $\mathbb{V}$ spanned by $\mathbb{U}$.
All $\F$-algebras in this paper shall be finite-dimensional associative unital algebras. All modules in this paper shall be left modules of some $\F$-algebras.
\subsection{Schemes}
Let $\mathbb{X}$, $\mathbb{E}$ be nonempty finite sets. Let $\{R_a: a\!\in\!\mathbb{E}\}$ be a partition of
$\mathbb{X}\times\mathbb{X}$. Call $\mathfrak{S}\!=\!(\mathbb{X}, \{R_a: a\!\in\!\mathbb{E}\})$ a {\em scheme} if the following conditions hold together:
\begin{enumerate}[(S1)]
\item There is a unique $0_{\mathfrak{S}}\in\mathbb{E}$ such that $R_{0_{\mathfrak{S}}}$ equals the diagonal set $\{(a,a): a\!\in\!\mathbb{X}\}$.
\item For any $g\!\in\!\mathbb{E}$, there is a unique $g^*\!\in\!\mathbb{E}$ such that $R_{g^*}$ equals $\{(a, b): (b,a)\in R_g\}$.
\item For any $g, h, i\in\mathbb{E}$ and $(x,y), (\widetilde{x},\widetilde{y})\!\in\! R_i$, there is a constant $p_{g,h}^i\in\mathbb{N}_0$ such that
$$p_{g,h}^i=|\{a: (x,a)\in R_g, (a,y)\in R_h\}|=|\{a: (\widetilde{x},a)\in R_g, (a,\widetilde{y})\in R_h\}|.$$
\end{enumerate}
From now on, let $\mathfrak{S}\!=\!(\mathbb{X}, \{R_a\!:\! a\!\in\!\mathbb{E}\})$ be a fixed scheme. Call $\mathfrak{S}$ a {\em symmetric scheme} if $g^*=g$ for any $g\in\mathbb{E}$. Call $\mathfrak{S}$ a {\em commutative scheme} if $p_{g,h}^i=p_{h,g}^i$ for any $g, h, i\!\in\!\mathbb{E}$. So the symmetric schemes are commutative schemes by (S3). If $\mathbb{U}, \mathbb{V}\!\!\subseteq\!\!\{R_a: a\!\in\!\mathbb{E}\}$, let $\mathbb{U}\mathbb{V}\!=\!\{R_a: \exists\ R_b\in\mathbb{U}, \exists\  R_c\in\mathbb{V}, p_{b,c}^a\!\neq\!0\}$. Let $R_g R_h\!=\!\{R_g\}\{R_h\}$ for any $g, h\in\mathbb{E}$.

Let $x R_g\!=\!\{a: (x,a)\!\in\! R_g\}$ and $k_g\!=\!p_{g,g^*}^{0_{\mathfrak{S}}}$ for any $x\!\in\!\mathbb{X}$ and $g\!\in\!\mathbb{E}$. If $x, y\!\in\!\mathbb{X}$ and $g\!\in\!\mathbb{E}$, (S2) and (S3) give $k_g\!=\!|x R_g|\!=\!|y R_g|\!\!>\!\!0$. If $g\!\in\!\mathbb{N}_0$ and $g\!\nmid\! k_h$ for any $h\in\mathbb{E}$, call $\mathfrak{S}$ a {\em $g'$-valenced scheme}. If $x, y\!\in\! \mathbb{X}$, $g, h, i\!\in\! \mathbb{E}$, and $y\!\in\! x R_i$, (S2) and (S3) give $p_{g,h}^i\!\!=\!\!|x R_g\cap y R_{h^*}|$. Call $\mathfrak{S}$ a {\em triply regular scheme} if, for any $x, y, z\!\in\!\mathbb{X}$, $g, h, i, j, k, \ell\!\in\!\mathbb{E}$, $y\!\in\! x R_j$, $z\!\in\! x R_k\cap y R_\ell$, $|x R_g\cap y R_h\cap z R_i|$ only depends on $g, h, i, j, k, \ell$ and is independent of the choices of $x, y, z$ that satisfy $y\in x R_j$ and $z\in x R_k\cap y R_\ell$.

Let $g\in\mathbb{N}$ and $\mathbb{X}_h, \mathbb{E}_h$ be nonempty finite sets for any $h\in[1, g]$. Let $\{R_a^h: a\in\mathbb{E}_h\}$ be a partition of $\mathbb{X}_h\times\mathbb{X}_h$ for any $h\!\in\![1, g]$. Let $\mathfrak{S}_h=(\mathbb{X}_h, \{R_a^h: a\in\mathbb{E}_h\})$ be a scheme for any $h\!\in\![1, g]$. If $\mathbf{x}, \mathbf{y}\!\in\!\prod_{h=1}^g\mathbb{X}_h$ and $\mathbf{i}\!\in\!\prod_{h=1}^g\mathbb{E}_h$, let $\mathbf{x}\!=_\mathbf{i}\!\mathbf{y}$ if $(\mathbf{x}_h, \mathbf{y}_h)\!\in\! R_{\mathbf{i}_h}^h$ for any $h\in[1, g]$. Let $R_\mathbf{i}\!=\! \{(\mathbf{a}, \mathbf{b}): \mathbf{a}, \mathbf{b}\in\prod_{h=1}^g\mathbb{X}_h, \mathbf{a}=_\mathbf{i}\mathbf{b}\}$ for any $\mathbf{i}\in\prod_{h=1}^g\mathbb{E}_h$. So $(\prod_{h=1}^g\mathbb{X}_h, \{R_\mathbf{a}: \mathbf{a}\!\in\! \prod_{h=1}^g\mathbb{E}_h\})$ is a scheme (see \cite{B} or \cite{Z}). Call $\mathfrak{S}$ a {\em direct product} of $\mathfrak{S}_1, \mathfrak{S}_2, \ldots, \mathfrak{S}_g$ if
$\mathbb{X}=\prod_{h=1}^g\mathbb{X}_h$, $\mathbb{E}\!=\!\prod_{h=1}^g\mathbb{E}_h$, and $\{R_a: a\!\in\!\mathbb{E}\}\!=\!\{R_\mathbf{a}: \mathbf{a}\!\in\!\prod_{h=1}^g\mathbb{E}_h\}$.
If $\mathfrak{S}$ is a direct product of $\mathfrak{S}_1, \mathfrak{S}_2, \ldots, \mathfrak{S}_g$, it is clear that $0_\mathfrak{S}=(0_{\mathfrak{S}_1}, 0_{\mathfrak{S}_2},\ldots, 0_{\mathfrak{S}_g})$.
If $\mathfrak{S}$ is a direct product of $\mathfrak{S}_1, \mathfrak{S}_2, \ldots, \mathfrak{S}_g$, notice that $\mathfrak{S}$ is a symmetric scheme if and only if $\mathfrak{S}_h$ is a symmetric scheme for any $h\in[1, g]$. The next lemma is required.
\begin{lem}\label{L;Lemma2.1}\cite[Theorem 2.6.1 (iv)]{Z}
Assume that $g\in\mathbb{N}$. Assume that $\mathfrak{S}$ is a direct product of the scheme sequence $\mathfrak{S}_1, \mathfrak{S}_2, \ldots, \mathfrak{S}_g$. Then $p_{\mathbf{i},\mathbf{j}}^\mathbf{k}=\prod_{h=1}^g p_{\mathbf{i}_h,\mathbf{j}_h}^{\mathbf{k}_h}$ if $\mathbf{i}, \mathbf{j}, \mathbf{k}\in\mathbb{E}$.
\end{lem}

Let $g, h\in\mathbb{N}$ and $\mathbb{U}$ be a set satisfying $|\mathbb{U}|=gh$. Let $\mathbb{G}$ be a partition of $\mathbb{U}$.
Call $\mathbb{G}$ a {\em group divisible $(g, h)$-partition} of $\mathbb{U}$ if $\mathbb{G}$ has $g$-many members
and each member in $\mathbb{G}$ has cardinality $h$. Call each member in a group divisible $(g, h)$-partition of $\mathbb{U}$ a {\em $\mathbb{U}$-group}. Hence a $\mathbb{U}$-group does not necessarily have a finite group structure. Let $x, y\!\in\!\mathbb{U}$ and $\mathbb{G}$ be a group divisible $(g, h)$-partition of $\mathbb{U}$. Let $x\!\sim_\mathbb{G}\!y$ if $x\!\neq\! y$ and $x, y$ belong to a single $\mathbb{U}$-group in $\mathbb{G}$. Let $x\!\not\sim_\mathbb{G}\!y$ if $x, y$ belong to distinct $\mathbb{U}$-groups in $\mathbb{G}$.

Let $g, h\in\mathbb{N}\setminus\{1\}$ and $\mathbb{U}$ be a set satisfying $|\mathbb{U}|=gh$. Let $\mathbb{G}$ be a group divisible $(g, h)$-partition of $\mathbb{U}$. So $\{(a, b): a, b\in\mathbb{U}, a\sim_\mathbb{G} b\}$ and $\{(a, b): a, b\in\mathbb{U}, a\not\sim_\mathbb{G} b\}$ are nonempty. Let $R(\mathbb{G})_0, R(\mathbb{G})_1, R(\mathbb{G})_2$ be $\{(a,a): a\!\in\!\mathbb{U}\}$, $\{(a, b): a, b\!\in\!\mathbb{U}, a\!\sim_\mathbb{G}\! b\}$, $\{(a, b): a, b\!\in\!\mathbb{U}, a\not\sim_\mathbb{G}b\}$, respectively. Hence $(\mathbb{U}, \{R(\mathbb{G})_a: a\!\in\![0,2]\})$ is a scheme (see \cite{B,KM}). Call $\mathfrak{S}$ a {\em group divisible scheme} of parameter $(g, h)$ if $\mathbb{X}=\mathbb{U}$, $\mathbb{E}=[0,2]$, and $\{R_a: a\!\in\!\mathbb{E}\}=\{R(\mathbb{G})_a: a\!\in\![0,2]\}$. If $\mathfrak{S}$ is a group divisible scheme of parameter $(g, h)$, then $\mathfrak{S}$ is a symmetric scheme. If $\mathfrak{S}$ is a group divisible scheme of parameter $(g, h)$, recall that $\mathfrak{S}$ is a triply regular scheme (see \cite{MA}). The next lemma is required.
\begin{lem}\label{L;Lemma2.2}\cite[Example 1.1]{B}
Assume that $g, h\in\mathbb{N}\setminus\{1\}$. Assume that $\mathfrak{S}$ is a group divisible scheme of parameter $(g, h)$ and $\mathbb{U}$ denotes the set containing exactly $(0, 0, 0), (0, 1, 1), (0, 2, 2), (1, 0, 1), (1, 1, 0), (1, 1, 1), (1, 2, 2), (2, 0, 2), (2,1,2)$, $(2, 2, 0)$, $(2,2,1)$, and $(2,2,2)$. Assume that $i, j, k\in [0,2]$. Then $p_{i,j}^k\neq0$ only if $(i, j, k)\in\mathbb{U}$.
Moreover, $p_{0,0}^0=p_{0,1}^1=p_{0,2}^2=p_{1,0}^1=p_{2,0}^2=1$, $p_{1,1}^0=p_{1,2}^2=p_{2,1}^2=h-1$, $p_{1,1}^1=h-2$, $p_{2,2}^0\!=\!p_{2,2}^1\!\!=\!\!(g-1)h$, $p_{2,2}^2\!=\!(g-2)h$. In particular, $\{(a, b, c): a, b, c\!\in\![0,2], p_{a,b}^c\!\neq\! 0\}$ is
\[\begin{cases}
\mathbb{U}\setminus\{(1,1,1), (2,2,2)\}, &\text{if $g=h=2$},\\
\mathbb{U}\setminus\{(1,1,1)\}, &\text{if $g>h=2$},\\
\mathbb{U}\setminus\{(2,2,2)\}, &\text{if $h>g=2$},\\
\mathbb{U}, &\text{if $\min\{g, h\}>2$}.
\end{cases}\]
\end{lem}
Fix $n\in\mathbb{N}$. Let $\ell_1, \ell_2, \ldots, \ell_n, m_1, m_2,\ldots, m_n$ be a fixed sequence in $\mathbb{N}\setminus\{1\}$.
Use $n_1, n_2, n_3, n_4$ to denote $|\{a: a\!\in\! [1,n], \ell_a\!=\!m_a\!=\!2\}|$, $|\{a: a\!\in\! [1,n], \ell_a\!>\!m_a\!=\!2\}|$, $|\{a: a\in [1,n], m_a\!>\!\ell_a\!=\!2\}|$, $|\{a: a\in [1,n], \min\{\ell_a, m_a\}\!>\!2\}|$, respectively. If $\mathbf{i}$ is a $n$-tuple whose entries are in $[0,2]$, let $\mathbbm{i}_j\!=\!\{a: a\in [1,n], \mathbf{i}_a=j\}$ for any $j\in[0,2]$. For any $\mathbb{V}\subseteq[1,n]$, let $\mathbb{V}^\bullet=\{a: a\in\mathbb{V}, \ell_a>2\}$ and $\mathbb{V}^\circ=\{a: a\in\mathbb{V}, m_a>2\}$. For any sets $\mathbb{V}, \mathbb{W}$, let $\mathbb{V}\triangle\mathbb{W}\!\!=\!\!(\mathbb{V}\setminus\mathbb{W})\cup(\mathbb{W}\setminus\mathbb{V})$.
For any sets $\mathbb{U}_1, \mathbb{U}^1, \mathbb{V}_1, \mathbb{V}^1, \mathbb{W}_1, \mathbb{W}^1$, let
$|(\mathbb{U}_1, \mathbb{V}_1, \mathbb{W}_1)|=|\mathbb{U}_1|+|\mathbb{V}_1|+|\mathbb{W}_1|$ and $(\mathbb{U}_1, \mathbb{V}_1,\mathbb{W}_1)\preceq(\mathbb{U}^1, \mathbb{V}^1, \mathbb{W}^1)$ if $\mathbb{U}_1\subseteq\mathbb{U}^1$,
$\mathbb{V}_1\!\subseteq\!\mathbb{V}^1$, and $\mathbb{W}_1\subseteq \mathbb{W}^1$.
Notice that $\preceq$ is a partial order on the set of all set triples.

Fix a set sequence $\mathbb{U}_1, \mathbb{U}_2, \ldots, \mathbb{U}_n$ that satisfies $|\mathbb{U}_i|=\ell_im_i$ for any $i\in [1,n]$.
Fix a group divisible $(\ell_i, m_i)$-partition $\mathbb{G}_i$ of $\mathbb{U}_i$ for any $i\in[1,n]$. The group divisible scheme $(\mathbb{U}_i, \{R(\mathbb{G}_i)_a: a\in [0,2]\})$ of parameter $(\ell_i, m_i)$ is denoted by $\mathfrak{GD}(\ell_i, m_i)$ for any $i\in [1,n]$. If $\mathfrak{S}$ is a direct product of $\mathfrak{GD}(\ell_1, m_1), \mathfrak{GD}(\ell_2, m_2),\ldots, \mathfrak{GD}(\ell_n, m_n)$, each element in $\mathbb{E}$ is a $n$-tuple whose entries are in $[0,2]$. If $\mathfrak{S}$ is a direct product of $\mathfrak{GD}(\ell_1, m_1), \mathfrak{GD}(\ell_2, m_2),\ldots, \mathfrak{GD}(\ell_n, m_n)$, notice that $\mathfrak{S}$ is also a symmetric scheme as $\mathfrak{GD}(\ell_i, m_i)$ is a symmetric scheme for any $i\in [1,n]$. The next lemma is required.
\begin{lem}\label{L;Lemma2.3}\cite[Theorem 10]{W}
Assume that $\mathfrak{S}$ is a direct product of the triply regular schemes. Then $\mathfrak{S}$ is a triply regular scheme. In particular, if $\mathfrak{S}$ is a direct product of $\mathfrak{GD}(\ell_1, m_1), \mathfrak{GD}(\ell_2, m_2),\ldots, \mathfrak{GD}(\ell_n, m_n)$, then $\mathfrak{S}$ is also a triply regular scheme.
\end{lem}
\subsection{Algebras}
Let $\mathbb{U}$ be an $\F$-linear subspace of an $\F$-linear space $\mathbb{V}$. Let $\mathbb{V}/\mathbb{U}$ be the quotient $\F$-linear space of $\mathbb{V}$ with respect to $\mathbb{U}$. Let $\mathbb{A}$ be an $\F$-algebra with the zero element $0_\mathbb{A}$ and the identity element $1_\mathbb{A}$. Let $\mathbb{B}$ be an $\F$-basis of $\mathbb{A}$ and $e\in\mathbb{A}$. Then $e$ is uniquely written as an $\F$-linear combination of the elements in $\mathbb{B}$. If $f\in \mathbb{B}$, let $c_\mathbb{B}(e, f)$ be the coefficient of $f$ in this $\F$-linear combination that represents $e$. Let $\mathrm{Supp}_{\mathbb{B}}(e)=\{a: a\in \mathbb{B}, c_{\mathbb{B}}(e, a)\in\F^\times\}$. Hence $\mathrm{Supp}_{\mathbb{B}}(e)\!=\!\varnothing$ if and only if $e=0_\mathbb{A}$.
The {\em{center}} of $\mathbb{A}$ is defined to be $\{a: a\!\in\! \mathbb{A}, ab=ba\ \forall\ b\in\mathbb{A}\}$. Denote the center of $\mathbb{A}$ by $\mathrm{Z}(\mathbb{A})$. Notice that $\mathrm{Z}(\mathbb{A})$ is an $\F$-subalgebra of $\mathbb{A}$ with the identity element $1_\mathbb{A}$.

Let $\mathbb{I}$ be a two-sided ideal of $\mathbb{A}$. Let $\mathbb{A}/\mathbb{I}$ be the quotient $\F$-algebra of $\mathbb{A}$ with respect to $\mathbb{I}$. Call $\mathbb{I}$ a {\em minimal two-sided ideal} of $\mathbb{A}$ if no nonzero two-sided ideal of $\mathbb{A}$ is contained in $\mathbb{I}$ as a proper subset. Call $\mathbb{A}$ a {\em semisimple $\F$-algebra} if $\mathbb{A}$ is a direct sum of its minimal two-sided ideals. Call $\mathbb{I}$ a {\em nilpotent two-sided ideal} of $\mathbb{A}$ if there is $h\in\mathbb{N}$ such that the product of any $h$-many elements in $\mathbb{I}$ is $0_\mathbb{A}$. If $\mathbb{I}$ is a nilpotent two-sided ideal of $\mathbb{A}$ and $h\in\mathbb{N}$, call $h$ the {\em nilpotent index} of $\mathbb{I}$ if $h$ is the smallest number such that the product of any $h$-many elements in $\mathbb{I}$ is $0_\mathbb{A}$. The {\em Jacobson radical} of $\mathbb{A}$ is defined to be the sum of all nilpotent two-sided ideals of $\mathbb{A}$. Use $\mathrm{Rad}(\mathbb{A})$ to denote the Jacobson radical of $\mathbb{A}$. Notice that $\mathrm{Rad}(\mathbb{A})$ is a nilpotent two-sided ideal of $\mathbb{A}$. Call $e$ an {\em idempotent} of $\mathbb{A}$ if $e^2=e$.
Then $\{eae: a\in\mathbb{A}\}$ is denoted by $e\mathbb{I}e$. Assume further that $e$ is also an idempotent of $\mathbb{A}$. Notice that
$e\mathbb{I}e$ is an $\F$-subalgebra of $\mathbb{A}$ with the identity element $e$. The next lemmas are required.
\begin{lem}\label{L;Lemma2.4}\cite[Theorem 3.1.1, Proposition 3.2.4]{DK}
The $\F$-algebra $\mathbb{A}$ is a semisimple $\F$-algebra if and only if $\mathrm{Rad}(\mathbb{A})=\{0_\mathbb{A}\}$. Moreover, assume that $e$ is an idempotent of the $\F$-algebra $\mathbb{A}$. Then $\mathrm{Rad}(e\mathbb{A}e)=e\mathrm{Rad}(\mathbb{A})e\subseteq\mathrm{Rad}(\mathbb{A})$. In particular, if $\mathbb{A}$ is a semisimple $\F$-algebra, then both $\mathrm{Z}(\mathbb{A})$ and $e\mathbb{A}e$ are semisimple
$\mathbb{F}$-subalgebras of $\mathbb{A}$.
\end{lem}
\begin{lem}\label{L;Lemma2.5}\cite[Corollary 3.1.14]{DK}
Assume that $\mathbb{I}$ is a nilpotent two-sided ideal of the $\F$-algebra $\mathbb{A}$. Then $\mathrm{Rad}(\mathbb{A}/\mathbb{I})=\mathrm{Rad}(\mathbb{A})/\mathbb{I}$ and $\mathbb{A}/\mathrm{Rad}(\mathbb{A})$
is a semisimple $\F$-algebra.
\end{lem}
Let $\mathbb{V}$ be an $\mathbb{A}$-module. Call $\mathbb{V}$ an {\em irreducible $\mathbb{A}$-module} if $\mathbb{V}$ has no nonzero proper $\mathbb{A}$-submodule. Let $\mathbb{K}$ be an extension field of $\F$. Notice that $\mathbb{K}$ itself is an $\F$-algebra via the addition and the multiplication of the elements in $\mathbb{K}$. In particular, $\mathbb{K}$ itself is an $\F$-linear space via the addition and the multiplication of the elements in $\mathbb{K}$. Let $\mathbb{A}_\mathbb{K}$ be the tensor product $\mathbb{K}\otimes_\F\mathbb{A}$ of $\F$-algebras. Let $\mathbb{V}_\mathbb{K}$ be the tensor product $\mathbb{K}\otimes_\F\mathbb{V}$ of $\F$-linear spaces. So $\mathbb{V}_\mathbb{K}$ is an $\mathbb{A}_\mathbb{K}$-module via the diagonal action of $\mathbb{A}_\mathbb{K}$ on $\mathbb{V}_\mathbb{K}$. Call $\mathbb{V}$ an {\em absolutely irreducible $\mathbb{A}$-module} if the $\mathbb{A}_\mathbb{L}$-module $\mathbb{V}_{\mathbb{L}}$ is an irreducible $\mathbb{A}_{\mathbb{L}}$-module for any an extension field $\mathbb{L}$ of $\F$. Then $\mathbb{V}$ is an absolutely irreducible $\mathbb{A}$-module only if $\mathbb{V}$ is an irreducible $\mathbb{A}$-module. However, an irreducible $\mathbb{A}$-module is not necessary to be an absolutely irreducible $\mathbb{A}$-module. The next lemma is required.
\begin{lem}\label{L;Lemma2.6}\cite[Corollaries 3.1.7, 1.10.4]{DK,K}
Assume that $g\in\mathbb{N}$. Assume that $\mathbb{A}$ is an $\mathbb{F}$-algebra and $\mathbb{A}/\mathrm{Rad}(\mathbb{A})$ is isomorphic to a direct sum of $g$-many full matrix $\F$-algebras of square $\F$-matrices. Then the cardinality of the set containing exactly all pairwise nonisomorphic irreducible $\mathbb{A}$-modules is $g$. Moreover, an $\mathbb{A}$-module is an irreducible $\mathbb{A}$-module if and only if it is an absolutely irreducible $\mathbb{A}$-module.
\end{lem}
\subsection{Terwilliger $\F$-algebras of schemes}
Let $g\in\mathbb{N}$ and $\mathrm{M}_g(\F)$ be the full matrix $\F$-algebra of $(g\times g)$-matrices whose entries are in $\F$. Let $h\in\mathbb{N}$ and $g\mathrm{M}_h(\F)$ be the direct sum of $g$-many copies of $\mathrm{M}_h(\F)$. Let $i, j\!\in\!\mathbb{N}$. The Wedderburn-Artin Theorem implies that $g\mathrm{M}_h(\F)\!\cong\!i\mathrm{M}_j(\F)$ if and only if $g=i$ and $h=j$. Let $\mathbb{U}$ be a nonempty finite set. Let $\mathrm{M}_{\mathbb{U}}(\F)$ be the full matrix $\F$-algebra of square $\F$-matrices whose rows and columns are labeled by the elements in $\mathbb{U}$. So $\mathrm{M}_{\mathbb{U}}(\F)\cong\mathrm{M}_{|\mathbb{U}|}(\F)$ as $\F$-algebras. Denote the identity matrix and the zero matrix in $\mathrm{M}_\mathbb{X}(\F)$ by $I$ and $O$, respectively. If $x, y\in\mathbb{X}$, use $E_{x, y}$ to denote the $\{\overline{0},\overline{1}\}$-matrix in $\mathrm{M}_{\mathbb{X}}(\F)$ whose unique nonzero entry is the $(x, y)$-entry. The transpose of a matrix $M$ in $\mathrm{M}_{\mathbb{X}}(\F)$ is denoted by $M^T$.

Let $g, h\!\in\!\mathbb{E}$. Use $A_g$ to denote the adjacency $\F$-matrix with respect to $R_g$. Let $x\!\in\!\!\mathbb{X}$. Use $E_g^*(x)$ to denote the dual $\F$-idempotent with respect to $x, R_g$. Recall that
\begin{align}\label{Eq;1}
A_g=\sum_{(y, z)\in R_g}E_{y, z}\ \text{and}\ E_g^*(x)=\sum_{y\in xR_g}E_{y, y}.
\end{align}
Notice that $A_g^T=A_g$ if $\mathfrak{S}$ is a symmetric scheme. More generally, recall the equations
\begin{align}\label{Eq;2}
A_g^T=A_{g^*}\ \text{and}\ E_g^*(x)^T=E_g^*(x),
\end{align}
\begin{align}\label{Eq;3}
E_g^*(x)E_h^*(x)=\delta_{g,h}E_g^*(x),
\end{align}
\begin{align}\label{Eq;4}
A_{0_\mathfrak{S}}=I=\sum_{i\in\mathbb{E}}E_i^*(x).
\end{align}
The Terwilliger $\F$-algebra $\mathbb{T}(x)$ of $\mathfrak{S}$ with respect to $x$ is the $\F$-subalgebra of $\mathrm{M}_\mathbb{X}(\F)$ generated by $\{A_a: a\in\mathbb{E}\}\cup\{E_a^*(x): a\in \mathbb{E}\}$. Hence Equation \eqref{Eq;4} shows that the identity element of
$\mathbb{T}(x)$ is $I$. By Equation \eqref{Eq;2}, $M\!\in\!\mathbb{T}(x)$ if and only if $M^T\in\mathbb{T}(x)$.

Let $i\!\in\!\mathbb{E}$. Then $E_g^*(x)A_hE_i^*(x)\!\!\neq\!\! O$ if and only if $p_{g^*,i}^h\!\!\neq\!\! 0$ (see \cite[Lemma 3.2]{Han}). Then $\mathbb{T}(x)$ has an $\F$-linearly independent subset $\{E_a^*(x)A_bE_c^*(x)\!: a, b,c\in \mathbb{E}, p_{a^*,c}^b\!\neq\! 0\}$ by the definition of $\mathbb{T}(x)$ and Equation \eqref{Eq;3}. In general, the algebraic structures of $\mathbb{T}(x)$ and $\mathrm{Rad}(\mathbb{T}(x))$ depend on the choices of $\F$ and $x$ (see \cite[5.1]{Han}). See \cite{CX,Han,J1,J2,J3} for some recent progress on both $\mathbb{T}(x)$ and $\mathrm{Rad}(\mathbb{T}(x))$. The next lemmas are required.
\begin{lem}\label{L;Lemma2.7}\cite[Theorem 3.4]{Han}
Assume that $x\in \mathbb{X}$. Then $\mathbb{T}(x)$ is a semisimple $\F$-algebra only if $\mathfrak{S}$ is a $p'$-valenced scheme.
\end{lem}
\begin{lem}\label{L;Lemma2.8}\cite[Lemma 4]{MA}
Assume that $x\in \mathbb{X}$ and $\mathfrak{S}$ is a triply regular scheme. Then $\mathbb{T}(x)$ has an $\F$-basis
$\{E_a^*(x)A_bE_c^*(x): a, b, c\in \mathbb{E}, p_{a^*,c}^b\neq0\}$. In particular, the $\F$-dimension of $\mathbb{T}(x)$ only depends on $\mathfrak{S}$ and is independent of the choice of $x$.
\end{lem}
\begin{lem}\label{L;Lemma2.9}\cite[Lemma 2.10]{J3}
Assume that $x\in \mathbb{X}$ and $\mathfrak{S}$ is a symmetric scheme. If $\mathfrak{S}$ is also a triply regular scheme, then the $\F$-subalgebra $E_g^*(x)\mathbb{T}(x)E_g^*(x)$ of $\mathbb{T}(x)$ is a commutative $\F$-algebra for any $g\in \mathbb{E}$.
\end{lem}
We end this section by simplifying the notation. From now on, assume that $\mathfrak{S}$ is a direct product of $\mathfrak{GD}(\ell_1, m_1), \mathfrak{GD}(\ell_2, m_2),\ldots, \mathfrak{GD}(\ell_n, m_n)$. Recall that all entries of each $n$-tuple in $\mathbb{E}$ are arbitrarily chosen from $[0,2]$. This shows that $|\mathbb{E}|=3^n$. We shall quote the fact that $\mathfrak{S}$ is a symmetric scheme without citation. From now on, fix $\mathbf{x}\in\mathbb{X}$. For convenience, we abbreviate $\mathbb{T}=\mathbb{T}(\mathbf{x})$ and $E_\mathbf{g}^*=E_\mathbf{g}^*(\mathbf{x})$ for any $\mathbf{g}\in \mathbb{E}$.
\section{Algebraic structure of $\mathbb{T}$: $\F$-Basis}
In this section, we present two $\F$-bases for $\mathbb{T}$. We first describe the $n$-tuples $\mathbf{g}, \mathbf{h}, \mathbf{i}$ in $\mathbb{E}$ satisfying the inequality $p_{\mathbf{g},\mathbf{h}}^\mathbf{i}\neq 0$ and display an $\F$-basis of $\mathbb{T}$ as an application.
\begin{lem}\label{L;Lemma3.1}
Assume that $\mathbf{g}, \mathbf{h}\in\mathbb{E}$. Then $\mathbbm{g}_0, \mathbbm{g}_1, \mathbbm{g}_2$ are pairwise disjoint subsets of $[1,n]$ and $\mathbbm{g}_0\cup\mathbbm{g}_1\cup\mathbbm{g}_2=[1,n]$. Moreover, $\mathbf{g}=\mathbf{h}$ if and only if $\mathbbm{g}_1=\mathbbm{h}_1$ and $\mathbbm{g}_2=\mathbbm{h}_2$.
\end{lem}
\begin{proof}
The desired lemma follows as $\mathbbm{g}_i=\{a: a\!\in\![1,n], \mathbf{g}_a=i\}$ for any $i\in[0,2]$.
\end{proof}
\begin{lem}\label{L;Lemma3.2}
Assume that $\mathbf{g}, \mathbf{h}, \mathbf{i}\in\mathbb{E}$. Then $p_{\mathbf{g},\mathbf{h}}^\mathbf{i}\neq0$ if and only if
\begin{align}\label{Eq;5}
(\mathbbm{g}_0\cap\mathbbm{h}_1)\cup(\mathbbm{g}_1\cap\mathbbm{h}_0)\subseteq \mathbbm{i}_1\subseteq(\mathbbm{g}_0\cap\mathbbm{h}_1)\cup(\mathbbm{g}_1\cap\mathbbm{h}_0)\cup(\mathbbm{g}_1\cap\mathbbm{h}_1)^\circ
\cup(\mathbbm{g}_2\cap\mathbbm{h}_2)
\end{align}
\begin{align}\label{Eq;6}
\text{and}\ \mathbbm{g}_2\triangle\mathbbm{h}_2\subseteq\mathbbm{i}_2\subseteq(\mathbbm{g}_2\triangle\mathbbm{h}_2)
\cup(\mathbbm{g}_2\cap\mathbbm{h}_2)^\bullet.
\end{align}
\end{lem}
\begin{proof}
The desired lemma follows from combining Lemmas \ref{L;Lemma2.1}, \ref{L;Lemma2.2}, and \ref{L;Lemma3.1}.
\end{proof}
We next list a notation and apply Lemma \ref{L;Lemma2.8} to $\mathfrak{S}$ to produce an $\F$-basis of $\mathbb{T}$.
\begin{nota}\label{N;Notation3.3}
Use $\mathbb{P}$ to denote the set of all triples $(\mathbf{g}, \mathbf{h}, \mathbf{i})$ in $\mathbb{E}\times \mathbb{E}\times \mathbb{E}$ that satisfy Equations (\ref{Eq;5}) and (\ref{Eq;6}). Define $\mathbb{P}_{\mathbf{g}, \mathbf{h}}=\{\mathbf{a}: (\mathbf{g}, \mathbf{h}, \mathbf{a})\in\mathbb{P}\}$ for any $\mathbf{g}, \mathbf{h}, \mathbf{i}\in \mathbb{E}$. For any $\mathbf{g}, \mathbf{h}, \mathbf{i}\in \mathbb{E}$, notice that Lemma \ref{L;Lemma3.2} implies that $p_{\mathbf{g},\mathbf{h}}^\mathbf{i}\neq 0$ if and only if $\mathbf{i}\in\mathbb{P}_{\mathbf{g}, \mathbf{h}}$.
\end{nota}
\begin{lem}\label{L;Lemma3.4}
$\mathbb{T}$ has an $\F$-basis $\{E_\mathbf{a}^*A_\mathbf{b}E_\mathbf{c}^*\!: (\mathbf{a}, \mathbf{c}, \mathbf{b})\!\in\!\mathbb{P}\}$ whose cardinality equals $|\mathbb{P}|$.
\end{lem}
\begin{proof}
The desired lemma follows from combining Lemmas \ref{L;Lemma2.3}, \ref{L;Lemma2.8}, and \ref{L;Lemma3.2}.
\end{proof}
The $\F$-basis of $\mathbb{T}$ in Lemma \ref{L;Lemma3.4} motivates us to introduce two additional lemmas.
\begin{lem}\label{L;Lemma3.5}
Assume that $\mathbf{g}, \mathbf{h}, \mathbf{i}, \mathbf{j}, \mathbf{k}\in \mathbb{E}$. Then
\begin{align*}
E_\mathbf{g}^*A_\mathbf{h}E_\mathbf{i}^*A_\mathbf{j}E_\mathbf{k}^*=\sum_{\mathbf{l}\in\mathbb{P}_{\mathbf{g},\mathbf{k}}}c_{\mathbf{g},\mathbf{h},\mathbf{i},\mathbf{j},\mathbf{k},\mathbf{l}}E_\mathbf{g}^*A_\mathbf{l} E_\mathbf{k}^*,
\end{align*}
where $c_{\mathbf{g},\mathbf{h},\mathbf{i},\mathbf{j},\mathbf{k},\mathbf{l}}=\overline{|\mathbf{y}R_\mathbf{h}\cap\mathbf{x}R_\mathbf{i}\cap\mathbf{z}R_\mathbf{j}|}$ for any $\mathbf{y}\in\mathbf{x}R_\mathbf{g}$ and $\mathbf{z}\in\mathbf{x}R_\mathbf{k}\cap\mathbf{y}R_\mathbf{l}$. Moreover, the coefficient $c_{\mathbf{g},\mathbf{h},\mathbf{i},\mathbf{j},\mathbf{k},\mathbf{l}}$ only depends on $\mathbf{g}, \mathbf{h}, \mathbf{i}, \mathbf{j}, \mathbf{k}, \mathbf{l}$ and is independent of the choices of $\mathbf{y}$, $\mathbf{z}$. In particular, if $c_{\mathbf{g},\mathbf{h},\mathbf{i},\mathbf{j},\mathbf{k},\mathbf{l}}\in\F^{\times}$ for some $\mathbf{l}\in \mathbb{P}_{\mathbf{g},\mathbf{k}}$, then $R_\mathbf{l}\in R_\mathbf{g}R_\mathbf{k}\cap R_\mathbf{h}R_\mathbf{j}$.
\end{lem}
\begin{proof}
The desired lemma is from combining Lemma \ref{L;Lemma3.4}, Equations \eqref{Eq;3}, \eqref{Eq;1}.
\end{proof}
\begin{lem}\label{L;Lemma3.6}
The $\F$-dimension of $\mathbb{T}$ is equal to $2^{n_1+2n_4}3^{n_4}5^{n_1}11^{n_2+n_3}$.
\end{lem}
\begin{proof}
Recall that $n_1\!=\!|\{a: a\!\in\! [1,n], \ell_a\!=\!m_a\!=\!2\}|$, $n_2=|\{a: a\!\in\! [1,n], \ell_a\!\!>\!m_a\!=\!2\}|$,
$n_3=|\{a: a\in[1,n], m_a>\ell_a=2\}|$, and $n_4=|\{a: a\!\in\! [1,n], \min\{\ell_a, m_a\}\!>\!2\}|$. Notice that the sum of $n_1$, $n_2$, $n_3$, and $n_4$ is equal to $n$. If $g\in [1, n]$, Lemmas \ref{L;Lemma2.8} and \ref{L;Lemma2.2} imply that
the $\F$-dimension of a Terwilliger $\F$-algebra of $\mathfrak{GD}(\ell_g, m_g)$ equals
\begin{align}\label{Eq;7}
\begin{cases}
10, &\ \text{if $\ell_g=m_g=2$,}\\
11, &\ \text{if $\ell_g>m_g=2$ or $m_g>\ell_g=2$,}\\
12, &\ \text{if $\min\{\ell_g, m_g\}>2$.}
\end{cases}
\end{align}
Let $h_g$ be the $\F$-dimension of a Terwilliger $\F$-algebra of $\mathfrak{GD}(\ell_g, m_g)$ for any $g\in [1, n]$. By combining Lemmas \ref{L;Lemma3.4}, \ref{L;Lemma2.1}, and \ref{L;Lemma2.8}, notice that the $\F$-dimension of $\mathbb{T}$ equals $\prod_{g=1}^n h_g$. The desired lemma thus follows from an application of Equation \eqref{Eq;7}.
\end{proof}
We next use the $\F$-basis in Lemma \ref{L;Lemma3.4} to give another $\F$-basis of $\mathbb{T}$ that shall be utilized in the next sections. The following notations and four lemmas are necessary.
\begin{nota}\label{N;Notation3.7}
Assume that $\mathbf{g}, \mathbf{h}\in \mathbb{E}$. Then $\mathbb{U}_{\mathbf{g}, \mathbf{h}}$ is defined to be the triple set $\{(\mathbbm{a}_{(1)}, \mathbbm{a}_{(2)}, \mathbbm{a}_{(3)}):\mathbbm{a}_{(1)}\!\subseteq\! (\mathbbm{g}_1\cap\mathbbm{h}_1)^\circ,\mathbbm{a}_{(2)}\subseteq (\mathbbm{g}_2\cap\mathbbm{h}_2)^\bullet, \mathbbm{a}_{(2)}\!\subseteq\!\mathbbm{a}_{(3)}\!\subseteq\!\mathbbm{g}_2\cap\mathbbm{h}_2\}$. Let $\mathbb{U}_{\mathbf{g}, \mathbf{h}, \mathfrak{i}}$ be $\{R_\mathbf{a}: \mathbf{a}\in \mathbb{P}_{\mathbf{g}, \mathbf{h}}, \mathbbm{a}_1\cap(\mathbbm{g}_1\cap\mathbbm{h}_1)^\circ\subseteq\mathbbm{i}_{(1)},
\mathbbm{a}_2\cap(\mathbbm{g}_2\cap\mathbbm{h}_2)^\bullet\subseteq\mathbbm{i}_{(2)},
\mathbbm{a}_1\cap(\mathbbm{g}_2\cap\mathbbm{h}_2)\subseteq\mathbbm{i}_{(3)}\}$ for any $\mathfrak{i}=(\mathbbm{i}_{(1)},\mathbbm{i}_{(2)}, \mathbbm{i}_{(3)})\in\mathbb{U}_{\mathbf{g}, \mathbf{h}}$.
Set $\mathfrak{o}\!=\!(\varnothing, \varnothing,\varnothing)$. Notice that $\mathfrak{o}\!\in\!\mathbb{U}_{\mathbf{g},\mathbf{h}}$.
As there is $\mathbf{j}\in \mathbb{E}$ such that $\mathbbm{j}_1=(\mathbbm{g}_0\cap\mathbbm{h}_1)\cup(\mathbbm{g}_1\cap\mathbbm{h}_0)$ and $\mathbbm{j}_2=\mathbbm{g}_2\triangle\mathbbm{h}_2$, notice that $R_\mathbf{j}\in \mathbb{U}_{\mathbf{g}, \mathbf{h}, \mathfrak{i}}$ for any $\mathfrak{i}\in\mathbb{U}_{\mathbf{g}, \mathbf{h}}$. Moreover, $\mathbb{U}_{\mathbf{g}, \mathbf{h}}\!=\!\mathbb{U}_{\mathbf{h},\mathbf{g}}\!=\!\mathbb{U}_{\mathbf{g}, \mathbf{g}}\cap\mathbb{U}_{\mathbf{h}, \mathbf{h}}$ and $\mathbb{U}_{\mathbf{g}, \mathbf{h},\mathfrak{i}}\!=\!\mathbb{U}_{\mathbf{h}, \mathbf{g}, \mathfrak{i}}$ for any $\mathfrak{i}\in\mathbb{U}_{\mathbf{g}, \mathbf{h}}$. For example, assume that $n=\ell_1=2$, $\ell_2=m_1=m_2=3$, $\mathbf{g}=(0,2)$, and $\mathbf{h}=(1, 2)$. Then $\mathbb{U}_{\mathbf{g},\mathbf{h}}=\{\mathfrak{o}, (\varnothing,
\varnothing, \{2\}), (\varnothing, \{2\}, \{2\})\}$. Moreover, notice that $\mathbb{U}_{\mathbf{g},\mathbf{h},\mathfrak{o}}=\{R_{(1, 0)}\}$.
\end{nota}
\begin{nota}\label{N;Notation3.8}
Assume that $\mathbf{g}, \mathbf{h}\!\in\! \mathbb{E}$ and $\mathfrak{i}\!\in\!\mathbb{U}_{\mathbf{g}, \mathbf{h}}$. Set $B_{\mathbf{g}, \mathbf{h}, \mathfrak{i}}\!=\!\!\sum_{R_\mathbf{j}\in \mathbb{U}_{\mathbf{g}, \mathbf{h}, \mathfrak{i}}}\!E_{\mathbf{g}}^*A_\mathbf{j}E_\mathbf{h}^*$. Notice that $\mathbb{U}_{\mathbf{g}, \mathbf{h}, \mathfrak{i}}\!\neq\!\varnothing$. Lemma \ref{L;Lemma3.4} thus implies that $B_{\mathbf{g},\mathbf{h},\mathfrak{i}}$ is a nonzero matrix in $\mathbb{T}$.
\end{nota}
\begin{lem}\label{L;Lemma3.9}
Assume that $\mathbf{g}, \mathbf{h}\!\in\!\mathbb{E}$ and $\mathfrak{i}\!\in\!\mathbb{U}_{\mathbf{g},\mathbf{h}}$. Then $B_{\mathbf{g}, \mathbf{h}, \mathfrak{i}}^T\!=\!B_{\mathbf{h}, \mathbf{g}, \mathfrak{i}}$ and $B_{\mathbf{g},\mathbf{g},\mathfrak{o}}=E_\mathbf{g}^*$.
\end{lem}
\begin{proof}
As $\mathbb{U}_{\mathbf{g}, \mathbf{h},\mathfrak{i}}=\mathbb{U}_{\mathbf{h}, \mathbf{g},\mathfrak{i}}$, the first equation is from Equation \eqref{Eq;2}. By Lemma \ref{L;Lemma3.2}, $\mathbf{h}\!\in\!\mathbb{P}_{\mathbf{g},\mathbf{g}}$ if and only if $\mathbbm{h}_1\!\subseteq\!{\mathbbm{g}_1}^{\circ}\cup\mathbbm{g}_2$ and $\mathbbm{h}_2\!\subseteq\!{\mathbbm{g}_2}^\bullet$. So $\mathbb{U}_{\mathbf{g}, \mathbf{g}, \mathfrak{o}}\!=\!\{R_{0_\mathfrak{S}}\}$ by Lemma \ref{L;Lemma3.1}. The desired lemma thus follows from an application of Equations \eqref{Eq;4}, \eqref{Eq;3}.
\end{proof}
\begin{lem}\label{L;Lemma3.10}
Assume that $\mathbf{g}, \mathbf{h}\in \mathbb{E}$ and $\mathfrak{i}=(\mathbbm{i}_{(1)},\mathbbm{i}_{(2)}, \mathbbm{i}_{(3)}) \in\mathbb{U}_{\mathbf{g},\mathbf{h}}$. Assume that $\mathfrak{j}\in\mathbb{U}_{\mathbf{g},\mathbf{h}}$, $\mathbf{k}\in\mathbb{E}$,
$\mathbbm{k}_1=(\mathbbm{g}_0\cap\mathbbm{h}_1)\cup(\mathbbm{g}_1\cap\mathbbm{h}_0)\cup
\mathbbm{i}_{(1)}\cup(\mathbbm{i}_{(3)}\setminus\mathbbm{i}_{(2)})$, $\mathbbm{k}_2=(\mathbbm{g}_2\triangle\mathbbm{h}_2)\cup\mathbbm{i}_{(2)}$.
Then $R_\mathbf{k}\in\mathbb{U}_{\mathbf{g}, \mathbf{h}, \mathfrak{i}}$.
Moreover, assume that $R_\mathbf{k}\in\mathbb{U}_{\mathbf{g}, \mathbf{h}, \mathfrak{j}}$.
Then $\mathfrak{i}\preceq\mathfrak{j}$.
\end{lem}
\begin{proof}
The first statement is obvious. Assume that $\mathfrak{j}\!=\!(\mathbbm{j}_{(1)}, \mathbbm{j}_{(2)}, \mathbbm{j}_{(3)})$. If $R_\mathbf{k}\!\in\!\mathbb{U}_{\mathbf{g}, \mathbf{h}, \mathfrak{j}}$, then $\mathbbm{i}_{(1)}\!\subseteq\!\mathbbm{j}_{(1)}$, $\mathbbm{i}_{(2)}\!\subseteq\!\mathbbm{j}_{(2)}$, $\mathbbm{i}_{(3)}\!=\!(\mathbbm{i}_{(3)}\setminus\mathbbm{i}_{(2)})\cup\mathbbm{i}_{(2)}\subseteq\mathbbm{j}_{(3)}$. The desired lemma follows.
\end{proof}
\begin{lem}\label{L;Lemma3.11}
Assume that $\mathbf{g}, \mathbf{h}, \mathbf{i}, \mathbf{j}\in \mathbb{E}$, $\mathfrak{k}\in \mathbb{U}_{\mathbf{g}, \mathbf{h}}$, $\mathfrak{l}\in \mathbb{U}_{\mathbf{i}, \mathbf{j}}$. Then $B_{\mathbf{g}, \mathbf{h}, \mathfrak{k}}\!=\!B_{\mathbf{i}, \mathbf{j}, \mathfrak{l}}$ if and only if $\mathbf{g}=\mathbf{i}$, $\mathbf{h}=\mathbf{j}$, and $\mathfrak{k}=\mathfrak{l}$.
\end{lem}
\begin{proof}
One direction is obvious. For the other direction, assume that $B_{\mathbf{g}, \mathbf{h}, \mathfrak{k}}\!=\!B_{\mathbf{i}, \mathbf{j}, \mathfrak{l}}$. Notice that $\mathbf{g}=\mathbf{i}$ and $\mathbf{h}=\mathbf{j}$ by Equation \eqref{Eq;3}.
As $B_{\mathbf{g}, \mathbf{h}, \mathfrak{k}}=B_{\mathbf{g}, \mathbf{h}, \mathfrak{l}}$, Lemma \ref{L;Lemma3.4} thus implies that $\mathbb{U}_{\mathbf{g}, \mathbf{h}, \mathfrak{k}}=\mathbb{U}_{\mathbf{g}, \mathbf{h}, \mathfrak{l}}$. The desired lemma thus follows from Lemma \ref{L;Lemma3.10}.
\end{proof}
\begin{lem}\label{L;Lemma3.12}
$\mathbb{T}$ has an $\F$-linearly independent subset $\{B_{\mathbf{a}, \mathbf{b},\mathfrak{c}}: \mathbf{a}, \mathbf{b}\in \mathbb{E}, \mathfrak{c}\in\mathbb{U}_{\mathbf{a}, \mathbf{b}}\}$.
\end{lem}
\begin{proof}
Let $\mathbb{U}=\{B_{\mathbf{a}, \mathbf{b},\mathfrak{c}}: \mathbf{a}, \mathbf{b}\!\in\! \mathbb{E}, \mathfrak{c}\in\mathbb{U}_{\mathbf{a},\mathbf{b}}\}$. Let $L$ be a nonzero $\F$-linear combination of the matrices in $\mathbb{U}$. Assume that $L=O$. If $M\in\mathbb{U}$, let $c_M$ be the coefficient of $M$ in $L$. Then there exists $N\in\mathbb{U}$ such that $c_N\in\F^\times$. Equation \eqref{Eq;3} thus implies that $N\!=\!E_\mathbf{g}^*NE_\mathbf{h}^*$ for some $\mathbf{g}, \mathbf{h}\in\mathbb{E}$. Hence $\mathbb{V}\!=\!\{A: A\in\mathbb{U}, c_A\in\F^\times, A\!=\!E_\mathbf{g}^*AE_\mathbf{h}^*\}\neq\varnothing$. According to Lemma \ref{L;Lemma3.11}, there exist $i\in\mathbb{N}$ and pairwise distinct $\mathfrak{j}_1, \mathfrak{j}_2,\ldots,\mathfrak{j}_i\in\mathbb{U}_{\mathbf{g}, \mathbf{h}}$ such that $\mathbb{V}=\{B_{\mathbf{g},\mathbf{h},\mathfrak{j}_1}, B_{\mathbf{g},\mathbf{h},\mathfrak{j}_2},\ldots, B_{\mathbf{g},\mathbf{h},\mathfrak{j}_i}\}$. Let us consider the cases $i=1$ and $i>1$.

If $i=1$, there is $c\in \F^\times$ such that $cB_{\mathbf{g}, \mathbf{h},\mathfrak{j}_1}=E_\mathbf{g}^*LE_\mathbf{h}^*=O$. This is a contradiction since $B_{\mathbf{g}, \mathbf{h}, \mathfrak{j}_1}\neq O$. Assume further that $i>1$. Then there is no loss to assume that $\mathfrak{j}_1$ is maximal in $\{\mathfrak{j}_1, \mathfrak{j}_2,\ldots, \mathfrak{j}_i\}$ with respect to the partial order $\preceq$. As $E_\mathbf{g}^*LE_\mathbf{h}^*\!=\!O$, $B_{\mathbf{g}, \mathbf{h}, \mathfrak{j}_1}$ is also an $\F$-linear combination of the matrices in $\{B_{\mathbf{g},\mathbf{h},\mathfrak{j}_2}, B_{\mathbf{g},\mathbf{h},\mathfrak{j}_3},\ldots, B_{\mathbf{g},\mathbf{h},\mathfrak{j}_i}\}$. The combination of Lemmas \ref{L;Lemma3.1}, \ref{L;Lemma3.4}, \ref{L;Lemma3.10} thus implies that $\mathfrak{j}_1\preceq\mathfrak{j}\in\{\mathfrak{j}_2, \mathfrak{j}_3,\ldots, \mathfrak{j}_i\}$.
Hence $\mathfrak{j}_1\!\in\!\{\mathfrak{j}_2, \mathfrak{j}_3,\ldots, \mathfrak{j}_i\}$ by the choice of $\mathfrak{j}_1$. This is also a contradiction. The above two contradictions thus imply that $L\neq O$. The desired lemma thus follows.
\end{proof}
We are now ready to present another $\F$-basis of $\mathbb{T}$ and an additional notation.
\begin{thm}\label{T;Theorem3.13}
$\mathbb{T}$ has an $\F$-basis $\{B_{\mathbf{a}, \mathbf{b},\mathfrak{c}}: \mathbf{a}, \mathbf{b}\in\mathbb{E}, \mathfrak{c}\in\mathbb{U}_{\mathbf{a},\mathbf{b}}\}$ whose cardinality is $2^{n_1+2n_4}3^{n_4}5^{n_1}11^{n_2+n_3}$.
\end{thm}
\begin{proof}
Let $\mathbb{U}=\{(\mathbf{a}, \mathbf{b},\mathfrak{c}): \mathbf{a}, \mathbf{b}\in\mathbb{E}, \mathfrak{c}\in\mathbb{U}_{\mathbf{a},\mathbf{b}}\}$ and $\mathbf{g},\mathbf{h}\in\mathbb{E}$. Let  $i=|(\mathbbm{g}_1\cap\mathbbm{h}_1)^\circ|$, $j=|(\mathbbm{g}_2\cap\mathbbm{h}_2)^\bullet|$, and $k=|\mathbbm{g}_2\cap\mathbbm{h}_2|$.
Notice that $|\mathbb{P}_{\mathbf{g}, \mathbf{h}}|=|\mathbb{U}_{\mathbf{g},\mathbf{h}}|=2^{i-j+k}3^j$ by Lemmas
\ref{L;Lemma3.2} and \ref{L;Lemma3.1}. Since $\mathbf{g},\mathbf{h}$ are chosen from $\mathbb{E}$ arbitrarily and $|\mathbb{P}_{\mathbf{g}, \mathbf{h}}|=|\mathbb{U}_{\mathbf{g},\mathbf{h}}|$, notice that
$$|\mathbb{P}|=\sum_{\mathbf{l}\in\mathbb{E}}\sum_{\mathbf{m}\in\mathbb{E}}|\mathbb{P}_{\mathbf{l},\mathbf{m}}|=
\sum_{\mathbf{l}\in\mathbb{E}}\sum_{\mathbf{m}\in\mathbb{E}}|\mathbb{U}_{\mathbf{l},\mathbf{m}}|=|\mathbb{U}|.$$
The desired theorem follows from combining Lemmas \ref{L;Lemma3.11}, \ref{L;Lemma3.4}, \ref{L;Lemma3.12}, and \ref{L;Lemma3.6}.
\end{proof}
\begin{nota}\label{N;Notation3.14}
Use $\mathbb{B}_1$ to denote $\{E_\mathbf{a}^*A_\mathbf{b}E_\mathbf{c}^*: (\mathbf{a}, \mathbf{c}, \mathbf{b})\in\mathbb{P}\}$. Also use $\mathbb{B}_2$ to denote $\{B_{\mathbf{a}, \mathbf{b},\mathfrak{c}}: \mathbf{a}, \mathbf{b}\in\mathbb{E}, \mathfrak{c}\!\in\!\mathbb{U}_{\mathbf{a},\mathbf{b}}\}$. Let $M\!\in\!\mathbb{T}$ and $\mathbf{g}, \mathbf{h}\!\in\!\mathbb{E}$. If $\mathbf{i}\!\in\!\mathbb{P}_{\mathbf{g}, \mathbf{h}}$, then $c_{\mathbb{B}_1}(M, E_\mathbf{g}^*A_\mathbf{i}E_\mathbf{h}^*)$ and $\mathrm{Supp}_{\mathbb{B}_1}(M)$ are defined by Lemma \ref{L;Lemma3.4}. Abbreviate $c_{\mathbf{g}, \mathbf{i},\mathbf{h}}(M)\!=\!c_{\mathbb{B}_1}(M, E_\mathbf{g}^*A_\mathbf{i}E_\mathbf{h}^*)$ for any $\mathbf{i}\in\mathbb{P}_{\mathbf{g}, \mathbf{h}}$. If $\mathfrak{i}\!\in\!\mathbb{U}_{\mathbf{g}, \mathbf{h}}$, then $c_{\mathbb{B}_2}(M, B_{\mathbf{g}, \mathbf{h}, \mathfrak{i}})$ and $\mathrm{Supp}_{\mathbb{B}_2}(M)$ are also defined by Theorem \ref{T;Theorem3.13}. Similarly, also abbreviate $c_{\mathbf{g},\mathbf{h},\mathfrak{i}}(M)=c_{\mathbb{B}_2}(M, B_{\mathbf{g}, \mathbf{h},\mathfrak{i}})$ for any $\mathfrak{i}\in\mathbb{U}_{\mathbf{g}, \mathbf{h}}$.
\end{nota}
We next study the structure constants of $\mathbb{B}_2$ in $\mathbb{T}$. We begin with some notation.
\begin{nota}\label{N;Notation3.15}
Assume that $\mathbf{g}, \mathbf{h}, \mathbf{i}\in \mathbb{E}$. Assume that $\mathfrak{j}=(\mathbbm{j}_{(1)},\mathbbm{j}_{(2)}, \mathbbm{j}_{(3)})\in \mathbb{U}_{\mathbf{g},\mathbf{h}}$ and $\mathfrak{k}=(\mathbbm{k}_{(1)},\mathbbm{k}_{(2)}, \mathbbm{k}_{(3)})\in \mathbb{U}_{\mathbf{h},\mathbf{i}}$. Let $\mathbbm{l}_{(1)}=((\mathbbm{g}_1\cap\mathbbm{i}_1)^\circ\setminus \mathbbm{h}_1)\cup(\mathbbm{g}_1\cap\mathbbm{i}_1\cap(\mathbbm{j}_{(1)}\cup\mathbbm{k}_{(1)}))$, $\mathbbm{l}_{(2)}\!=\!((\mathbbm{g}_2\cap\mathbbm{i}_2)^\bullet\setminus \mathbbm{h}_2)\cup(\mathbbm{g}_2\cap\mathbbm{i}_2\cap(\mathbbm{j}_{(2)}\cup\mathbbm{k}_{(2)}))$,
$\mathbbm{l}_{(3)}\!=\!((\mathbbm{g}_2\cap\mathbbm{i}_2)\setminus \mathbbm{h}_2)\cup(\mathbbm{g}_2\cap\mathbbm{i}_2\cap(\mathbbm{j}_{(3)}\cup\mathbbm{k}_{(3)}))$.
Set $(\mathbf{g},\mathbf{h}, \mathbf{i},\mathfrak{j}, \mathfrak{k})\!=\!(\mathbbm{l}_{(1)}, \mathbbm{l}_{(2)}, \mathbbm{l}_{(3)})$. Notice that $(\mathbf{g},\mathbf{h},\mathbf{i},\mathfrak{j}, \mathfrak{k})\in\mathbb{U}_{\mathbf{g},\mathbf{i}}$ and $B_{\mathbf{g}, \mathbf{i}, (\mathbf{g},\mathbf{h},\mathbf{i},\mathfrak{j}, \mathfrak{k})}$ is defined. If $\mathbf{g}\!=\!\mathbf{h}\!=\!\mathbf{i}$, notice that $(\mathbf{g},\mathbf{g},\mathbf{g},\mathfrak{j}, \mathfrak{k})=(\mathbbm{j}_{(1)}\cup\mathbbm{k}_{(1)}, \mathbbm{j}_{(2)}\cup\mathbbm{k}_{(2)}, \mathbbm{j}_{(3)}\cup\mathbbm{k}_{(3)})$ and denote $(\mathbf{g},\mathbf{g},\mathbf{g},\mathfrak{j}, \mathfrak{k})$ by $\mathfrak{j}\cup\mathfrak{k}$. For example, assume that $n=\ell_1=2$, $\ell_2=m_1=m_2=3$, $\mathbf{g}=(0,2)$, $\mathbf{h}=(1,2)$, and $\mathfrak{i}=(\varnothing, \varnothing, \{2\})$.
It is obvious that $(\mathbf{g},\mathbf{h},\mathbf{g}, \mathfrak{i}, \mathfrak{i})=\mathfrak{i}\in\mathbb{U}_{\mathbf{g}, \mathbf{g}}$.
\end{nota}
\begin{lem}\label{L;Lemma3.16}
Assume that $\mathbf{g}, \mathbf{h}, \mathbf{i}\in \mathbb{E}$, $\mathfrak{j}\in \mathbb{U}_{\mathbf{g},\mathbf{h}}$, $\mathfrak{k}\in \mathbb{U}_{\mathbf{h},\mathbf{i}}$, $R_\mathbf{l}\in R_\mathbf{g}R_\mathbf{i}\cap\mathbb{U}_{\mathbf{g},\mathbf{h},\mathfrak{j}}\mathbb{U}_{\mathbf{h},\mathbf{i},\mathfrak{k}}$. Then $R_\mathbf{l}\in \mathbb{U}_{\mathbf{g}, \mathbf{i}, (\mathbf{g},\mathbf{h},\mathbf{i},\mathfrak{j}, \mathfrak{k})}$.
\end{lem}
\begin{proof}
As $R_\mathbf{l}\in R_\mathbf{g}R_\mathbf{i}$, notice that $\mathbf{l}\in\mathbb{P}_{\mathbf{g},\mathbf{i}}$ by Lemma \ref{L;Lemma3.2}.
There are $R_\mathbf{m}\!\in\! \mathbb{U}_{\mathbf{g}, \mathbf{h}, \mathfrak{j}}$ and
$R_\mathbf{q}\in\mathbb{U}_{\mathbf{h}, \mathbf{i}, \mathfrak{k}}$ such that $R_\mathbf{l}\in R_\mathbf{m}R_\mathbf{q}$. Let $\mathfrak{j}=(\mathbbm{j}_{(1)}, \mathbbm{j}_{(2)}, \mathbbm{j}_{(3)})$ and $\mathfrak{k}=(\mathbbm{k}_{(1)}, \mathbbm{k}_{(2)}, \mathbbm{k}_{(3)})$. As $R_\mathbf{m}\in\mathbb{U}_{\mathbf{g}, \mathbf{h}, \mathfrak{j}}$ and $R_\mathbf{q}\in\mathbb{U}_{\mathbf{h}, \mathbf{i}, \mathfrak{k}}$, notice that $(\mathbbm{g}_1\cap\mathbbm{h}_1)^\circ\cap\mathbbm{m}_1\subseteq\mathbbm{j}_{(1)}, (\mathbbm{g}_2\cap\mathbbm{h}_2)^\bullet\cap\mathbbm{m}_2\subseteq\mathbbm{j}_{(2)}$, $(\mathbbm{g}_2\cap\mathbbm{h}_2)\cap\mathbbm{m}_1\subseteq\mathbbm{j}_{(3)}$,
$(\mathbbm{h}_1\cap\mathbbm{i}_1)^\circ\cap\mathbbm{q}_1\subseteq\mathbbm{k}_{(1)}$, $(\mathbbm{h}_2\cap\mathbbm{i}_2)^\bullet\cap\mathbbm{q}_2
\subseteq\mathbbm{k}_{(2)}$, $(\mathbbm{h}_2\cap\mathbbm{i}_2)\cap\mathbbm{q}_1\subseteq\mathbbm{k}_{(3)}$, $\mathbbm{l}_1\subseteq\mathbbm{m}_1\cup\mathbbm{q}_1\cup(\mathbbm{m}_2\cap\mathbbm{q}_2)$,  $\mathbbm{l}_2\subseteq\mathbbm{m}_2\cup\mathbbm{q}_2$,
$\mathbbm{m}_2\subseteq\mathbbm{g}_2\cup\mathbbm{h}_2$,
$\mathbbm{q}_2\subseteq\mathbbm{h}_2\cup\mathbbm{i}_2$ by Lemma \ref{L;Lemma3.2}.

Notice that $(\mathbbm{g}_1\cap\mathbbm{i}_1)^\circ\cap\mathbbm{l}_1\subseteq((\mathbbm{g}_1\cap\mathbbm{i}_1)^\circ\setminus \mathbbm{h}_1)\cup((\mathbbm{g}_1\cap\mathbbm{h}_1\cap\mathbbm{i}_1)^\circ\cap(\mathbbm{m}_1\cup\mathbbm{q}_1\cup(\mathbbm{m}_2\cap\mathbbm{q}_2)))$. Notice that
$(\mathbbm{g}_1\cap\mathbbm{i}_1)^\circ\cap\mathbbm{l}_1\subseteq((\mathbbm{g}_1\cap\mathbbm{i}_1)^\circ\setminus \mathbbm{h}_1)\cup(\mathbbm{g}_1\cap\mathbbm{i}_1\cap(\mathbbm{j}_{(1)}\cup\mathbbm{k}_{(1)}))$ by Lemma \ref{L;Lemma3.1}.
Notice that $(\mathbbm{g}_2\cap\mathbbm{i}_2)\cap\mathbbm{l}_1\subseteq((\mathbbm{g}_2\cap\mathbbm{i}_2)\setminus \mathbbm{h}_2)\cup(\mathbbm{g}_2\cap\mathbbm{h}_2\cap\mathbbm{i}_2\cap(\mathbbm{m}_1\cup\mathbbm{q}_1\cup(\mathbbm{m}_2\cap\mathbbm{q}_2)))$.
Notice that $(\mathbbm{g}_2\cap\mathbbm{i}_2)\cap\mathbbm{l}_1\subseteq((\mathbbm{g}_2\cap\mathbbm{i}_2)\setminus \mathbbm{h}_2)\cup(\mathbbm{g}_2\cap\mathbbm{i}_2\cap(\mathbbm{j}_{(3)}\cup\mathbbm{k}_{(3)}))$ by Lemma \ref{L;Lemma3.2}.
Notice that $(\mathbbm{g}_2\cap\mathbbm{i}_2)^\bullet\cap\mathbbm{l}_2\subseteq((\mathbbm{g}_2\cap\mathbbm{i}_2)^\bullet\setminus \mathbbm{h}_2)\cup((\mathbbm{g}_2\cap\mathbbm{h}_2\cap\mathbbm{i}_2)^\bullet\cap(\mathbbm{m}_2\cup\mathbbm{q}_2))$.
Therefore $(\mathbbm{g}_2\cap\mathbbm{i}_2)^\bullet\!\cap\!\mathbbm{l}_2\!\!\subseteq\!\!((\mathbbm{g}_2\cap\mathbbm{i}_2)^\bullet\setminus \mathbbm{h}_2)\cup(\mathbbm{g}_2\cap\mathbbm{i}_2\cap(\mathbbm{j}_{(2)}\cup\mathbbm{k}_{(2)}))$. The desired lemma follows.
\end{proof}
\begin{lem}\label{L;Lemma3.17}
Assume that $\mathbf{g}, \mathbf{h}, \mathbf{i}\in\mathbb{E}$, $\mathfrak{j}\in\mathbb{U}_{\mathbf{g},\mathbf{h}}$, and $\mathfrak{k}\in\mathbb{U}_{\mathbf{h},\mathbf{i}}$. Then
\begin{align*}
\mathrm{Supp}_{\mathbb{B}_1}(B_{\mathbf{g}, \mathbf{h}, \mathfrak{j}}B_{\mathbf{h}, \mathbf{i}, \mathfrak{k}})\subseteq\{E_\mathbf{g}^*A_\mathbf{a}E_\mathbf{i}^*: R_\mathbf{a}\in\mathbb{U}_{\mathbf{g}, \mathbf{i}, (\mathbf{g},\mathbf{h},\mathbf{i},\mathfrak{j}, \mathfrak{k})}\}.
\end{align*}
\end{lem}
\begin{proof}
The desired lemma follows from an application of Lemmas \ref{L;Lemma3.5} and \ref{L;Lemma3.16}.
\end{proof}
\begin{nota}\label{N;Notation3.18}
Assume that $\mathbf{g}, \mathbf{h}, \mathbf{i}\in \mathbb{E}$ and $\mathbb{U}, \mathbb{V}, \mathbb{W}\subseteq[1,n]$. For convenience, set
\begin{align}\label{Eq;8}
k_{(\mathbb{U},\mathbb{V}, \mathbb{W})}=\prod_{j\in\mathbb{U}}(m_j-1)\prod_{k\in\mathbb{V}}(\ell_k-1)m_k\prod_{\ell\in \mathbb{W}\setminus \mathbb{V}}m_\ell,
\end{align}
where a product over the empty set equals one. Hence $k_\mathfrak{j}$ is defined for any $\mathfrak{j}\in \mathbb{U}_{\mathbf{g}, \mathbf{h}}$. Assume that $\mathfrak{j}\!=\!(\mathbbm{j}_{(1)}, \mathbbm{j}_{(2)}, \mathbbm{j}_{(3)})\in\mathbb{U}_{\mathbf{g},\mathbf{h}}$ and $\mathfrak{k}\!=\!(\mathbbm{k}_{(1)}, \mathbbm{k}_{(2)}, \mathbbm{k}_{(3)})\in\mathbb{U}_{\mathbf{h},\mathbf{i}}$. Let $\mathfrak{j}\cap\mathfrak{k}$, $\mathfrak{j}\setminus \mathbf{i}$, $\mathbf{i}\cap\mathfrak{j}$ be $(\mathbbm{j}_{(1)}\cap\mathbbm{k}_{(1)}, \mathbbm{j}_{(2)}\cap\mathbbm{k}_{(2)}, \mathbbm{j}_{(3)}\cap\mathbbm{k}_{(3)})$, $(\mathbbm{j}_{(1)}\setminus\mathbbm{i}_1, \mathbbm{j}_{(2)}\setminus\mathbbm{i}_2, \mathbbm{j}_{(3)}\setminus\mathbbm{i}_2)$, $(\mathbbm{i}_1\cap\mathbbm{j}_{(1)}, \mathbbm{i}_2\cap\mathbbm{j}_{(2)}, \mathbbm{i}_2\cap\mathbbm{j}_{(3)})$, respectively. So $\mathfrak{j}\cap\mathfrak{k}\in\mathbb{U}_{\mathbf{g},\mathbf{h}}\cap\mathbb{U}_{\mathbf{h},\mathbf{i}}$, $\mathfrak{j}\setminus \mathbf{i}\in\mathbb{U}_{\mathbf{g},\mathbf{h}}$, $\mathbf{i}\cap\mathfrak{j}\in\mathbb{U}_{\mathbf{i}, \mathbf{i}}$, $k_{\mathfrak{j}\cap\mathfrak{k}}, k_{\mathfrak{j}\setminus \mathbf{i}}, k_{\mathbf{i}\cap\mathfrak{j}}$ are defined.

If $\mathbb{U}=(\mathbbm{g}_1\cap\mathbbm{i}_1)\setminus\mathbbm{h}_1$ and $\mathbb{V}=(\mathbbm{g}_2\cap\mathbbm{i}_2)\setminus\mathbbm{h}_2$, set $k_{(\mathbf{g}, \mathbf{h}, \mathbf{i})}\!=\!k_{(\mathbb{U},\mathbb{V}, \mathbb{V})}$. If $\mathbb{U}=\mathbbm{h}_1\setminus(\mathbbm{g}_1\cup\mathbbm{i}_1)$ and $\mathbb{V}\!=\!\mathbbm{h}_2\setminus(\mathbbm{g}_2\cup\mathbbm{i}_2)$, set $k_{[\mathbf{g}, \mathbf{h}, \mathbf{i}]}\!=\!k_{(\mathbb{U},\mathbb{V}, \mathbb{V})}$. For example, assume that $n=\ell_1=2$, $\ell_2=m_1=m_2=3$, $\mathbf{g}=\mathbf{h}=(0,2)$, and $\mathbf{i}=(1,2)$. It is clear that  $k_{(\mathbf{g},\mathbf{h},\mathbf{i})}=k_{[\mathbf{g},\mathbf{h},\mathbf{i}]}\!=\!1$.
\end{nota}
\begin{lem}\label{L;Lemma3.19}
Assume that $\mathbb{U}, \mathbb{V}, \mathbb{W}\subseteq[1, n]$. Then $k_{(\mathbb{U},\mathbb{V},\mathbb{W})}=k_{(\mathbb{U}^\circ,\mathbb{V},\mathbb{W})}$. Moreover, assume that $\mathbb{V}\subseteq\mathbb{W}$. Then $k_{(\mathbb{U}, \mathbb{V}, \mathbb{W})}=k_{(\mathbb{U}, \mathbb{V}^\bullet, \mathbb{W})}$. In particular, if $\mathbf{g}\in\mathbb{E}$ and a product over the empty set equals one, then
$$k_\mathbf{g}=k_{(\mathbbm{g}_1, \mathbbm{g}_2,\mathbbm{g}_2)}=k_{({\mathbbm{g}_1}^\circ, {\mathbbm{g}_2}^\bullet, \mathbbm{g}_2)}=\prod_{h\in\mathbbm{g}_1}(m_h-1)\prod_{i\in\mathbbm{g}_2}(\ell_i-1)m_i.$$
\end{lem}
\begin{proof}
The desired lemma follows from the above hypotheses and Equation \eqref{Eq;8}.
\end{proof}
\begin{lem}\label{L;Lemma3.20}
Assume that $\mathbb{U}, \mathbb{U}^1, \mathbb{V}, \mathbb{V}^1, \mathbb{W}, \mathbb{W}^1\subseteq [1,n]$. If $\mathbb{U}\cap\mathbb{U}^1=\varnothing$, $\mathbb{V}\cap\mathbb{V}^1=\varnothing$, $\mathbb{V}\cap \mathbb{W}^1=\varnothing$, $\mathbb{V}^1\cap\mathbb{W}=\varnothing$, $\mathbb{W}\cap\mathbb{W}^1=\varnothing$, then $k_{(\mathbb{U}, \mathbb{V}, \mathbb{W})}k_{(\mathbb{U}^1, \mathbb{V}^1, \mathbb{W}^1)}=k_{(\mathbb{U}\cup\mathbb{U}^1, \mathbb{V}\cup\mathbb{V}^1, \mathbb{W}\cup\mathbb{W}^1)}$. In particular, if $\mathbf{g}, \mathbf{h}, \mathbf{i}\in\mathbb{E}$ and $\mathfrak{j}\in\mathbb{U}_{\mathbf{g},\mathbf{h}}$, then $k_{\mathfrak{j}\setminus\mathbf{i}}k_{\mathbf{i}\cap\mathfrak{j}}=k_\mathfrak{j}$.
\end{lem}
\begin{proof}
The desired lemma follows from the above hypotheses and Equation \eqref{Eq;8}.
\end{proof}
\begin{lem}\label{L;Lemma3.21}
Assume that $\mathbb{U}, \mathbb{V}, \mathbb{W}$ are pairwise disjoint subsets of $[1, n]$. Assume that $\mathbb{W}^1=\{\mathbf{a}: \mathbf{a}\in\mathbb{E}, \mathbb{U}\subseteq\mathbbm{a}_1\subseteq\mathbb{U}\cup\mathbb{W}, \mathbbm{a}_2=\mathbb{V}\}$. Then $k_{(\mathbb{U}, \mathbb{V}, \mathbb{W})}=\sum_{\mathbf{g}\in\mathbb{W}^1}k_\mathbf{g}$.
\end{lem}
\begin{proof}
As $\mathbb{U}, \mathbb{V}, \mathbb{W}$ are pairwise disjoint subsets of $[1,n]$ and Equation \eqref{Eq;8} holds,
$$k_{(\mathbb{U}, \mathbb{V}, \mathbb{W})}=k_{(\mathbb{U},\mathbb{V}, \varnothing)}k_{(\varnothing,\varnothing,\mathbb{W})}=k_{(\mathbb{U},\mathbb{V},\varnothing)}\sum_{\mathbb{U}^1\subseteq\mathbb{W}}k_{(\mathbb{U}^1,\varnothing, \varnothing)}
=\sum_{\mathbb{U}^1\subseteq\mathbb{W}}k_{(\mathbb{U}\cup\mathbb{U}^1,\mathbb{V}, \varnothing)}=\sum_{\mathbf{g}\in\mathbb{W}^1}k_\mathbf{g}$$
by combining Lemmas \ref{L;Lemma3.20}, \ref{L;Lemma3.1}, and \ref{L;Lemma3.19}. The desired lemma thus follows.
\end{proof}
\begin{lem}\label{L;Lemma3.22}
Assume that $\mathbf{g}, \mathbf{h}, \mathbf{i}\!\in\!\mathbb{E}$. Assume that $\mathfrak{j}\!=\!(\mathbbm{j}_{(1)},\mathbbm{j}_{(2)}, \mathbbm{j}_{(3)} )\!\in\!\mathbb{U}_{\mathbf{g},\mathbf{h}}$ and $\mathfrak{k}\!\!=\!\!(\mathbbm{k}_{(1)},\mathbbm{k}_{(2)}, \mathbbm{k}_{(3)})\!\!\in\!\!\mathbb{U}_{\mathbf{h},\mathbf{i}}$. Then $[1,n]$ is a union of the pairwise disjoint subsets $\mathbbm{h}_0$, $\mathbbm{h}_1\setminus(\mathbbm{g}_1\cup\mathbbm{i}_1)$, $\mathbbm{h}_2\setminus(\mathbbm{g}_2\cup\mathbbm{i}_2)$, $\mathbbm{j}_{(1)}\setminus \mathbbm{i}_1$, $\mathbbm{j}_{(2)}\setminus \mathbbm{i}_2$, $\mathbbm{k}_{(1)}\setminus \mathbbm{g}_1$, $\mathbbm{k}_{(2)}\setminus \mathbbm{g}_2$, $\mathbbm{j}_{(1)}\cap\mathbbm{k}_{(1)}$, $\mathbbm{j}_{(2)}\cap\mathbbm{k}_{(2)}$, $(\mathbbm{g}_1\cap\mathbbm{k}_{(1)})\setminus\mathbbm{j}_{(1)}$, $(\mathbbm{g}_1\cap\mathbbm{h}_1)\setminus(\mathbbm{i}_1\cup\mathbbm{j}_{(1)})$,
$(\mathbbm{g}_2\cap\mathbbm{k}_{(3)})\setminus\mathbbm{j}_{(3)}$, $(\mathbbm{g}_2\cap\mathbbm{h}_2)\setminus(\mathbbm{i}_2\cup\mathbbm{j}_{(3)})$, $\mathbbm{j}_{(3)}\setminus(\mathbbm{i}_2\cup\mathbbm{j}_{(2)})$, $(\mathbbm{j}_{(3)}\cap\mathbbm{k}_{(2)})\setminus\mathbbm{j}_{(2)}$,
$(\mathbbm{h}_1\cap\mathbbm{i}_1)\setminus\mathbbm{k}_{(1)}$, $(\mathbbm{h}_2\cap\mathbbm{i}_2)\setminus\mathbbm{k}_{(3)}$, $\mathbbm{k}_{(3)}\setminus(\mathbbm{g}_2\cup\mathbbm{k}_{(2)})$, $(\mathbbm{j}_{(3)}\cap\mathbbm{k}_{(3)})\setminus\mathbbm{k}_{(2)}$.
\end{lem}
\begin{proof}
The desired lemma follows from Lemma \ref{L;Lemma3.1} and a direct computation.
\end{proof}
\begin{nota}\label{N;Notation3.23}
Assume that $\mathbf{g}, \mathbf{h}, \mathbf{i}\!\in\! \mathbb{E}$. Assume that $\mathfrak{j}\!=\!(\mathbbm{j}_{(1)},\mathbbm{j}_{(2)}, \mathbbm{j}_{(3)} )\!\in\!\mathbb{U}_{\mathbf{g},\mathbf{h}}$ and $\mathfrak{k}\!=\!(\mathbbm{k}_{(1)},\mathbbm{k}_{(2)}, \mathbbm{k}_{(3)})\!\in\!\mathbb{U}_{\mathbf{h},\mathbf{i}}$.
Let $\mathbb{U}_{\mathbf{g}, \mathbf{h}, \mathbf{i}, \mathfrak{j}, \mathfrak{k}, 1}$ be the union of $\mathbbm{h}_0$, $\mathbbm{h}_1\setminus(\mathbbm{g}_1\cup\mathbbm{i}_1)$, $\mathbbm{h}_2\setminus(\mathbbm{g}_2\cup\mathbbm{i}_2)$, $\mathbbm{j}_{(1)}\setminus \mathbbm{i}_1$, $\mathbbm{j}_{(2)}\setminus \mathbbm{i}_2$, $\mathbbm{k}_{(1)}\setminus \mathbbm{g}_1$, $\mathbbm{k}_{(2)}\setminus \mathbbm{g}_2$, $\mathbbm{j}_{(1)}\cap\mathbbm{k}_{(1)}$, $\mathbbm{j}_{(2)}\cap\mathbbm{k}_{(2)}$. Let $\mathbb{U}_{\mathbf{g}, \mathbf{h}, \mathbf{i}, \mathfrak{j}, \mathfrak{k}, 2}$ be the union of $(\mathbbm{g}_1\cap\mathbbm{k}_{(1)})\setminus\mathbbm{j}_{(1)}$, $(\mathbbm{g}_1\cap\mathbbm{h}_1)\setminus(\mathbbm{i}_1\cup\mathbbm{j}_{(1)})$,
$(\mathbbm{g}_2\cap\mathbbm{k}_{(3)})\setminus\mathbbm{j}_{(3)}$, $(\mathbbm{g}_2\cap\mathbbm{h}_2)\setminus(\mathbbm{i}_2\cup\mathbbm{j}_{(3)})$, $\mathbbm{j}_{(3)}\setminus(\mathbbm{i}_2\cup\mathbbm{j}_{(2)})$, $(\mathbbm{j}_{(3)}\cap\mathbbm{k}_{(2)})\setminus\mathbbm{j}_{(2)}$.
Set $\mathbb{U}_{\mathbf{g}, \mathbf{h}, \mathbf{i}, \mathfrak{j}, \mathfrak{k}, 3}=[1,n]\setminus(\mathbb{U}_{\mathbf{g}, \mathbf{h}, \mathbf{i}, \mathfrak{j}, \mathfrak{k}, 1}\cup\mathbb{U}_{\mathbf{g}, \mathbf{h}, \mathbf{i}, \mathfrak{j}, \mathfrak{k}, 2})$. Let $\mathbf{y}, \mathbf{z}\in\mathbb{X}$. Then Lemmas \ref{L;Lemma3.22} and \ref{L;Lemma3.1} give a unique $\mathbf{w}\in\mathbb{X}$, where
$\mathbf{w}_\ell\!=\!\mathbf{x}_\ell$, $\mathbf{w}_{m}\!=\!\mathbf{y}_m$, $\mathbf{w}_q\!=\!\mathbf{z}_q$ for any $\ell\!\in\!\mathbb{U}_{\mathbf{g}, \mathbf{h}, \mathbf{i}, \mathfrak{j}, \mathfrak{k}, 1}$, $m\!\in\!\mathbb{U}_{\mathbf{g}, \mathbf{h}, \mathbf{i}, \mathfrak{j}, \mathfrak{k}, 2}$, and $q\!\in\!\mathbb{U}_{\mathbf{g}, \mathbf{h}, \mathbf{i}, \mathfrak{j}, \mathfrak{k}, 3}$. Define $(\mathbf{y},\mathbf{z})_{\mathbf{g}, \mathbf{h}, \mathbf{i}, \mathfrak{j}, \mathfrak{k}}=\mathbf{w}$.
\end{nota}
\begin{lem}\label{L;Lemma3.24}
Assume that $\mathbf{y}, \mathbf{z}\in \mathbb{X}$, $\mathbf{g}, \mathbf{h}, \mathbf{i}\in \mathbb{E}$, $\mathfrak{j}=(\mathbbm{j}_{(1)},\mathbbm{j}_{(2)}, \mathbbm{j}_{(3)} )\in\mathbb{U}_{\mathbf{g},\mathbf{h}}$, and $\mathfrak{k}\!=\!(\mathbbm{k}_{(1)},\mathbbm{k}_{(2)}, \mathbbm{k}_{(3)} )\in\mathbb{U}_{\mathbf{h},\mathbf{i}}$. If $R_\mathbf{l}\!\in\! \mathbb{U}_{\mathbf{g}, \mathbf{h}, \mathfrak{j}}$, $R_\mathbf{m}\!\in\!\mathbb{U}_{\mathbf{h}, \mathbf{i}, \mathfrak{k}}$, $\mathbf{w}\in\mathbf{y}R_\mathbf{l}\cap\mathbf{x}R_\mathbf{h}\cap\mathbf{z}R_\mathbf{m}$, then there are $\mathbf{q}\in\mathbb{E} $ and $\mathbb{U}\!\subseteq\!(\mathbbm{j}_{(3)}\setminus(\mathbbm{i}_2\cup\mathbbm{j}_{(2)}))\cup
(\mathbbm{k}_{(3)}\setminus(\mathbbm{g}_2\cup\mathbbm{k}_{(2)}))\cup((\mathbbm{j}_{(3)}\cap\mathbbm{k}_{(3)})\setminus
(\mathbbm{j}_{(2)}\cap\mathbbm{k}_{(2)}))$ such that $\mathbf{w}\in(\mathbf{y},\mathbf{z})_{\mathbf{g}, \mathbf{h}, \mathbf{i},\mathfrak{j}, \mathfrak{k}}R_\mathbf{q}$, $\mathbbm{q}_1=(\mathbbm{h}_1\setminus(\mathbbm{g}_1\cup\mathbbm{i}_1))\cup(\mathbbm{j}_{(1)}\setminus\mathbbm{i}_1)\cup
(\mathbbm{k}_{(1)}\setminus\mathbbm{g}_1)\cup(\mathbbm{j}_{(1)}\cap\mathbbm{k}_{(1)})\cup\mathbb{U}$, and
$\mathbbm{q}_2=(\mathbbm{h}_2\setminus(\mathbbm{g}_2\cup\mathbbm{i}_2))\cup(\mathbbm{j}_{(2)}\setminus\mathbbm{i}_2)\cup
(\mathbbm{k}_{(2)}\setminus\mathbbm{g}_2)\cup(\mathbbm{j}_{(2)}\cap\mathbbm{k}_{(2)})$.
\end{lem}
\begin{proof}
The union of $(\mathbbm{j}_{(3)}\cap\mathbbm{l}_1)\setminus(\mathbbm{i}_2\cup\mathbbm{j}_{(2)})$, $(\mathbbm{j}_{(3)}\cap\mathbbm{k}_{(2)}\cap\mathbbm{l}_1)\setminus
\mathbbm{j}_{(2)}$, $(\mathbbm{k}_{(3)}\cap\mathbbm{m}_1)\setminus(\mathbbm{g}_2\cup\mathbbm{k}_{(2)})$, and $(\mathbbm{j}_{(3)}\cap\mathbbm{k}_{(3)}\cap\mathbbm{m}_1)\setminus
\mathbbm{k}_{(2)}$ is denoted by $\mathbb{U}$. By Lemmas \ref{L;Lemma3.1} and \ref{L;Lemma3.22}, there exists $\mathbf{q}\in\mathbb{E}$ such that
$\mathbbm{q}_1=(\mathbbm{h}_1\setminus(\mathbbm{g}_1\cup\mathbbm{i}_1))\cup(\mathbbm{j}_{(1)}\setminus\mathbbm{i}_1)\cup
(\mathbbm{k}_{(1)}\setminus\mathbbm{g}_1)\cup(\mathbbm{j}_{(1)}\cap\mathbbm{k}_{(1)})\cup\mathbb{U}$ and $\mathbbm{q}_2=(\mathbbm{h}_2\setminus(\mathbbm{g}_2\cup\mathbbm{i}_2))\cup(\mathbbm{j}_{(2)}\setminus\mathbbm{i}_2)\cup
(\mathbbm{k}_{(2)}\setminus\mathbbm{g}_2)\cup(\mathbbm{j}_{(2)}\cap\mathbbm{k}_{(2)})$. Abbreviate $\mathbf{v}=(\mathbf{y},\mathbf{z})_{\mathbf{g}, \mathbf{h}, \mathbf{i},\mathfrak{j}, \mathfrak{k}}$.

If $r\in \mathbbm{q}_1\setminus \mathbb{U}$, then $\mathbf{w}_r\sim_{\mathbb{G}_r}\mathbf{x}_r=\mathbf{v}_r$.
If $r\in\mathbbm{q}_2$, then $\mathbf{w}_r\not\sim_{\mathbb{G}_r}\mathbf{x}_r=\mathbf{v}_r$. If $r\in\mathbb{U}$,
then $\mathbf{w}_r\sim_{\mathbb{G}_r}\mathbf{y}_r=\mathbf{v}_r$ or $\mathbf{w}_r\sim_{\mathbb{G}_r}\mathbf{z}_r=\mathbf{v}_r$.
If $r$ is in the union of $\mathbbm{h}_0$, $(\mathbbm{g}_1\cap\mathbbm{k}_{(1)})\setminus\mathbbm{j}_{(1)}$,
$(\mathbbm{g}_1\cap\mathbbm{h}_1)\setminus(\mathbbm{i}_1\cup\mathbbm{j}_{(1)})$, $(\mathbbm{g}_2\cap\mathbbm{k}_{(3)})\setminus\mathbbm{j}_{(3)}$,
$(\mathbbm{g}_2\cap\mathbbm{h}_2)\setminus(\mathbbm{i}_2\cup\mathbbm{j}_{(3)})$,
$(\mathbbm{j}_{(3)}\cap\mathbbm{l}_0)\setminus(\mathbbm{i}_2\cup\mathbbm{j}_{(2)})$, $(\mathbbm{j}_{(3)}\!\cap\!\mathbbm{k}_{(2)}\cap\mathbbm{l}_0)\setminus\mathbbm{j}_{(2)}$,
$(\mathbbm{h}_1\cap\mathbbm{i}_1)\setminus\mathbbm{k}_{(1)}$,
$(\mathbbm{h}_2\cap\mathbbm{i}_2)\setminus\mathbbm{k}_{(3)}$,
$(\mathbbm{k}_{(3)}\!\cap\!\mathbbm{m}_0)\!\setminus\!(\mathbbm{g}_2\cup\mathbbm{k}_{(2)})$, and
$(\mathbbm{j}_{(3)}\!\cap\mathbbm{k}_{(3)}\!\cap\!\mathbbm{m}_0)\!\setminus\!\mathbbm{k}_{(2)}$, then
$\mathbf{w}_r=\mathbf{v}_r$. The desired lemma follows from Lemma \ref{L;Lemma3.22}.
\end{proof}
\begin{lem}\label{L;Lemma3.25}
Assume that $\mathbf{y}, \mathbf{z}\in\mathbb{X}$, $\mathbf{g}, \mathbf{h}, \mathbf{i}\in \mathbb{E}$, $\mathfrak{j}=(\mathbbm{j}_{(1)},\mathbbm{j}_{(2)}, \mathbbm{j}_{(3)} )\in\mathbb{U}_{\mathbf{g},\mathbf{h}}$, and $\mathfrak{k}=(\mathbbm{k}_{(1)},\mathbbm{k}_{(2)}, \mathbbm{k}_{(3)} )\in\mathbb{U}_{\mathbf{h},\mathbf{i}}$. Assume that $\mathbf{y}\in\mathbf{x}R_\mathbf{g}$ and $\mathbf{z}\in\mathbf{x}R_\mathbf{i}$. If there are $\mathbf{l}\in \mathbb{E}$ and $\mathbb{U}\subseteq(\mathbbm{j}_{(3)}\!\setminus\!(\mathbbm{i}_2\!\cup\!\mathbbm{j}_{(2)}))\cup
(\mathbbm{k}_{(3)}\setminus(\mathbbm{g}_2\cup\mathbbm{k}_{(2)}))\cup((\mathbbm{j}_{(3)}\cap\mathbbm{k}_{(3)})\setminus
(\mathbbm{j}_{(2)}\cap\mathbbm{k}_{(2)}))$ such that $\mathbf{w}\in(\mathbf{y}, \mathbf{z})_{\mathbf{g}, \mathbf{h}, \mathbf{i}, \mathfrak{j}, \mathfrak{k}}R_\mathbf{l}$,
$\mathbbm{l}_1=(\mathbbm{h}_1\setminus(\mathbbm{g}_1\cup\mathbbm{i}_1))\cup(\mathbbm{j}_{(1)}\setminus\mathbbm{i}_1)\cup
(\mathbbm{k}_{(1)}\setminus\mathbbm{g}_1)\cup(\mathbbm{j}_{(1)}\cap\mathbbm{k}_{(1)})\cup\mathbb{U}$, and
$\mathbbm{l}_2=(\mathbbm{h}_2\setminus(\mathbbm{g}_2\cup\mathbbm{i}_2))\cup(\mathbbm{j}_{(2)}\setminus\mathbbm{i}_2)\cup
(\mathbbm{k}_{(2)}\setminus\mathbbm{g}_2)\cup(\mathbbm{j}_{(2)}\cap\mathbbm{k}_{(2)})$, then $\mathbf{w}\in \mathbf{x}R_\mathbf{h}$.
\end{lem}
\begin{proof}
Define $\mathbf{v}\!=\!(\mathbf{y}, \mathbf{z})_{\mathbf{g}, \mathbf{h}, \mathbf{i}, \mathfrak{j}, \mathfrak{k}}$.
If $m\in \mathbbm{l}_1\setminus \mathbb{U}$, then $\mathbf{w}_m\sim_{\mathbb{G}_m}\mathbf{v}_m=\mathbf{x}_m$. If $m\in\mathbbm{l}_2$, then $\mathbf{w}_m\not\sim_{\mathbb{G}_{m}}\mathbf{v}_m\!=\!\mathbf{x}_m$. If $m\!\in\!(\mathbbm{j}_{(3)}\!\setminus\!(\mathbbm{i}_2\!\cup\!\mathbbm{j}_{(2)}))\cup
(\mathbbm{k}_{(3)}\setminus(\mathbbm{g}_2\cup\mathbbm{k}_{(2)}))\cup((\mathbbm{j}_{(3)}\cap\mathbbm{k}_{(3)})\setminus
(\mathbbm{j}_{(2)}\cap\mathbbm{k}_{(2)}))$, then $\mathbf{w}_m\sim_{\mathbb{G}_m}\mathbf{v}_m\not\sim_{\mathbb{G}_m}\mathbf{x}_m$ or
$\mathbf{w}_m=\mathbf{v}_m\not\sim_{\mathbb{G}_m}\mathbf{x}_m$. If $m\in\mathbbm{h}_0$, then $\mathbf{w}_m=\mathbf{v}_m=\mathbf{x}_m$.
If $m\!\in\!((\mathbbm{g}_1\!\cap\!\mathbbm{k}_{(1)})\!\setminus\!\mathbbm{j}_{(1)})\!\cup\!((\mathbbm{g}_1\!\cap\!\mathbbm{h}_1)\!\setminus\!(\mathbbm{i}_1\!\cup\!\mathbbm{j}_{(1)}))
\!\cup\!((\mathbbm{h}_1\cap\mathbbm{i}_1)\!\setminus\!\mathbbm{k}_{(1)})$, then $\mathbf{w}_m\!\!=\!\!\mathbf{v}_m\!\sim_{\mathbb{G}_m}\!\mathbf{x}_m$. If
$m\in((\mathbbm{g}_2\cap\mathbbm{k}_{(3)})\setminus\mathbbm{j}_{(3)})\cup
((\mathbbm{g}_2\cap\mathbbm{h}_2)\setminus(\mathbbm{i}_2\cup\mathbbm{j}_{(3)}))
\cup((\mathbbm{h}_2\cap\mathbbm{i}_2)\setminus\mathbbm{k}_{(3)})$, then $\mathbf{w}_m=\mathbf{v}_m\not\sim_{\mathbb{G}_m}\!\mathbf{x}_m$.
The desired lemma follows from an application of Lemma \ref{L;Lemma3.22}.
\end{proof}
\begin{lem}\label{L;Lemma3.26}
Assume that $\mathbf{y}, \mathbf{z}\in\mathbb{X}$, $\mathbf{g}, \mathbf{h}, \mathbf{i}\in \mathbb{E}$, $\mathfrak{j}=(\mathbbm{j}_{(1)},\mathbbm{j}_{(2)}, \mathbbm{j}_{(3)})\in\mathbb{U}_{\mathbf{g}, \mathbf{h}}$, and $\mathfrak{k}\!\!=\!\!(\mathbbm{k}_{(1)}, \mathbbm{k}_{(2)}, \mathbbm{k}_{(3)})\!\in\!\mathbb{U}_{\mathbf{h}, \mathbf{i}}$. Assume that $R_\mathbf{l}\in\mathbb{U}_{\mathbf{g}, \mathbf{i}, (\mathbf{g}, \mathbf{h}, \mathbf{i}, \mathfrak{j}, \mathfrak{k})}$ and $\mathbf{y}\in\mathbf{x}R_\mathbf{g}\cap\mathbf{z}R_\mathbf{l}$. If there are $\mathbf{m}\!\in\! \mathbb{E}$ and $\mathbb{U}\!\!\subseteq\!\!(\mathbbm{j}_{(3)}\setminus(\mathbbm{i}_2\cup\mathbbm{j}_{(2)}))\cup
(\mathbbm{k}_{(3)}\setminus(\mathbbm{g}_2\cup\mathbbm{k}_{(2)}))\cup((\mathbbm{j}_{(3)}\cap\mathbbm{k}_{(3)})\setminus
(\mathbbm{j}_{(2)}\cap\mathbbm{k}_{(2)}))$ such that $\mathbf{w}\!\in\!(\mathbf{y},\mathbf{z})_{\mathbf{g}, \mathbf{h}, \mathbf{i},\mathfrak{j}, \mathfrak{k}}R_\mathbf{m}$,
$\mathbbm{m}_1\!=\!(\mathbbm{h}_1\setminus(\mathbbm{g}_1\cup\mathbbm{i}_1))\cup(\mathbbm{j}_{(1)}\setminus\mathbbm{i}_1)\cup
(\mathbbm{k}_{(1)}\setminus\mathbbm{g}_1)\cup(\mathbbm{j}_{(1)}\cap\mathbbm{k}_{(1)})\cup\mathbb{U}$, $\mathbbm{m}_2\!\!=\!\!(\mathbbm{h}_2\!\setminus\!(\mathbbm{g}_2\!\cup\!\mathbbm{i}_2))\!\cup\!(\mathbbm{j}_{(2)}\!\setminus\!\mathbbm{i}_2)\cup
(\mathbbm{k}_{(2)}\setminus\mathbbm{g}_2)\!\cup\!(\mathbbm{j}_{(2)}\cap\mathbbm{k}_{(2)})$, then $\mathbf{w}\!\in\!\mathbf{y}R_\mathbf{q}$ for some $R_\mathbf{q}\!\in\!\mathbb{U}_{\mathbf{g}, \mathbf{h}, \mathfrak{j}}$.
\end{lem}
\begin{proof}
Set $\mathbbm{r}=\{a: a\in [1,n], \mathbf{w}_a\!\!\sim_{\mathbb{G}_a}\!\!\mathbf{y}_a\}$ and $\mathbbm{s}\!=\!\{a: a\!\in\! [1,n], \mathbf{w}_a\!\!\not\sim_{\mathbb{G}_a}\!\!\mathbf{y}_a\}$. By Lemma \ref{L;Lemma3.1}, there is $\mathbf{q}\in \mathbb{E}$ such that $\mathbbm{q}_1=(\mathbbm{g}_0\cap\mathbbm{h}_1)\cup(\mathbbm{g}_1\cap\mathbbm{h}_0)\cup(\mathbbm{j}_{(1)}\cap\mathbbm{r})\cup(\mathbbm{j}_{(3)}\cap\mathbbm{r})$ and $\mathbbm{q}_2=(\mathbbm{g}_2\triangle\mathbbm{h}_2)\cup(\mathbbm{j}_{(2)}\cap\mathbbm{s})$. Then $R_\mathbf{q}\in\mathbb{U}_{\mathbf{g}, \mathbf{h}, \mathfrak{j}}$ by Lemma \ref{L;Lemma3.2}. Define $\mathbf{v}=(\mathbf{y}, \mathbf{z})_{\mathbf{g}, \mathbf{h}, \mathbf{i}, \mathfrak{j}, \mathfrak{k}}$.

If $t\!\in\!\mathbbm{g}_0\cap\mathbbm{h}_0$, then $\mathbf{w}_t\!=\!\mathbf{v}_t=\mathbf{y}_t$. If $t\!\in\! \mathbbm{h}_0\!\setminus\!\mathbbm{g}_0$, then $\mathbf{w}_t\!=\!\mathbf{v}_t\sim_{\mathbb{G}_t}\mathbf{y}_t$
or $\mathbf{w}_t\!=\!\mathbf{v}_t\not\sim_{\mathbb{G}_t}\mathbf{y}_t$. If $t\in(\mathbbm{h}_1\setminus(\mathbbm{g}_1\cup\mathbbm{i}_1))\cup(\mathbbm{k}_{(1)}\setminus\mathbbm{g}_1)$, then
$\mathbf{w}_t\sim_{\mathbbm{G}_t}\mathbf{v}_t=\mathbf{y}_t$ or $\mathbf{w}_t\sim_{\mathbbm{G}_t}\mathbf{v}_t\not\sim_{\mathbb{G}_t}\mathbf{y}_t$. If $t\in(\mathbbm{h}_2
\setminus(\mathbbm{g}_2\cup\mathbbm{i}_2))\cup(\mathbbm{k}_{(2)}\setminus\mathbbm{g}_2)$, then $\mathbf{w}_t
\not\sim_{\mathbb{G}_t}\mathbf{v}_t=\mathbf{y}_t$ or $\mathbf{w}_t
\not\sim_{\mathbb{G}_t}\mathbf{v}_t\sim_{\mathbb{G}_t}\mathbf{y}_t$. If $t\in(\mathbbm{j}_{(1)}\!\setminus\!\mathbbm{i}_1)\!\cup\!(\mathbbm{j}_{(1)}\!\cap\!\mathbbm{k}_{(1)})\!\cup\!(\mathbbm{j}_{(3)}\!\setminus\!(\mathbbm{i}_2
\cup\mathbbm{j}_{(2)}))\cup((\mathbbm{j}_{(3)}\cap\mathbbm{k}_{(3)})\setminus\mathbbm{j}_{(2)})$, then $\mathbf{w}_t=\mathbf{y}_t$ or $\mathbf{w}_t\sim_{\mathbb{G}_t}\mathbf{y}_t$. If $t\in\mathbbm{k}_{(3)}\setminus(\mathbbm{g}_2\cup\mathbbm{k}_{(2)})$, then $\mathbf{w}_t=\mathbf{v}_t\not\sim_{\mathbb{G}_t}\mathbf{y}_t$ or $\mathbf{w}_t\sim_{\mathbb{G}_t}\mathbf{v}_t\not\sim_{\mathbb{G}_t}\!\mathbf{y}_t$. If $t\in((\mathbbm{g}_1\!\cap\!\mathbbm{k}_{(1)})\!\setminus\!\mathbbm{j}_{(1)})\cup
((\mathbbm{g}_1\!\cap\!\mathbbm{h}_1)\!\setminus\!(\mathbbm{i}_1\!\cup\!\mathbbm{j}_{(1)}))\cup((\mathbbm{g}_2\!\cap\!
\mathbbm{k}_{(3)})\!\setminus\!\mathbbm{j}_{(3)})\cup((\mathbbm{g}_2\cap\mathbbm{h}_2)\!\setminus\!(\mathbbm{i}_2\!\cup\!\mathbbm{j}_{(3)}))$,
then $\mathbf{w}_t\!\!=\!\!\mathbf{v}_t\!=\!\mathbf{y}_t$. If $t\!\in\!((\mathbbm{g}_2\!\cap\!\mathbbm{h}_1\!\cap\!\mathbbm{i}_1)\!\setminus\!\mathbbm{k}_{(1)})\!\cup\!
((\mathbbm{h}_2\!\cap\!\mathbbm{i}_2)\!\setminus\!(\mathbbm{j}_{(3)}\!\cup\mathbbm{k}_{(3)}))$, then $\mathbf{w}_t\!=\!\mathbf{v}_t\!=\!\mathbf{y}_t$ or $\mathbf{w}_t\!=\!\mathbf{v}_t\!\not\sim_{\mathbb{G}_t}\!\mathbf{y}_t$.
If $t\!\!\in\!\!((\mathbbm{h}_1\!\cap\!\mathbbm{i}_1)\!\setminus\!(\mathbbm{g}_2\!\cup\!\mathbbm{k}_{(1)}))\!\cup\!
((\mathbbm{i}_2\!\cap\!\mathbbm{j}_{(3)})\!\setminus\!(\mathbbm{j}_{(2)}\!\cup\!\mathbbm{k}_{(3)}))$, then $\mathbf{w}_t\!\!=\!\!\mathbf{v}_t\!\!=\!\mathbf{y}_t$ or $\mathbf{w}_t\!\!=\!\!\mathbf{v}_t\!\!\sim_{\mathbb{G}_t}\!\!\mathbf{y}_t$.
The desired lemma follows from an application of Lemma \ref{L;Lemma3.22}.
\end{proof}
\begin{lem}\label{L;Lemma3.27}
Assume that $\mathbf{y}, \mathbf{z}\in\mathbb{X}$, $\mathbf{g}, \mathbf{h}, \mathbf{i}\in \mathbb{E}$, $\mathfrak{j}=(\mathbbm{j}_{(1)},\mathbbm{j}_{(2)}, \mathbbm{j}_{(3)})\in\mathbb{U}_{\mathbf{g}, \mathbf{h}}$, and $\mathfrak{k}\!=\!(\mathbbm{k}_{(1)}, \mathbbm{k}_{(2)}, \mathbbm{k}_{(3)})\in\mathbb{U}_{\mathbf{h}, \mathbf{i}}$. Assume that $R_\mathbf{l}\in\mathbb{U}_{\mathbf{g}, \mathbf{i}, (\mathbf{g}, \mathbf{h}, \mathbf{i}, \mathfrak{j}, \mathfrak{k})}$ and $\mathbf{z}\in\mathbf{x}R_\mathbf{i}\cap\mathbf{y}R_\mathbf{l}$. If there are $\mathbf{m}\!\in\! \mathbb{E}$ and $\mathbb{U}\!\subseteq\!\!(\mathbbm{j}_{(3)}\setminus(\mathbbm{i}_2\cup\mathbbm{j}_{(2)}))\cup
(\mathbbm{k}_{(3)}\setminus(\mathbbm{g}_2\cup\mathbbm{k}_{(2)}))\cup((\mathbbm{j}_{(3)}\cap\mathbbm{k}_{(3)})\setminus
(\mathbbm{j}_{(2)}\cap\mathbbm{k}_{(2)}))$ such that $\mathbf{w}\!\in\!(\mathbf{y},\mathbf{z})_{\mathbf{g}, \mathbf{h}, \mathbf{i},\mathfrak{j}, \mathfrak{k}}R_\mathbf{m}$, $\mathbbm{m}_1\!=\!(\mathbbm{h}_1\setminus(\mathbbm{g}_1\cup\mathbbm{i}_1))\cup(\mathbbm{j}_{(1)}\setminus\mathbbm{i}_1)\cup
(\mathbbm{k}_{(1)}\setminus\mathbbm{g}_1)\cup(\mathbbm{j}_{(1)}\cap\mathbbm{k}_{(1)})\cup\mathbb{U}$, $\mathbbm{m}_2\!\!=\!\!(\mathbbm{h}_2\!\setminus\!(\mathbbm{g}_2\!\cup\!\mathbbm{i}_2))\!\cup\!(\mathbbm{j}_{(2)}\!\setminus\!\mathbbm{i}_2)\cup
(\mathbbm{k}_{(2)}\setminus\mathbbm{g}_2)\!\cup\!(\mathbbm{j}_{(2)}\cap\mathbbm{k}_{(2)})$, then $\mathbf{w}\!\in\!\mathbf{z}R_\mathbf{q}$ for some $R_\mathbf{q}\in\mathbb{U}_{\mathbf{h}, \mathbf{i}, \mathfrak{k}}$.
\end{lem}
\begin{proof}
Set $\mathbbm{r}=\{a: a\in [1,n], \mathbf{w}_a\sim_{\mathbb{G}_a}\mathbf{z}_a\}$ and $\mathbbm{s}\!=\!\{a: a\!\in\! [1,n], \mathbf{w}_a\not\sim_{\mathbb{G}_a}\!\!\mathbf{z}_a\}$. By Lemma \ref{L;Lemma3.1}, there is $q\in \mathbb{E}$ such that
$\mathbbm{q}_1=(\mathbbm{h}_0\cap\mathbbm{i}_1)\cup(\mathbbm{h}_1\cap\mathbbm{i}_0)\cup(\mathbbm{k}_{(1)}\cap\mathbbm{r})\cup(\mathbbm{k}_{(3)}\cap\mathbbm{r})$ and $\mathbbm{q}_2=(\mathbbm{h}_2\triangle\mathbbm{i}_2)\cup(\mathbbm{k}_{(2)}\cap\mathbbm{s})$. Then $R_\mathbf{q}\in\mathbb{U}_{\mathbf{h}, \mathbf{i}, \mathfrak{k}}$ by Lemma \ref{L;Lemma3.2}. Define $\mathbf{v}=(\mathbf{y}, \mathbf{z})_{\mathbf{g}, \mathbf{h}, \mathbf{i}, \mathfrak{j}, \mathfrak{k}}$.

If $t\!\in\!\mathbbm{h}_0\cap\mathbbm{i}_0$, then $\mathbf{w}_t\!=\!\mathbf{v}_t\!=\!\mathbf{z}_t$. If $t\in\mathbbm{h}_0\!\setminus\!\mathbbm{i}_0$, then $\mathbf{w}_t=\mathbf{v}_t\sim_{\mathbb{G}_t}\mathbf{z}_t$ or $\mathbf{w}_t=\mathbf{v}_t\not\sim_{\mathbb{G}_t}\mathbf{z}_t$. If $t\in(\mathbbm{h}_1\setminus(\mathbbm{g}_1\cup\mathbbm{i}_1))\cup(\mathbbm{j}_{(1)}\setminus\mathbbm{i}_1)$, then
$\mathbf{w}_t\sim_{\mathbb{G}_t}\mathbf{v}_t=\mathbf{z}_t$ or $\mathbf{w}_t\sim_{\mathbb{G}_t}\mathbf{v}_t\not
\sim_{\mathbb{G}_t}\mathbf{z}_t$. If $t\in(\mathbbm{h}_2
\setminus(\mathbbm{g}_2\cup\mathbbm{i}_2))\cup(\mathbbm{j}_{(2)}\setminus\mathbbm{i}_2)$, then $\mathbf{w}_t\not\sim_{\mathbb{G}_t}\mathbf{v}_t=\mathbf{z}_t$
or $\mathbf{w}_t\!\not\sim_{\mathbb{G}_t}\mathbf{v}_t\sim_{\mathbb{G}_t}\mathbf{z}_t$. If $t\!\in\!(\mathbbm{k}_{(1)}\!\setminus\!\mathbbm{g}_1)\!\cup\!(\mathbbm{j}_{(1)}\!\cap\!\mathbbm{k}_{(1)})\cup(\mathbbm{k}_{(3)}\!\setminus\!
(\mathbbm{g}_2\cup\mathbbm{k}_{(2)}))\cup((\mathbbm{j}_{(3)}\cap\mathbbm{k}_{(3)})\setminus\mathbbm{k}_{(2)})$, then $\mathbf{w}_t=\mathbf{z}_t$ or $\mathbf{w}_t\sim_{\mathbb{G}_t}\mathbf{z}_t$. If $t\in\mathbbm{j}_{(3)}\setminus(\mathbbm{i}_2\cup\mathbbm{j}_{(2)})$, then $\mathbf{w}_t=\mathbf{v}_t\not\sim_{\mathbb{G}_t}\mathbf{z}_t$ or $\mathbf{w}_t\sim_{\mathbb{G}_t}\mathbf{v}_t\not\sim_{\mathbb{G}_t}\mathbf{z}_t$. If $t\in((\mathbbm{g}_1\cap\mathbbm{k}_{(1)})\setminus\mathbbm{j}_{(1)})\cup((\mathbbm{g}_2\cap\mathbbm{k}_{(3)})\setminus(\mathbbm{j}_{(3)}\cup \mathbbm{k}_{(2)}))\cup((\mathbbm{h}_1\cap\mathbbm{i}_1)\setminus\mathbbm{k}_{(1)})\cup((\mathbbm{h}_2\cap\mathbbm{i}_2)\setminus\mathbbm{k}_{(3)})$, then $\mathbf{w}_t=\mathbf{z}_t$ or $\mathbf{w}_t\sim_{\mathbb{G}_t}\mathbf{z}_t$. If
$t\in((\mathbbm{g}_1\cap\mathbbm{h}_1)\setminus(\mathbbm{i}_1\cup\mathbbm{j}_{(1)}))\cup
((\mathbbm{g}_2\cap\mathbbm{h}_2)\setminus(\mathbbm{i}_2\cup\mathbbm{j}_{(3)}))$, then $\mathbf{w}_t=\mathbf{v}_t\sim_{\mathbb{G}_t}\!\mathbf{z}_t$ or $\mathbf{w}_t=\mathbf{v}_t\not\sim_{\mathbb{G}_t}\!\mathbf{z}_t$.
The desired lemma follows from Lemma \ref{L;Lemma3.22}.
\end{proof}
We are now ready to display the structure constants of $\mathbb{B}_2$ occurred in $\mathbb{T}$ explicitly.
\begin{thm}\label{T;Theorem3.28}
Assume that $\mathbf{g}, \mathbf{h}, \mathbf{i}\!\in\!\mathbb{E}$, $\mathfrak{j}\in\mathbb{U}_{\mathbf{g}, \mathbf{h}}$, and $\mathfrak{k}\!\in\!\mathbb{U}_{\mathbf{l}, \mathbf{i}}$ for some $\mathbf{l}\in \mathbb{E}$. Then
$$B_{\mathbf{g}, \mathbf{h}, \mathfrak{j}}B_{\mathbf{l}, \mathbf{i}, \mathfrak{k}}=\delta_{\mathbf{h}, \mathbf{l}}\overline{k_{[\mathbf{g}, \mathbf{h}, \mathbf{i}]}}\overline{k_{\mathfrak{j}\setminus \mathbf{i}}}\overline{k_{\mathfrak{k}\setminus \mathbf{g}}}\overline{k_{\mathfrak{j}\cap\mathfrak{k}}}B_{\mathbf{g}, \mathbf{i}, (\mathbf{g}, \mathbf{h}, \mathbf{i}, \mathfrak{j}, \mathfrak{k})}.$$
\end{thm}
\begin{proof}
According to Equation \eqref{Eq;3}, there is no loss to assume further that $\mathbf{h}=\mathbf{l}$. By combining Lemmas \ref{L;Lemma3.5}, \ref{L;Lemma3.24}, \ref{L;Lemma3.25}, \ref{L;Lemma3.26}, \ref{L;Lemma3.27}, \ref{L;Lemma3.22}, \ref{L;Lemma3.21}, and \ref{L;Lemma3.20}, notice that
$$c_{\mathbf{g}, \mathbf{m}, \mathbf{i}}(B_{\mathbf{g}, \mathbf{h}, \mathfrak{j}}B_{\mathbf{h}, \mathbf{i}, \mathfrak{k}})=\sum_{R_\mathbf{q}\in \mathbb{U}_{\mathbf{g}, \mathbf{h}, \mathfrak{j}}}\sum_{R_\mathbf{r}\in \mathbb{U}_{\mathbf{h}, \mathbf{i}, \mathfrak{k}}}\overline{|\mathbf{y}R_\mathbf{q}\cap\mathbf{x}R_\mathbf{h}\cap\mathbf{z}R_\mathbf{r}|}=\overline{k_{[\mathbf{g}, \mathbf{h}, \mathbf{i}]}k_{\mathfrak{j}\setminus \mathbf{i}}k_{\mathfrak{k}\setminus \mathbf{g}}k_{\mathfrak{j}\cap\mathfrak{k}}}$$
for any given $R_\mathbf{m}\in\mathbb{U}_{\mathbf{g}, \mathbf{i}, (\mathbf{g}, \mathbf{h}, \mathbf{i}, \mathfrak{j}, \mathfrak{k})}$, $\mathbf{y}, \mathbf{z}\in\mathbb{X}$, and $\mathbf{z}\in\mathbf{y}R_{\mathbf{m}}$. The desired theorem thus follows from combining Lemmas \ref{L;Lemma3.4}, \ref{L;Lemma3.17}, and the above presented equation.
\end{proof}
We conclude this section by presenting an example of Theorems \ref{T;Theorem3.13} and \ref{T;Theorem3.28}.
\begin{eg}\label{E;Example3.29}
Assume that $\mathfrak{g}\!=\!(\varnothing, \varnothing, \{1\})$, $\mathfrak{h}=(\varnothing, \{1\}, \{1\})$, and $\mathfrak{i}=(\{1\}, \varnothing, \varnothing)$. Assume that $n=1$ and $j$ denotes the 1-tuple $(j)$ for any $j\in [0,2]$. Let $\mathbb{U}$ be the set containing precisely $B_{0,0,\mathfrak{o}}$, $B_{0,1,\mathfrak{o}}$, $B_{0,2,\mathfrak{o}}$, $B_{1,0,\mathfrak{o}}$, $B_{1,1,\mathfrak{o}}$, $B_{1,2,\mathfrak{o}}$, $B_{2,0,\mathfrak{o}}$, $B_{2,1,\mathfrak{o}}$, $B_{2,2,\mathfrak{o}}$, and $B_{2,2,\mathfrak{g}}$. By Theorem \ref{T;Theorem3.13} and a direct computation, it is clear that $\mathbb{T}$ has an $\F$-basis
\[\begin{cases}
\mathbb{U}, &\text{if $\ell_1=m_1=2$},\\
\mathbb{U}\cup\{B_{2,2,\mathfrak{h}}\}, &\text{if $\ell_1>m_1=2$},\\
\mathbb{U}\cup\{B_{1,1,\mathfrak{i}}\}, &\text{if $m_1>\ell_1=2$},\\
\mathbb{U}\cup\{B_{1,1,\mathfrak{h}}, B_{2,2,\mathfrak{i}}\}, &\text{if $\min\{\ell_1, m_1\}>2$}.
\end{cases}\]

Assume further that $n=\ell_1=2$, $\ell_2=m_1=m_2=3$, $\mathfrak{g}\!=\!(\varnothing, \varnothing, \{2\})$, and $\mathfrak{h}=(\varnothing, \{2\}, \{2\})$. Assume that $\mathbf{k}=(0,2)$ and $\mathbf{l}=(1,2)$. As $\mathbb{U}_{\mathbf{k},\mathbf{l}}=\mathbb{U}_{\mathbf{l},\mathbf{k}}=\{\mathfrak{o}, \mathfrak{g}, \mathfrak{h}\}$, notice that
both $B_{\mathbf{k}, \mathbf{l}, \mathfrak{g}}$ and $B_{\mathbf{l}, \mathbf{k}, \mathfrak{h}}$ are defined. By Theorem
\ref{T;Theorem3.28}, $B_{\mathbf{k}, \mathbf{l}, \mathfrak{g}}B_{\mathbf{l}, \mathbf{k}, \mathfrak{h}}=\overline{6}B_{\mathbf{k}, \mathbf{k}, \mathfrak{h}}$.
\end{eg}
\section{Algebraic structure of $\mathbb{T}$: Center}
In this section, we present an $\F$-basis of $\mathrm{Z}(\mathbb{T})$ and list a property of this $\F$-basis. We first recall Notations \ref{N;Notation3.7}, \ref{N;Notation3.8}, \ref{N;Notation3.14}, \ref{N;Notation3.15}, \ref{N;Notation3.18} and display an additional notation.
\begin{nota}\label{N;Notation4.1}
Assume that $\mathbf{g}\in \mathbb{E}$ and $\mathfrak{h}\in\mathbb{U}_{\mathbf{g}, \mathbf{g}}$. Define $C_\mathfrak{h}=\sum_{\mathbf{i}\in\mathbb{E}}\overline{k_{\mathfrak{h}\setminus \mathbf{i}}}B_{\mathbf{i}, \mathbf{i}, \mathbf{i}\cap\mathfrak{h}}$. As $k_{\mathfrak{h}\setminus \mathbf{g}}=1$, Theorem \ref{T;Theorem3.13} implies that $C_\mathfrak{h}$ is a defined nonzero matrix in $\mathbb{T}$. As $\mathbf{i}\cap\mathfrak{o}\!=\!\mathfrak{o}$ and $k_{\mathfrak{o}\setminus\mathbf{i}}=1$ for any $\mathbf{i}\in\mathbb{E}$, Lemma \ref{L;Lemma3.9} and Equation (\ref{Eq;4}) thus imply that $C_\mathfrak{o}=I$.
\end{nota}
The following three lemmas give some properties of the objects in Notation \ref{N;Notation4.1}.
\begin{lem}\label{L;Lemma4.2}
Assume that $\mathbf{g}, \mathbf{h}\in \mathbb{E}$, $\mathfrak{i}\in\mathbb{U}_{\mathbf{g},\mathbf{g}}$, and $\mathfrak{j}\in \mathbb{U}_{\mathbf{h}, \mathbf{h}}$. Then $c_{\mathbf{g}, \mathbf{g}, \mathfrak{i}}(C_\mathfrak{i})=\overline{1}$. In particular, $C_\mathfrak{i}=C_\mathfrak{j}$ if and only if $\mathfrak{i}=\mathfrak{j}$.
\end{lem}
\begin{proof}
As $k_{\mathfrak{i}\setminus \mathbf{g}}\!=\!1$, the first statement follows from Theorem \ref{T;Theorem3.13}. One direction in the second statement is obvious. For the other direction, assume that $C_\mathfrak{i}=C_\mathfrak{j}$. The first statement and Theorem \ref{T;Theorem3.13} imply that $B_{\mathbf{g}, \mathbf{g},\mathfrak{i}}\in\mathrm{Supp}_{\mathbb{B}_2}(C_\mathfrak{i})=\mathrm{Supp}_{\mathbb{B}_2}(C_\mathfrak{j})$. So Theorem \ref{T;Theorem3.13} yields $\mathfrak{i}=\mathbf{g}\cap\mathfrak{j}\preceq\mathfrak{j}$. Similarly, $\mathfrak{j}\preceq\mathfrak{i}$. The desired lemma follows.
\end{proof}
\begin{lem}\label{L;Lemma4.3}
$\mathbb{T}$ has an $\F$-linearly independent subset $\{C_\mathfrak{a}:\exists\ \mathbf{b}\in \mathbb{E}, \mathfrak{a}\in \mathbb{U}_{\mathbf{b},\mathbf{b}}\}$.
\end{lem}
\begin{proof}
Let $\mathbb{U}\!=\!\{C_\mathfrak{a}: \exists\ \mathbf{b}\in \mathbb{E}, \mathfrak{a}\in \mathbb{U}_{\mathbf{b},\mathbf{b}}\}$. Let $L$ be a nonzero $\F$-linear combination of the matrices in $\mathbb{U}$. Assume that $L=O$. If $C_\mathfrak{g}\in\mathbb{U}$, let $c_\mathfrak{g}$ be the coefficient of $C_\mathfrak{g}$ in $L$. There is $C_\mathfrak{h}\in\mathbb{U}$ such that $c_\mathfrak{h}\in\mathbb{F}^{\times}$. Therefore $\mathbb{V}\!=\!\{C_\mathfrak{a}: C_\mathfrak{a}\in \mathbb{U}, c_\mathfrak{a}\in\mathbb{F}^{\times}\}\neq\varnothing$.
According to Lemma \ref{L;Lemma4.2}, there must exist $i\in\mathbb{N}$ and pairwise distinct $\mathfrak{h}_1, \mathfrak{h}_2, \ldots, \mathfrak{h}_i$ such that $\mathbb{V}=\{C_{\mathfrak{h}_1}, C_{\mathfrak{h}_2}, \ldots, C_{\mathfrak{h}_i}\}$. Let us distinguish the cases $i=1$ and $i>1$.

If $i=1$, then $O=L=c_{\mathfrak{h}_1}C_{\mathfrak{h}_1}\neq O$, which is a contradiction. Assume that $i>1$.
There is no loss to assume further that $\mathfrak{h}_1$ is maximal in $\{\mathfrak{h}_1, \mathfrak{h}_2, \ldots, \mathfrak{h}_i\}$ with respect to the partial order $\preceq$. As $L=O$, notice that $C_{\mathfrak{h}_1}$ is an $\F$-linear combination of the matrices in $\{C_{\mathfrak{h}_2}, C_{\mathfrak{h}_3}, \ldots, C_{\mathfrak{h}_i}\}$. According to Lemma \ref{L;Lemma4.2} and Theorem \ref{T;Theorem3.13}, notice that $\mathfrak{h}_1=\mathbf{j}\cap\mathfrak{h}_k\preceq\mathfrak{h}_k$ for some $\mathbf{j}\in \mathbb{E}$ and $k\in [2, i]$. Then $\mathfrak{h}_1=\mathfrak{h}_k$ by the choice of $\mathfrak{h}_1$. This is also a contradiction. Therefore $L\neq O$. The desired lemma follows.
\end{proof}
\begin{lem}\label{L;Lemma4.4}
The $\F$-dimension of $\langle\{C_\mathfrak{a}:\exists\  \mathbf{b}\in \mathbb{E}, \mathfrak{a}\in \mathbb{U}_{\mathbf{b}, \mathbf{b}}\}\rangle_\mathbb{T}$ equals $2^{n_1+2n_4}3^{n_2+n_3}$.
\end{lem}
\begin{proof}
Recall that $n_1\!=\!|\{a: a\!\in\! [1,n], \ell_a\!=\!m_a\!=\!2\}|$, $n_2=|\{a: a\!\in\! [1,n], \ell_a\!\!>\!m_a\!=\!2\}|$,
$n_3=|\{a: a\in[1,n], m_a>\ell_a=2\}|$, and $n_4=|\{a: a\!\in\! [1,n], \min\{\ell_a, m_a\}\!>\!2\}|$. Notice that the sum of $n_1$, $n_2$, $n_3$, and $n_4$ is equal to $n$.
According to Lemma \ref{L;Lemma4.2}, notice that $|\{C_{\mathfrak{a}}: \exists\ \mathbf{b}\!\in\! \mathbb{E}, \mathfrak{a}\!\in\!\mathbb{U}_{\mathbf{b}, \mathbf{b}}\}|\!=\!|\{\mathfrak{a}:\exists\  \mathbf{b}\!\in\! \mathbb{E}, \mathfrak{a}\!\in\!\mathbb{U}_{\mathbf{b}, \mathbf{b}}\}|$. Moreover, notice that
$$|\{\mathfrak{a}:\exists\  \mathbf{b}\in \mathbb{E}, \mathfrak{a}\in \mathbb{U}_{\mathbf{b}, \mathbf{b}}\}|=\sum_{g=0}^{n_3}\sum_{h=0}^{n_4}\sum_{i=0}^{n_2+n_4-h}{n_3\choose g}{n_4\choose h}{n_2+n_4-h\choose i}2^{n-g-h-i}.$$
The desired lemma follows from the above presented equation and Lemma \ref{L;Lemma4.3}.
\end{proof}
Our main goal is to find an $\F$-basis for $\mathrm{Z}(\mathbb{T})$. We begin with a sequence of lemmas.
\begin{lem}\label{L;Lemma4.5}
Assume that $\mathbf{g}, \mathbf{h}\!\!\in\!\!\mathbb{E}$ and $\mathfrak{i}\!\!\in\!\!\mathbb{U}_{\mathbf{g}, \mathbf{h}}$. Then $B_{\mathbf{g}, \mathbf{h}, \mathfrak{o}}B_{\mathbf{h}, \mathbf{h}, \mathfrak{i}}\!\!=\!\!B_{\mathbf{g}, \mathbf{h}, \mathfrak{i}}\!=\!B_{\mathbf{g}, \mathbf{g}, \mathfrak{i}}B_{\mathbf{g}, \mathbf{h}, \mathfrak{o}}$.
\end{lem}
\begin{proof}
As $k_{[\mathbf{g},\mathbf{h},\mathbf{h}]}=k_{\mathfrak{o}\setminus \mathbf{h}}=k_{\mathfrak{i}\setminus \mathbf{g}}=k_{\mathfrak{o}\cap\mathfrak{i}}=1$, the leftmost equation is from Theorem \ref{T;Theorem3.28}. The desired lemma thus follows from the leftmost equation and Lemma \ref{L;Lemma3.9}.
\end{proof}
\begin{lem}\label{L;Lemma4.6}
Assume that $\mathbf{g}\in \mathbb{E}$ and $\mathfrak{h}, \mathfrak{i}\in\mathbb{U}_{\mathbf{g}, \mathbf{g}}$. Then $B_{\mathbf{g}, \mathbf{g}, \mathfrak{h}}B_{\mathbf{g}, \mathbf{g}, \mathfrak{i}}=\overline{k_{\mathfrak{h}\cap\mathfrak{i}}}B_{\mathbf{g}, \mathbf{g}, \mathfrak{h}\cup\mathfrak{i}}$ and the commutative $\F$-subalgebra $E_\mathbf{g}^*\mathbb{T}E_\mathbf{g}^*$ of $\mathbb{T}$ has an $\F$-basis $\{B_{\mathbf{g}, \mathbf{g}, \mathfrak{a}}: \mathfrak{a}\in\mathbb{U}_{\mathbf{g}, \mathbf{g}}\}$.
\end{lem}
\begin{proof}
Notice that $k_{[\mathbf{g},\mathbf{g},\mathbf{g}]}=k_{\mathfrak{h}\setminus \mathbf{g}}=k_{\mathfrak{i}\setminus \mathbf{g}}=1$. Therefore the desired lemma follows from combining Theorems \ref{T;Theorem3.28}, \ref{T;Theorem3.13}, Equation \eqref{Eq;3}, Lemmas \ref{L;Lemma2.3}, and \ref{L;Lemma2.9}.
\end{proof}
\begin{lem}\label{L;Lemma4.7}
Assume that $\mathbf{g}, \mathbf{h}, \mathbf{i}\in \mathbb{E}$ and $\mathfrak{j}\in \mathbb{U}_{\mathbf{i}, \mathbf{i}}$. Then $k_{\mathfrak{j}\setminus \mathbf{g}}k_{(\mathbf{g}\cap\mathfrak{j})\setminus \mathbf{h}}=k_{\mathfrak{j}\setminus \mathbf{h}}k_{(\mathbf{h}\cap\mathfrak{j})\setminus \mathbf{g}}$ and $(\mathbf{g}, \mathbf{g}, \mathbf{h}, \mathbf{g}\cap\mathfrak{j}, \mathfrak{o})=(\mathbf{g}, \mathbf{h}, \mathbf{h}, \mathfrak{o}, \mathbf{h}\cap\mathfrak{j})$.
\end{lem}
\begin{proof}
Notice that $k_{\mathfrak{j}\setminus \mathbf{g}}k_{(\mathbf{g}\cap\mathfrak{j})\setminus \mathbf{h}}k_{\mathbf{h}\cap(\mathbf{g}\cap\mathfrak{j})}\!=\!k_\mathfrak{j}\!=\!k_{\mathfrak{j}\setminus \mathbf{h}}k_{(\mathbf{h}\cap\mathfrak{j})\setminus \mathbf{g}}k_{\mathbf{g}\cap(\mathbf{h}\cap\mathfrak{j})}$ by Lemma \ref{L;Lemma3.20}. As $k_{\mathbf{h}\cap(\mathbf{g}\cap\mathfrak{j})}\!=\!k_{\mathbf{g}\cap(\mathbf{h}\cap\mathfrak{j})}$, the first statement follows. Assume that $\mathfrak{j}\!=\!(\mathbbm{j}_{(1)}, \mathbbm{j}_{(2)}, \mathbbm{j}_{(3)})$. Then
$(\mathbf{g}, \mathbf{g}, \mathbf{h}, \mathbf{g}\cap\mathfrak{j}, \mathfrak{o})\!=\!(\mathbf{g}, \mathbf{h}, \mathbf{h}, \mathfrak{o}, \mathbf{h}\cap\mathfrak{j})=(\mathbbm{g}_1\cap\mathbbm{h}_1\cap\mathbbm{j}_{(1)}, \mathbbm{g}_2\cap\mathbbm{h}_2\cap\mathbbm{j}_{(2)}, \mathbbm{g}_2\cap\mathbbm{h}_2\cap\mathbbm{j}_{(3)})$ by a direct computation. The desired lemma follows from the above discussion.
\end{proof}
\begin{lem}\label{L;Lemma4.8}
$\mathrm{Z}(\mathbb{T})$ has an $\F$-linearly independent subset $\{C_\mathfrak{a}\!: \exists\ \mathbf{b}\in \mathbb{E}, \mathfrak{a}\!\in\! \mathbb{U}_{\mathbf{b} ,\mathbf{b}}\}$.
\end{lem}
\begin{proof}
Pick $\mathbf{g}, \mathbf{h}, \mathbf{i}\in \mathbb{E}$ and $\mathfrak{j}\in \mathbb{U}_{\mathbf{i}, \mathbf{i}}$. So $C_\mathfrak{j}B_{\mathbf{g}, \mathbf{h}, \mathfrak{o}}=B_{\mathbf{g}, \mathbf{h}, \mathfrak{o}}C_\mathfrak{j}$ by combining Equation \eqref{Eq;3}, Lemma \ref{L;Lemma4.7}, and Theorem \ref{T;Theorem3.28}. Pick $\mathfrak{k}\in \mathbb{U}_{\mathbf{g}, \mathbf{h}}$. Then $C_\mathfrak{j}B_{\mathbf{g}, \mathbf{h}, \mathfrak{k}}=B_{\mathbf{g}, \mathbf{h}, \mathfrak{k}}C_\mathfrak{j}$ by combining Equation \eqref{Eq;3}, Lemmas \ref{L;Lemma4.5}, and \ref{L;Lemma4.6}. As $\mathbf{g}, \mathbf{h}, \mathbf{i}, \mathfrak{j}, \mathfrak{k}$ are chosen arbitrarily, the desired lemma follows from the above discussion, Theorem \ref{T;Theorem3.13}, Lemma \ref{L;Lemma4.3}.
\end{proof}
\begin{lem}\label{L;Lemma4.9}
Assume that $M\!\in\!\mathrm{Z}(\mathbb{T})$. Then $\mathrm{Supp}_{\mathbb{B}_2}(M)\subseteq\{B_{\mathbf{a}, \mathbf{a}, \mathfrak{b}}: \mathbf{a}\in\mathbb{E}, \mathfrak{b}\in \mathbb{U}_{\mathbf{a}, \mathbf{a}}\}$.
\end{lem}
\begin{proof}
As $M\in\mathrm{Z}(\mathbb{T})$, Theorem \ref{T;Theorem3.13} implies that $M$ is an $\F$-linear combination of the matrices in $\mathbb{B}_2$. Assume that $\mathrm{Supp}_{\mathbb{B}_2}(M)\not\subseteq\{B_{\mathbf{a}, \mathbf{a}, \mathfrak{b}}: \mathbf{a}\in \mathbb{E}, \mathfrak{b}\in\mathbb{U}_{\mathbf{a}, \mathbf{a}}\}$. Following Equation \eqref{Eq;4}, there exist distinct $\mathbf{g}, \mathbf{h}\in \mathbb{E}$ such that $E_\mathbf{g}^*ME_\mathbf{h}^*\neq O$. As $M\in\mathrm{Z}(\mathbb{T})$ and $E_\mathbf{g}^*, E_\mathbf{h}^*\in \mathbb{T}$, notice that $O\neq E_\mathbf{g}^*ME_\mathbf{h}^*=ME_\mathbf{g}^*E_\mathbf{h}^*=O$ by Equation \eqref{Eq;3}. This is an obvious contradiction. The desired lemma follows from this contradiction.
\end{proof}
\begin{lem}\label{L;Lemma4.10}
Assume that $\mathbf{g}, \mathbf{h}\in \mathbb{E}$ and $M\in\mathrm{Z}(\mathbb{T})$. Assume that $\mathfrak{i}\in\mathbb{U}_{\mathbf{g}, \mathbf{g}}\cap\mathbb{U}_{\mathbf{h}, \mathbf{h}}$ and $\mathfrak{i}$ is maximal in $\{\mathfrak{a}: \exists\ \mathbf{b}\in \mathbb{E}, \mathfrak{a}\in \mathbb{U}_{\mathbf{b}, \mathbf{b}}, B_{\mathbf{b}, \mathbf{b}, \mathfrak{a}}\in\mathrm{Supp}_{\mathbb{B}_2}(M)\}$ with respect to the partial order $\preceq$. Then $c_{\mathbf{g}, \mathbf{g}, \mathfrak{i}}(M)=c_{\mathbf{h}, \mathbf{h}, \mathfrak{i}}(M)$.
\end{lem}
\begin{proof}
As $M\in\mathrm{Z}(\mathbb{T})$, Lemma \ref{L;Lemma4.9} thus implies that $M$ is an $\F$-linear combination of the matrices in $\{B_{\mathbf{a}, \mathbf{a}, \mathfrak{b}}: \mathbf{a}\in\mathbb{E}, \mathfrak{b}\in\mathbb{U}_{\mathbf{a}, \mathbf{a}}\}$. According to Equation \eqref{Eq;3}, notice that
\begin{equation}\label{Eq;9}
\begin{aligned}
& ME_\mathbf{g}^*=c_{\mathbf{g}, \mathbf{g}, \mathfrak{i}}(M)B_{\mathbf{g}, \mathbf{g}, \mathfrak{i}}+\sum_{\mathfrak{j}\in\mathbb{U}_{\mathbf{g}, \mathbf{g}}\setminus{\{\mathfrak{i}}\}}c_{\mathbf{g}, \mathbf{g}, \mathfrak{j}}(M)B_{\mathbf{g}, \mathbf{g}, \mathfrak{j}}\ \text{and}\\
& E_\mathbf{h}^*M=c_{\mathbf{h}, \mathbf{h}, \mathfrak{i}}(M)B_{\mathbf{h}, \mathbf{h}, \mathfrak{i}}+\sum_{\mathfrak{j}\in\mathbb{U}_{\mathbf{h}, \mathbf{h}}\setminus{\{\mathfrak{i}}\}}c_{\mathbf{h}, \mathbf{h}, \mathfrak{j}}(M)B_{\mathbf{h}, \mathbf{h}, \mathfrak{j}}.
\end{aligned}
\end{equation}

By combining Equations \eqref{Eq;3}, \eqref{Eq;9}, Lemma \ref{L;Lemma4.5}, and Theorem \ref{T;Theorem3.28}, notice that
\begin{equation}\label{Eq;10}
\begin{aligned}
& MB_{\mathbf{g}, \mathbf{h}, \mathfrak{o}}=ME_\mathbf{g}^*B_{\mathbf{g}, \mathbf{h}, \mathfrak{o}}=c_{\mathbf{g}, \mathbf{g}, \mathfrak{i}}(M)B_{\mathbf{g}, \mathbf{h}, \mathfrak{i}}+\sum_{\mathfrak{j}\in\mathbb{U}_{\mathbf{g}, \mathbf{g}}\setminus{\{\mathfrak{i}}\}}c_{\mathbf{g}, \mathbf{g}, \mathfrak{j}}(M)\overline{k_{\mathfrak{j}\setminus \mathbf{h}}}B_{\mathbf{g}, \mathbf{h}, \mathbf{h}\cap\mathfrak{j}}\ \text{and}\\
& B_{\mathbf{g}, \mathbf{h}, \mathfrak{o}}M=B_{\mathbf{g}, \mathbf{h}, \mathfrak{o}}E_\mathbf{h}^*M=c_{\mathbf{h}, \mathbf{h}, \mathfrak{i}}(M)B_{\mathbf{g}, \mathbf{h}, \mathfrak{i}}+\sum_{\mathfrak{j}\in\mathbb{U}_{\mathbf{h}, \mathbf{h}}\setminus{\{\mathfrak{i}}\}}c_{\mathbf{h}, \mathbf{h}, \mathfrak{j}}(M)\overline{k_{\mathfrak{j}\setminus \mathbf{g}}}B_{\mathbf{g}, \mathbf{h}, \mathbf{g}\cap\mathfrak{j}}.
\end{aligned}
\end{equation}
If $c_{\mathbf{g}, \mathbf{g}, \mathfrak{j}}(M)\in\F^{\times}$ and $\mathfrak{i}=\mathbf{h}\cap\mathfrak{j}\preceq\mathfrak{j}$ for some $\mathfrak{j}\in\mathbb{U}_{\mathbf{g}, \mathbf{g}}\setminus\{\mathfrak{i}\}$, then $\mathfrak{i}=\mathfrak{j}$ by the choice of $\mathfrak{i}$. This is absurd. So $c_{\mathbf{g}, \mathbf{h}, \mathfrak{i}}(MB_{\mathbf{g}, \mathbf{h}, \mathfrak{o}})=c_{\mathbf{g}, \mathbf{g}, \mathfrak{i}}(M)$ by Equation \eqref{Eq;10} and Theorem \ref{T;Theorem3.13}. Similarly, $c_{\mathbf{g}, \mathbf{h}, \mathfrak{i}}(B_{\mathbf{g}, \mathbf{h}, \mathfrak{o}}M)=c_{\mathbf{h}, \mathbf{h}, \mathfrak{i}}(M)$. The desired lemma follows as $M\in\mathrm{Z}(\mathbb{T})$.
\end{proof}
\begin{lem}\label{L;Lemma4.11}
Assume that $M\in\mathrm{Z}(\mathbb{T})$. Then $M\in \langle\{C_\mathfrak{a}:\exists\  \mathbf{b}\in \mathbb{E}, \mathfrak{a}\in\mathbb{U}_{\mathbf{b}, \mathbf{b}}\}\rangle_\mathbb{T}$.
\end{lem}
\begin{proof}
If $L\in\mathrm{Z}(\mathbb{T})$, Lemma \ref{L;Lemma4.9} implies that $\mathrm{Supp}_{\mathbb{B}_2}(L)\subseteq\{B_{\mathbf{a}, \mathbf{a}, \mathfrak{b}}: \mathbf{a}\in \mathbb{E}, \mathfrak{b}\in\mathbb{U}_{\mathbf{a}, \mathbf{a}}\}$.
If $L\in\mathrm{Z}(\mathbb{T})\setminus\{O\}$, set $\max(L)=\max\{|\mathfrak{a}|: \exists\ \mathbf{b}\in\mathbb{E}, \mathfrak{a}\in\mathbb{U}_{\mathbf{b}, \mathbf{b}}, B_{\mathbf{b}, \mathbf{b}, \mathfrak{a}}\in\mathrm{Supp}_{\mathbb{B}_2}(L)\}$.
There is no loss to assume further that $M\neq O$. Work by induction on $\max(M)$. If $\max(M)=0$, notice that $M\in \langle\{I\}\rangle_\mathbb{T}$ by combining Lemmas \ref{L;Lemma4.9}, \ref{L;Lemma3.9}, and Equation \eqref{Eq;4}. This containment gives $M\!\in\!\langle \{C_\mathfrak{o}\}\rangle_\mathbb{T}$ as $C_\mathfrak{o}=I$. The base case is thus checked.

Assume that $\max(M)>0$ and $L\!\in\!\langle\{C_\mathfrak{a}: \exists\ \mathbf{b}\in \mathbb{E}, \mathfrak{a}\in\mathbb{U}_{\mathbf{b}, \mathbf{b}}\}\rangle_\mathbb{T}$ for any $L$ in $\mathrm{Z}(\mathbb{T})$ satisfying $\max(L)<\max(M)$. Assume that $g\in\mathbb{N}$ and $\mathfrak{h}_1, \mathfrak{h}_2, \ldots, \mathfrak{h}_g$ are all pairwise distinct elements in $\{\mathfrak{a}:\exists\  \mathbf{b}\in\mathbb{E}, \mathfrak{a}\!\in\!\mathbb{U}_{\mathbf{b}, \mathbf{b}}, |\mathfrak{a}|=\max(M), B_{\mathbf{b}, \mathbf{b}, \mathfrak{a}}\!\in\!\mathrm{Supp}_{\mathbb{B}_2}(M)\}$. So $\mathfrak{h}_1, \mathfrak{h}_2, \ldots, \mathfrak{h}_g$ are maximal in $\{\mathfrak{a}: \exists\ \mathbf{b}\in \mathbb{E}, \mathfrak{a}\!\in\!\mathbb{U}_{\mathbf{b}, \mathbf{b}}, B_{\mathbf{b}, \mathbf{b}, \mathfrak{a}}\!\in\!\mathrm{Supp}_{\mathbb{B}_2}(M)\}$ with respect to the partial order $\preceq$. Write $B_{\mathbf{i}_{\mathfrak{h}_1}, \mathbf{i}_{\mathfrak{h}_1}, \mathfrak{h}_1}, B_{\mathbf{i}_{\mathfrak{h}_2}, \mathbf{i}_{\mathfrak{h}_2}, \mathfrak{h}_2}, \ldots, B_{\mathbf{i}_{\mathfrak{h}_g}, \mathbf{i}_{\mathfrak{h}_g}, \mathfrak{h}_g}\in\mathrm{Supp}_{\mathbb{B}_2}(M)$. Define
\begin{align}\label{Eq;11}
N=M-\sum_{j=1}^gc_{\mathbf{i}_{\mathfrak{h}_j}, \mathbf{i}_{\mathfrak{h}_j}, \mathfrak{h}_j}(M)C_{\mathfrak{h}_j}.
\end{align}
There is no loss to assume that $N\neq O$. As $M\in \mathrm{Z}(\mathbb{T})$, notice that $N\in\mathrm{Z}(\mathbb{T})$ by Lemma \ref{L;Lemma4.8} and Equation \eqref{Eq;11}. Notice that $\max(N)\leq \max(M)$ by Equation \eqref{Eq;11}. Assume that $\max(N)=\max(M)$.
Then there are $\mathbf{k}\in\mathbb{E}$ and $\mathfrak{l}\in\mathbb{U}_{\mathbf{k}, \mathbf{k}}$ such that $B_{\mathbf{k}, \mathbf{k},\mathfrak{l}}\in\mathrm{Supp}_{\mathbb{B}_2}(N)$ and $|\mathfrak{l}|=\max(M)$. Then the combination of Lemmas \ref{L;Lemma4.2}, \ref{L;Lemma4.10}, Equation \eqref{Eq;11}, Theorem \ref{T;Theorem3.13} implies that $\mathfrak{l}\notin\{\mathfrak{h}_1, \mathfrak{h}_2, \ldots, \mathfrak{h}_g\}$. By the choice of $\mathfrak{l}$ and Theorem \ref{T;Theorem3.13}, notice that $B_{\mathbf{k},\mathbf{k},\mathfrak{l}}\in\mathrm{Supp}_{\mathbb{B}_2}(M)$. According to Theorem \ref{T;Theorem3.13}, the containment $B_{\mathbf{k},\mathbf{k},\mathfrak{l}}\in\mathrm{Supp}_{\mathbb{B}_2}(M)$ contradicts the choice of $g$. So $\max(N)<\max(M)$. The desired lemma follows from the inductive hypothesis and Equation \eqref{Eq;11}.
\end{proof}
We are now ready to present an $\F$-basis of $\mathrm{Z}(\mathrm{T})$ and a property of this $\F$-basis.
\begin{thm}\label{T;Center}
$\mathrm{Z}(\mathbb{T})$ has an $\F$-basis $\{C_\mathfrak{a}: \exists\ \mathbf{b}\in\mathbb{E}, \mathfrak{a}\in \mathbb{U}_{\mathbf{b}, \mathbf{b}}\}$ whose cardinality is $2^{n_1+2n_4}3^{n_2+n_3}$. In particular, the $\F$-dimension of $\mathrm{Z}(\mathbb{T})$  is independent of the choices of $\F$ and $\mathbf{x}$.
\end{thm}
\begin{proof}
The desired theorem follows from combining Lemmas \ref{L;Lemma4.8}, \ref{L;Lemma4.11}, and \ref{L;Lemma4.4}.
\end{proof}
\begin{lem}\label{L;Lemma4.13}
Assume that $\mathbf{g}\in\mathbb{E}$ and $\mathfrak{h}\in\mathbb{U}_{\mathbf{g}, \mathbf{g}}$. Then $C_\mathfrak{h}C_\mathfrak{h}=\overline{k_\mathfrak{h}}C_\mathfrak{h}$. In particular, $\mathrm{Z}(\mathbb{T})$ is a semisimple $\F$-subalgebra of $\mathbb{T}$ only if $p\nmid k_\mathfrak{j}$ for any $\mathbf{i}\in \mathbb{E}$ and $\mathfrak{j}\in\mathbb{U}_{\mathbf{i}, \mathbf{i}}$.
\end{lem}
\begin{proof}
The desired lemma follows from combining Lemmas \ref{L;Lemma4.6}, \ref{L;Lemma3.20}, and \ref{L;Lemma4.8}.
\end{proof}
We end this section by presenting an example of Theorem \ref{T;Center} and Lemma \ref{L;Lemma4.13}.
\begin{eg}\label{E;Example4.14}
Assume that $\mathfrak{g}\!=\!(\varnothing, \varnothing, \{1\})$, $\mathfrak{h}=(\varnothing, \{1\}, \{1\})$, and $\mathfrak{i}=(\{1\}, \varnothing, \varnothing)$. Assume that $n=1$. By Theorem \ref{T;Center} and a direct computation, $\mathrm{Z}(\mathbb{T})$ has an $\F$-basis
\[\begin{cases}
\{C_\mathfrak{o}, C_\mathfrak{g}\}, &\text{if $\ell_1=m_1=2$},\\
\{C_\mathfrak{o}, C_\mathfrak{g}, C_\mathfrak{h}\}, &\text{if $\ell_1>m_1=2$},\\
\{C_\mathfrak{o}, C_\mathfrak{g}, C_\mathfrak{i}\}, &\text{if $m_1>\ell_1=2$},\\
\{C_\mathfrak{o}, C_\mathfrak{g}, C_\mathfrak{h}, C_\mathfrak{i}\}, &\text{if $\min\{\ell_1, m_1\}>2$}.
\end{cases}\]

Assume further that $n=\ell_1=2$, $\ell_2=m_1=m_2=3$, and $\mathfrak{g}=(\varnothing, \varnothing, \{2\})$. Assume that $\mathbf{j}=(0,2)$. As $\mathfrak{g}\in\mathbb{U}_{\mathbf{j}, \mathbf{j}}$, notice that $C_\mathfrak{g}$ is defined.
By Lemma \ref{L;Lemma4.13}, $C_\mathfrak{g}C_\mathfrak{g}=\overline{3}C_\mathfrak{g}$.
\end{eg}
\section{Algebraic structure of $\mathbb{T}$: Semisimplicity}
In this section, we determine $\mathrm{Rad}(E_\mathbf{g}^*\mathbb{T}E_\mathbf{g}^*)$ for any $\mathbf{g}\in \mathbb{E}$. As an application, we also determine the semisimplicity of $\mathbb{T}$. For this purpose, we recall
Notations \ref{N;Notation3.7}, \ref{N;Notation3.8}, \ref{N;Notation3.14}, \ref{N;Notation3.15}, \ref{N;Notation3.18}. We next begin our discussion with several
preliminary lemmas.
\begin{lem}\label{L;Lemma5.1}
Assume that $\mathbb{U}, \mathbb{V}, \mathbb{W}\subseteq [1,n]$. Then $p\mid k_{(\mathbb{U}, \mathbb{V}, \mathbb{W})}$ if and only if $$\{a: a\in\mathbb{U}, p\mid m_a-1\}\cup\{a: a\in\mathbb{V}, p\mid (\ell_a-1)m_a\}\cup\{a:a\in\mathbb{W}\setminus \mathbb{V}, p\mid m_a\}\neq\varnothing.$$
\end{lem}
\begin{proof}
The desired lemma follows from the above hypotheses and Equation \eqref{Eq;8}.
\end{proof}
\begin{lem}\label{L;Lemma5.2}
Assume that $\mathbf{g}\!\in\!\mathbb{E}$. Then the $\F$-subalgebra $E_\mathbf{g}^*\mathbb{T}E_\mathbf{g}^*$ of $\mathbb{T}$ has a two-sided ideal $\langle\{B_{\mathbf{g}, \mathbf{g}, \mathfrak{a}}: \mathfrak{a}\in\mathbb{U}_{\mathbf{g}, \mathbf{g}}, p\mid k_\mathfrak{a}\}\rangle_\mathbb{T}$.
\end{lem}
\begin{proof}
Assume that $B_{\mathbf{g}, \mathbf{g}, \mathfrak{h}}, B_{\mathbf{g}, \mathbf{g}, \mathfrak{i}}\!\in\! E_\mathbf{g}^*\mathbb{T}E_\mathbf{g}^*$ and $p\mid k_\mathfrak{i}$. Assume that $\mathfrak{h}\!=\!(\mathbbm{h}_{(1)}, \mathbbm{h}_{(2)}, \mathbbm{h}_{(3)})$ and $\mathfrak{i}=(\mathbbm{i}_{(1)}, \mathbbm{i}_{(2)}, \mathbbm{i}_{(3)})$. If $j\in\mathbbm{i}_{(1)}$ and $p\mid m_j-1$, then $j\in\mathbbm{h}_{(1)}\cup\mathbbm{i}_{(1)}$ and $p\mid m_j-1$. If $j\in\mathbbm{i}_{(2)}$ and $p\mid (\ell_j-1)m_j$, then $j\in \mathbbm{h}_{(2)}\cup\mathbbm{i}_{(2)}$ and $p\mid (\ell_j-1)m_j$. If $j\in(\mathbbm{h}_{(3)}\cap\mathbbm{i}_{(3)})\setminus\mathbbm{i}_{(2)}$ and $p\mid m_j$, then $j\in(\mathbbm{h}_{(3)}\cap\mathbbm{i}_{(3)})\setminus(\mathbbm{h}_{(2)}\cap\mathbbm{i}_{(2)})$ and $p\mid m_j$. If $j\!\in\!\mathbbm{i}_{(3)}\setminus(\mathbbm{h}_{(3)}\cup\mathbbm{i}_{(2)})$ and $p\mid m_j$, then $j\!\in\!(\mathbbm{h}_{(3)}\cup\mathbbm{i}_{(3)})\setminus(\mathbbm{h}_{(2)}\cup\mathbbm{i}_{(2)})$ and $p\mid m_j$. So $p\mid k_{\mathfrak{h}\cap\mathfrak{i}}k_{\mathfrak{h}\cup\mathfrak{i}}$ by Lemma \ref{L;Lemma5.1}. The desired lemma follows from the choices of $B_{\mathbf{g}, \mathbf{g}, \mathfrak{h}}$, $B_{\mathbf{g}, \mathbf{g}, \mathfrak{i}}$ and Lemma \ref{L;Lemma4.6}.
\end{proof}
Lemmas \ref{L;Lemma5.1} and \ref{L;Lemma5.2} motivate us to give the following notation and two lemmas.
\begin{nota}\label{N;Notation5.3}
Assume that $\mathbf{g}, \mathbf{h}\in \mathbb{E}$ and $\mathfrak{i}=(\mathbbm{i}_{(1)}, \mathbbm{i}_{(2)}, \mathbbm{i}_{(3)})\in\mathbb{U}_{\mathbf{g}, \mathbf{h}}$. Write $\mathbb{U}_\mathfrak{i}$ for $\{a: a\in \mathbbm{i}_{(1)}, p\mid m_a-1\}\!\cup\!\{a: a\in \mathbbm{i}_{(2)}, p\mid (\ell_a-1)m_a\}\!\cup\!\{a: a\in\mathbbm{i}_{(3)}\setminus\mathbbm{i}_{(2)}, p\mid m_a\}$. So Lemma \ref{L;Lemma5.1} shows that $p\mid k_\mathfrak{i}$ if and only if $\mathbb{U}_\mathfrak{i}\!\neq\!\varnothing$. Set $\mathbb{I}_\mathbf{g}\!=\!\langle\{B_{\mathbf{g}, \mathbf{g},\mathfrak{a}}\!:\! \mathfrak{a}\!\in\!\mathbb{U}_{\mathbf{g}, \mathbf{g}}, p\mid k_\mathfrak{a}\}\rangle_\mathbb{T}$. Lemma \ref{L;Lemma5.2} thus shows that $\mathbb{I}_\mathbf{g}$ is a two-sided ideal of the $\F$-subalgebra $E_\mathbf{g}^*\mathbb{T}E_\mathbf{g}^*$ of $\mathbb{T}$.
\end{nota}
\begin{lem}\label{L;Lemma5.4}
Assume that $\mathbf{g}\in \mathbb{E}$ and $\mathbb{I}_\mathbf{g}\neq\{O\}$. Then there is a nonzero matrix product of $(|\{a: a\!\in\!\mathbbm{g}_1, p\mid m_a\!-\!1\}|\!+\!|\{a: a\!\in\!\mathbbm{g}_2, p\!\mid\! (\ell_a\!-\!1)m_a\}|)$-many matrices in $\mathbb{I}_\mathbf{g}$.
\end{lem}
\begin{proof}
Set $\mathbb{U}\!=\!\{a: a\in\mathbbm{g}_1, p\mid m_a-1\}\cup\{a: a\in\mathbbm{g}_2, p\mid (\ell_a-1)m_a\}$. As $\mathbb{I}_\mathbf{g}\!\neq\!\{O\}$, $\mathbb{U}\!\neq\!\varnothing$. Let $h\in\mathbb{N}$ and $\mathbb{U}=\{i_1, i_2, \ldots, i_h\}$. If $j\in\mathbbm{g}_1$ and $p\mid m_j-1$, set $\mathfrak{k}_j\!\!=\!\!(\{j\}, \varnothing, \varnothing)$. If $j\in{\mathbbm{g}_2}^\bullet$ and $p\mid (\ell_j-1)m_j$, set $\mathfrak{k}_j\!\!=\!\!(\varnothing, \{j\},\{j\})$. If $j\in\mathbbm{g}_2\setminus{\mathbbm{g}_2}^\bullet$ and $p\mid (\ell_j-1)m_j$, set $\mathfrak{k}_j\!\!=\!\!(\varnothing, \varnothing, \{j\})$. Hence $\mathfrak{k}_{i_1}, \mathfrak{k}_{i_2}, \ldots, \mathfrak{k}_{i_h}\in\mathbb{U}_{\mathbf{g}, \mathbf{g}}$ and $B_{\mathbf{g}, \mathbf{g}, \mathfrak{k}_{i_1}}, B_{\mathbf{g}, \mathbf{g}, \mathfrak{k}_{i_2}}, \ldots, B_{\mathbf{g}, \mathbf{g}, \mathfrak{k}_{i_h}}\in\mathbb{I}_\mathbf{g}$. Then $B_{\mathbf{g}, \mathbf{g}, \mathfrak{k}_{i_1}}B_{\mathbf{g}, \mathbf{g}, \mathfrak{k}_{i_2}} \cdots B_{\mathbf{g}, \mathbf{g}, \mathfrak{k}_{i_h}}\neq O$ by Lemma \ref{L;Lemma4.6}. The desired lemma follows.
\end{proof}
\begin{lem}\label{L;Lemma5.5}
Assume that $\mathbf{g}\in \mathbb{E}$. Then $\mathbb{I}_\mathbf{g}$ is a nilpotent two-sided ideal of the $\F$-subalgebra
$E_\mathbf{g}^*\mathbb{T}E_\mathbf{g}^*$ of $\mathbb{T}$. Furthermore, the nilpotent index of $\mathbb{I}_\mathbf{g}$ equals $$|\{a: a\in\mathbbm{g}_1, p\mid m_a-1\}|+|\{a: a\in\mathbbm{g}_2, p\mid (\ell_a-1)m_a\}|+1.$$
\end{lem}
\begin{proof}
Set $\mathbb{U}\!=\!\{a: a\in\mathbbm{g}_1, p\mid m_a-1\}\cup\{a: a\in\mathbbm{g}_2, p\mid (\ell_a-1)m_a\}$. So $\mathbb{I}_\mathbf{g}\!\neq\!\{O\}$ if and only if $\mathbb{U}\!\neq\!\varnothing$. So it suffices to check the case $\mathbb{I}_\mathbf{g}\!\neq\!\{O\}$. Set $h\!=\!|\mathbb{U}|\!>\!0$. Let $B_{\mathbf{g}, \mathbf{g}, \mathfrak{i}_1}, B_{\mathbf{g}, \mathbf{g}, \mathfrak{i}_2}, \ldots, B_{\mathbf{g}, \mathbf{g}, \mathfrak{i}_{h+1}}\!\in\!\mathbb{I}_\mathbf{g}$. So $\mathbb{U}_{\mathfrak{i}_1}, \mathbb{U}_{\mathfrak{i}_2},\ldots, \mathbb{U}_{\mathfrak{i}_{h+1}}$ are nonempty subsets of $\mathbb{U}$. By the Pigeonhole Principle, there must exist $\mathfrak{j}, \mathfrak{k}\!\in\!\{\mathfrak{i}_1, \mathfrak{i}_2,\ldots, \mathfrak{i}_{h+1}\}$ such that $\mathbb{U}_\mathfrak{j}\!\cap\!\mathbb{U}_\mathfrak{k}\!\neq\!\varnothing$.

Assume that $\mathfrak{j}=(\mathbbm{j}_{(1)}, \mathbbm{j}_{(2)}, \mathbbm{j}_{(3)})$ and $\mathfrak{k}=(\mathbbm{k}_{(1)}, \mathbbm{k}_{(2)}, \mathbbm{k}_{(3)})$. Assume that $\ell\!\in\!\mathbb{U}_\mathfrak{j}\cap\mathbb{U}_\mathfrak{k}$. Then $\ell\in(\mathbbm{j}_{(1)}\cap\mathbbm{k}_{(1)})\cup(\mathbbm{j}_{(2)}\cap\mathbbm{k}_{(2)})\cup((\mathbbm{j}_{(3)}\cap\mathbbm{k}_{(3)})\setminus
(\mathbbm{j}_{(2)}\cap\mathbbm{k}_{(2)}))$. Hence $\ell\in \mathbb{U}_{\mathfrak{j}\cap\mathfrak{k}}$ and $p\mid k_{\mathfrak{j}\cap\mathfrak{k}}$. Hence the product of any $h+1$ matrices in $\mathbb{I}_\mathbf{g}$ is the zero matrix by Lemma \ref{L;Lemma4.6}. The desired lemma thus follows from the above discussion and Lemma \ref{L;Lemma5.4}.
\end{proof}
For further discussion, the next notations and a sequence of lemmas are required.
\begin{nota}\label{N;Notation5.6}
Assume that $\mathbf{g}, \mathbf{h}\in \mathbb{E}$ and $\mathfrak{i}=(\mathbbm{i}_{(1)}, \mathbbm{i}_{(2)}, \mathbbm{i}_{(3)})\in\mathbb{U}_{\mathbf{g}, \mathbf{h}}$. Notice that $((\mathbbm{g}_1\cap\mathbbm{h}_1)^\circ, {\mathbbm{i}_{(3)}}^\bullet, \mathbbm{g}_2\cap\mathbbm{h}_2)\in\mathbb{U}_{\mathbf{g}, \mathbf{h}}$. Use $(\mathbf{g}, \mathbf{h}; \mathfrak{i})$ to denote $((\mathbbm{g}_1\!\cap\!\mathbbm{h}_1)^\circ, {\mathbbm{i}_{(3)}}^\bullet, \mathbbm{g}_2\!\cap\!\mathbbm{h}_2)$. Set $n_{\mathbf{g}, \mathbf{h}, \mathfrak{i}}=|(\mathbf{g}, \mathbf{h}; \mathfrak{i})|-|\mathfrak{i}|$. So $(\mathbf{g}, \mathbf{h}; \mathfrak{i})=(\mathbf{h}, \mathbf{g}; \mathfrak{i})$ and $n_{\mathbf{g}, \mathbf{h}, \mathfrak{i}}=n_{\mathbf{h}, \mathbf{g}, \mathfrak{i}}$. Let $j, k\in[0, n_{\mathbf{g}, \mathbf{h}, \mathfrak{i}}]$. Set $\mathbb{U}_{\mathbf{g}, \mathbf{h}, \mathfrak{i}, j}=\{\mathfrak{a}: \mathfrak{i}\preceq\mathfrak{a}\preceq (\mathbf{g}, \mathbf{h}; \mathfrak{i}), |\mathfrak{a}|\!-\!|\mathfrak{i}|\!=\!j, p\nmid k_\mathfrak{a}\}$. Hence $\mathbb{U}_{\mathbf{g}, \mathbf{h}, \mathfrak{i}, j}=\mathbb{U}_{\mathbf{h}, \mathbf{g}, \mathfrak{i}, j}\subseteq\mathbb{U}_{\mathbf{g}, \mathbf{h}}$. If $p\nmid k_\mathfrak{i}$, then $\mathbb{U}_{\mathbf{g}, \mathbf{h}, \mathfrak{i}, 0}=\{\mathfrak{i}\}$. If $j\neq k$, then $\mathbb{U}_{\mathbf{g}, \mathbf{h}, \mathfrak{i}, j}\cap\mathbb{U}_{\mathbf{g}, \mathbf{h}, \mathfrak{i}, k}=\varnothing$.
For example, assume that $p=n=\ell_1=2$, $\ell_2=m_1=m_2=3$, $\mathbf{g}=(0,2)$, $\mathbf{h}=(1,2)$, and $\mathfrak{i}=(\varnothing, \varnothing, \{2\})$. It is clear to see that $(\mathbf{g}, \mathbf{h}; \mathfrak{i})=(\varnothing, \{2\},\{2\})$, $n_{\mathbf{g},\mathbf{h}, \mathfrak{i}}=1$, $\mathbb{U}_{\mathbf{g}, \mathbf{h}, \mathfrak{i}, 0}=\{\mathfrak{i}\}$, $\mathbb{U}_{\mathbf{g}, \mathbf{h}, \mathfrak{i}, 1}=\varnothing$.
\end{nota}
\begin{nota}\label{N;Notation5.7}
Assume that $\mathbf{g}, \mathbf{h}\in\mathbb{E}$ and $\mathfrak{i}\in\mathbb{U}_{\mathbf{g}, \mathbf{h}}$. Assume that $p\nmid k_{[\mathbf{g}, \mathbf{h}, \mathbf{g}]}$. Define
\begin{align}\label{Eq;12}
D_{\mathbf{g}, \mathbf{h}, \mathfrak{i}}=\sum_{j=0}^{n_{\mathbf{g}, \mathbf{h}, \mathfrak{i}}}\sum_{\mathfrak{k}\in\mathbb{U}_{\mathbf{g}, \mathbf{h}, \mathfrak{i}, j}}\!(\overline{-1})^j\overline{k_{[\mathbf{g}, \mathbf{h}, \mathbf{g}]}}^{-1}\overline{k_{\mathfrak{k}}}^{-1}\!B_{\mathbf{g}, \mathbf{h}, \mathfrak{k}},
\end{align}
where the sum over the empty set is the zero matrix. If $p\nmid k_\mathfrak{i}$, Equation \eqref{Eq;12} and Theorem \ref{T;Theorem3.13} imply that $\mathbb{U}_{\mathbf{g}, \mathbf{h}, \mathfrak{i}, 0}=\{\mathfrak{i}\}$ and $D_{\mathbf{g}, \mathbf{h}, \mathfrak{i}}$ is a defined nonzero matrix in $\mathbb{T}$. If $\mathbf{g}=\mathbf{h}$, notice that $k_{[\mathbf{g}, \mathbf{g},\mathbf{g}]}=1$ and $D_{\mathbf{g}, \mathbf{g}, \mathfrak{i}}\in E_\mathbf{g}^*\mathbb{T}E_\mathbf{g}^*$ by Equations \eqref{Eq;12} and \eqref{Eq;3}.
\end{nota}
\begin{lem}\label{L;Lemma5.8}
Assume that $\mathbf{g}, \mathbf{h}, \mathbf{i}\in \mathbb{E}$ and $\mathbb{U}\subseteq\mathbbm{g}_1$. Assume that $\mathbb{U}^1\subseteq\mathbb{V}\subseteq\mathbbm{i}_1$ and $\mathbb{U}^1
\subseteq\mathbb{V}^1\subseteq(\mathbbm{h}_1\cap\mathbbm{i}_1)^\circ$. Then $\mathbb{U}^1$, $(\mathbbm{g}_1\cap\mathbb{V})\setminus(\mathbb{U}\cup\mathbb{U}^1)$, $(\mathbbm{h}_1\cap\mathbbm{i}_1)^\circ\setminus(\mathbbm{g}_1\cup
\mathbb{U}^1)$, $(\mathbbm{i}_1\cap\mathbb{U})\setminus\mathbb{U}^1$ are pairwise disjoint. Moreover, $\mathbbm{g}_1\cap\mathbbm{i}_1\cap(\mathbb{U}\cup\mathbb{V})=\mathbbm{g}_1\cap\mathbbm{i}_1\cap(\mathbb{U}\cup\mathbb{V}^1)$ if and only if
$\mathbb{V}^1=\mathbb{U}^1\cup((\mathbbm{g}_1\cap\mathbb{V})\setminus(\mathbb{U}\cup\mathbb{U}^1))\cup\mathbb{W}$ and $\mathbb{W}\subseteq((\mathbbm{h}_1\cap\mathbbm{i}_1)^\circ\setminus(\mathbbm{g}_1\cup
\mathbb{U}^1))\cup((\mathbbm{i}_1\cap\mathbb{U})\setminus\mathbb{U}^1)$.
\end{lem}
\begin{proof}
The first statement is clear. One direction of the second statement is obvious since $\mathbbm{g}_1\cap\mathbbm{i}_1\cap(\mathbb{U}\cup\mathbb{V})=\mathbbm{g}_1\cap\mathbbm{i}_1\cap(\mathbb{U}\cup\mathbb{U}^1\cup((\mathbbm{g}_1\cap
\mathbb{V})\setminus(\mathbb{U}\cup\mathbb{U}^1)))$. Assume further that $\mathbbm{g}_1\cap\mathbbm{i}_1\cap(\mathbb{U}\cup\mathbb{V})=\mathbbm{g}_1\cap\mathbbm{i}_1\cap(\mathbb{U}\cup\mathbb{V}^1)$. Hence $(\mathbbm{g}_1\cap\mathbb{V})\setminus(\mathbb{U}\cup\mathbb{U}^1)=(\mathbbm{g}_1\cap\mathbb{V}^1)\setminus(\mathbb{U}\cup\mathbb{U}^1)$. If $j\in\mathbb{V}^1
\setminus(\mathbbm{g}_1\cup\mathbb{U}^1)$, then $j\in(\mathbbm{h}_1\cap\mathbbm{i}_1)^\circ\setminus(\mathbbm{g}_1\cup
\mathbb{U}^1)$. The desired lemma follows.
\end{proof}
\begin{lem}\label{L;Lemma5.9}
Assume that $\mathbf{g}, \mathbf{h}, \mathbf{i}\in \mathbb{E}$ and $\mathbb{U}\!\subseteq\!\mathbbm{g}_2$. Assume that $\mathbb{U}^1\!\subseteq\!\mathbb{V}\subseteq\mathbbm{i}_2$ and
$\mathbb{U}^1\subseteq\mathbb{V}^1\subseteq\mathbb{W}\subseteq(\mathbbm{h}_2\cap\mathbbm{i}_2)^\bullet$. Then $\mathbb{U}^1$, $(\mathbbm{g}_2\cap\mathbb{V})\setminus(\mathbb{U}\cup\mathbb{U}^1)$, $\mathbb{W}\setminus(\mathbbm{g}_2\cup\mathbb{U}^1)$, $(\mathbb{U}\cap\mathbb{W})\setminus\mathbb{U}^1$ are pairwise disjoint. Moreover, $\mathbbm{g}_2\cap\mathbbm{i}_2\cap(\mathbb{U}\cup\mathbb{V})=\mathbbm{g}_2\cap\mathbbm{i}_2\cap(\mathbb{U}\cup\mathbb{V}^1)$ if and only if
$\mathbb{V}^1=\mathbb{U}^1\cup((\mathbbm{g}_2\cap\mathbb{V})\setminus(\mathbb{U}\cup\mathbb{U}^1))\cup\mathbb{W}^1$ and $\mathbb{W}^1\subseteq(\mathbb{W}\setminus(\mathbbm{g}_2\cup
\mathbb{U}^1))\cup((\mathbb{U}\cap\mathbb{W})\setminus\mathbb{U}^1)$.
\end{lem}
\begin{proof}
The first statement is clear. One direction of the second statement is obvious since $\mathbbm{g}_2\cap\mathbbm{i}_2\cap(\mathbb{U}\cup\mathbb{V})=\mathbbm{g}_2\cap\mathbbm{i}_2\cap(\mathbb{U}\cup\mathbb{U}^1\cup((\mathbbm{g}_2\cap
\mathbb{V})\setminus(\mathbb{U}\cup\mathbb{U}^1)))$. Assume further that $\mathbbm{g}_2\cap\mathbbm{i}_2\cap(\mathbb{U}\cup\mathbb{V})=\mathbbm{g}_2\cap\mathbbm{i}_2\cap(\mathbb{U}\cup\mathbb{V}^1)$. Hence $(\mathbbm{g}_2\cap\mathbb{V})\setminus(\mathbb{U}\cup\mathbb{U}^1)=(\mathbbm{g}_2\cap\mathbb{V}^1)\setminus(\mathbb{U}\cup\mathbb{U}^1)$. If $j\in\mathbb{V}^1
\setminus(\mathbbm{g}_2\cup\mathbb{U}^1)$, notice that $j\in\mathbb{W}\setminus(\mathbbm{g}_2\cup
\mathbb{U}^1)$. The desired lemma follows.
\end{proof}
\begin{lem}\label{L;Lemma5.10}
Assume that $\mathbf{g}, \mathbf{h}, \mathbf{i}\in \mathbb{E}$ and $\mathbb{U}\!\subseteq\!\mathbbm{g}_2$. Assume that $\mathbb{U}^1\subseteq\mathbb{V}\subseteq\mathbbm{i}_2$ and $\mathbb{U}^1\subseteq\mathbb{V}^1\subseteq\mathbbm{h}_2\cap\mathbbm{i}_2$. Then $\mathbb{U}^1$, $(\mathbbm{g}_2\cap\mathbb{V})\setminus(\mathbb{U}\cup\mathbb{U}^1)$, $(\mathbbm{h}_2\cap\mathbbm{i}_2)\setminus(\mathbbm{g}_2\cup
\mathbb{U}^1)$, $(\mathbbm{i}_2\cap\mathbb{U})\setminus\mathbb{U}^1$ are pairwise disjoint. Moreover, $\mathbbm{g}_2\cap\mathbbm{i}_2\cap(\mathbb{U}\cup\mathbb{V})=\mathbbm{g}_2\cap\mathbbm{i}_2\cap(\mathbb{U}\cup\mathbb{V}^1)$ if and only if $\mathbb{V}^1=\mathbb{U}^1\cup((\mathbbm{g}_2\cap\mathbb{V})\setminus(\mathbb{U}\cup\mathbb{U}^1))\cup\mathbb{W}$ and $\mathbb{W}\subseteq((\mathbbm{h}_2\cap\mathbbm{i}_2)\setminus(\mathbbm{g}_2\cup
\mathbb{U}^1))\cup((\mathbbm{i}_2\cap\mathbb{U})\setminus\mathbb{U}^1)$.
\end{lem}
\begin{proof}
The first statement is clear. One direction of the second statement is obvious since $\mathbbm{g}_2\cap\mathbbm{i}_2\cap(\mathbb{U}\cup\mathbb{V})=\mathbbm{g}_2\cap\mathbbm{i}_2\cap(\mathbb{U}\cup\mathbb{U}^1\cup((\mathbbm{g}_2\cap
\mathbb{V})\setminus(\mathbb{U}\cup\mathbb{U}^1)))$. Assume further that $\mathbbm{g}_2\cap\mathbbm{i}_2\cap(\mathbb{U}\cup\mathbb{V})=\mathbbm{g}_2\cap\mathbbm{i}_2\cap(\mathbb{U}\cup\mathbb{V}^1)$. Hence $(\mathbbm{g}_2\cap\mathbb{V})\setminus(\mathbb{U}\cup\mathbb{U}^1)=(\mathbbm{g}_2\cap\mathbb{V}^1)\setminus(\mathbb{U}\cup\mathbb{U}^1)$. If $j\in\mathbb{V}^1
\setminus(\mathbbm{g}_2\cup\mathbb{U}^1)$, then $j\in(\mathbbm{h}_2\cap\mathbbm{i}_2)\setminus(\mathbbm{g}_2\cup
\mathbb{U}^1)$. The desired lemma follows.
\end{proof}
\begin{lem}\label{L;Lemma5.11}
Assume that $\mathbf{g}, \mathbf{h}, \mathbf{i}\in \mathbb{E}$ and $\mathfrak{j}=(\mathbbm{j}_{(1)}, \mathbbm{j}_{(2)}, \mathbbm{j}_{(3)})\!\!\in\!\! \mathbb{U}_{\mathbf{g}, \mathbf{h}}$. Assume that $\mathfrak{k}=(\mathbbm{k}_{(1)}, \mathbbm{k}_{(2)}, \mathbbm{k}_{(3)})$, $\mathfrak{l}\!=\!(\mathbbm{l}_{(1)}, \mathbbm{l}_{(2)}, \mathbbm{l}_{(3)})$, $\mathfrak{m}\!=\!
(\mathbbm{m}_{(1)}, \mathbbm{m}_{(2)}, \mathbbm{m}_{(3)})\in\mathbb{U}_{\mathbf{h}, \mathbf{i}}$. Assume that $\mathfrak{k}\preceq\mathfrak{l}\preceq(\mathbf{h}, \mathbf{i}; \mathfrak{k})$ and $\mathfrak{k}\!\preceq\!\mathfrak{m}\!\preceq\!(\mathbf{h}, \mathbf{i}; \mathfrak{k})$. Then $(\mathbf{g}, \mathbf{h}, \mathbf{i}, \mathfrak{j}, \mathfrak{l})\!\!=\!\!(\mathbf{g}, \mathbf{h}, \mathbf{i}, \mathfrak{j}, \mathfrak{m})$ if and only if $\mathbbm{m}_{(1)}\!=\!\mathbbm{k}_{(1)}\cup((\mathbbm{g}_1\cap\mathbbm{l}_{(1)})\setminus(\mathbbm{j}_{(1)}\cup\mathbbm{k}_{(1)}))\cup\mathbb{U}$,  $\mathbb{U}\!\subseteq\!((\mathbbm{h}_1\cap\mathbbm{i}_1)^\circ\setminus(\mathbbm{g}_1\!\cup\!
\mathbbm{k}_{(1)}))\cup((\mathbbm{i}_1\cap\mathbbm{j}_{(1)})\setminus\mathbbm{k}_{(1)})$, $\mathbbm{m}_{(2)}=\mathbbm{k}_{(2)}\cup((\mathbbm{g}_2\cap\mathbbm{l}_{(2)})\setminus(\mathbbm{j}_{(2)}\cup\mathbbm{k}_{(2)}))\cup\mathbb{V}$, $\mathbb{V}
\subseteq({\mathbbm{k}_{(3)}}^\bullet\setminus (\mathbbm{g}_2\cup\mathbbm{k}_{(2)}))\cup((\mathbbm{j}_{(2)}\cap{\mathbbm{k}_{(3)}})\setminus\mathbbm{k}_{(2)})$,
$\mathbbm{m}_{(3)}\!=\!\mathbbm{k}_{(3)}\cup((\mathbbm{g}_2\cap\mathbbm{l}_{(3)})\setminus(\mathbbm{j}_{(3)}\cup\mathbbm{k}_{(3)}))\cup\mathbb{W}$, $\mathbb{W}\!\subseteq\!((\mathbbm{h}_2\cap\mathbbm{i}_2)\setminus(\mathbbm{g}_2\cup
\mathbbm{k}_{(3)}))\cup((\mathbbm{i}_2\cap\mathbbm{j}_{(3)})\setminus\mathbbm{k}_{(3)})$.
\end{lem}
\begin{proof}
$(\mathbf{g}, \mathbf{h}, \mathbf{i}, \mathfrak{j}, \mathfrak{l})\!=\!(\mathbf{g}, \mathbf{h}, \mathbf{i}, \mathfrak{j}, \mathfrak{m})$ if and only if
$\mathbbm{g}_1\cap\mathbbm{i}_1\cap(\mathbbm{j}_{(1)}\cup\mathbbm{l}_{(1)})\!=\!\mathbbm{g}_1\cap\mathbbm{i}_1\cap(\mathbbm{j}_{(1)}\cup\mathbbm{m}_{(1)})$, $\mathbbm{g}_2\cap\mathbbm{i}_2\cap(\mathbbm{j}_{(2)}\cup\mathbbm{l}_{(2)})=\mathbbm{g}_2\cap\mathbbm{i}_2\cap(\mathbbm{j}_{(2)}\cup\mathbbm{m}_{(2)})$,  $\mathbbm{g}_2\cap\mathbbm{i}_2\cap(\mathbbm{j}_{(3)}\cup\mathbbm{l}_{(3)})=\mathbbm{g}_2\cap\mathbbm{i}_2\cap(\mathbbm{j}_{(3)}\cup\mathbbm{m}_{(3)})$. The desired lemma thus follows from combining Lemmas \ref{L;Lemma5.8}, \ref{L;Lemma5.9}, and \ref{L;Lemma5.10}.
\end{proof}
\begin{lem}\label{L;Lemma5.12}
Assume that $\mathbf{g}, \mathbf{h}, \mathbf{i}\in \mathbb{E}$, $\mathfrak{j}\in \mathbb{U}_{\mathbf{g}, \mathbf{h}}$, $\mathfrak{k}, \mathfrak{l}, \mathfrak{m}\in\mathbb{U}_{\mathbf{h}, \mathbf{i}}$, $\mathfrak{k}\preceq\mathfrak{l}\preceq(\mathbf{h}, \mathbf{i}; \mathfrak{k})$, $\mathfrak{k}\preceq\mathfrak{m}\preceq(\mathbf{h}, \mathbf{i}; \mathfrak{k})$, $(\mathbf{g}, \mathbf{h}, \mathbf{i}, \mathfrak{j}, \mathfrak{l})=(\mathbf{g}, \mathbf{h}, \mathbf{i}, \mathfrak{j}, \mathfrak{m})$, and $p\nmid k_{[\mathbf{g}, \mathbf{h}, \mathbf{g}]}k_\mathfrak{j}k_\mathfrak{k}k_\mathfrak{l}$. Then $p\nmid k_\mathfrak{m}$.
\end{lem}
\begin{proof}
Assume that $\mathfrak{j}=(\mathbbm{j}_{(1)}, \mathbbm{j}_{(2)}, \mathbbm{j}_{(3)})$, $\mathfrak{k}=(\mathbbm{k}_{(1)}, \mathbbm{k}_{(2)}, \mathbbm{k}_{(3)})$, $\mathfrak{l}=(\mathbbm{l}_{(1)}, \mathbbm{l}_{(2)}, \mathbbm{l}_{(3)})$, and $\mathfrak{m}\!=\!(\mathbbm{m}_{(1)}, \mathbbm{m}_{(2)}, \mathbbm{m}_{(3)})$. As $p\nmid k_\mathfrak{j}k_\mathfrak{k}k_\mathfrak{l}$, notice that $\mathbb{U}_\mathfrak{j}\!=\!\mathbb{U}_\mathfrak{k}\!=\!\mathbb{U}_\mathfrak{l}\!=\!\varnothing$. As $p\nmid k_{[\mathbf{g}, \mathbf{h}, \mathbf{g
}]}$, notice that $\{a: a\in \mathbbm{h}_1\setminus\mathbbm{g}_1, p\mid m_a-1\}\cup\{a: a\in \mathbbm{h}_2\setminus\mathbbm{g}_2, p\mid (\ell_a-1)m_a\}=\varnothing$. So   $\mathbbm{m}_{(1)}\subseteq\mathbbm{j}_{(1)}\cup\mathbbm{k}_{(1)}\cup\mathbbm{l}_{(1)}\cup(\mathbbm{h}_1\setminus\mathbbm{g}_1)$,
$\mathbbm{m}_{(2)}\subseteq\mathbbm{j}_{(2)}\cup\mathbbm{k}_{(2)}\cup\mathbbm{l}_{(2)}\cup(\mathbbm{h}_2\setminus\mathbbm{g}_2)$, and $\mathbbm{m}_{(3)}\setminus\mathbbm{m}_{(2)}\subseteq\mathbbm{j}_{(3)}\cup(\mathbbm{k}_{(3)}\setminus\mathbbm{k}_{(2)})\cup(\mathbbm{l}_{(3)}\setminus\mathbbm{l}_{(2)})\cup(\mathbbm{h}_2
\setminus\mathbbm{g}_2)$ by Lemma \ref{L;Lemma5.11}. Hence $\mathbb{U}_\mathfrak{m}=\varnothing$ by the above discussion. Hence $p\nmid k_\mathfrak{m}$. The desired lemma thus follows.
\end{proof}
\begin{lem}\label{L;Lemma5.13}
Assume that $\mathbf{g}, \mathbf{h}, \mathbf{i}\in \mathbb{E}$, $\mathfrak{j}\in \mathbb{U}_{\mathbf{g}, \mathbf{h}}$, and $\mathfrak{k}, \mathfrak{l}, \mathfrak{m}\in\mathbb{U}_{\mathbf{h}, \mathbf{i}}$. Assume that $\mathfrak{k}\!\preceq\!\mathfrak{l}\!\preceq\!(\mathbf{h}, \mathbf{i}; \mathfrak{k})$, $\mathfrak{k}\!\preceq\!\mathfrak{m}\!\preceq\!(\mathbf{h}, \mathbf{i}; \mathfrak{k})$, $(\mathbf{g}, \mathbf{h}, \mathbf{i}, \mathfrak{j}, \mathfrak{l})=(\mathbf{g}, \mathbf{h}, \mathbf{i}, \mathfrak{j}, \mathfrak{m})$. Then
$k_{\mathbf{g}\cap\mathfrak{m}}k_{\mathfrak{j}\cap\mathfrak{l}}=k_{\mathbf{g}\cap\mathfrak{l}}k_{\mathfrak{j}\cap\mathfrak{m}}$.
\end{lem}
\begin{proof}
Assume that $\mathfrak{j}=(\mathbbm{j}_{(1)}, \mathbbm{j}_{(2)}, \mathbbm{j}_{(3)})$, $\mathfrak{k}=(\mathbbm{k}_{(1)}, \mathbbm{k}_{(2)}, \mathbbm{k}_{(3)})$, $\mathfrak{l}=(\mathbbm{l}_{(1)}, \mathbbm{l}_{(2)}, \mathbbm{l}_{(3)})$, and $\mathfrak{m}\!\!=\!\!(\mathbbm{m}_{(1)}, \mathbbm{m}_{(2)}, \mathbbm{m}_{(3)})$. By Lemma \ref{L;Lemma5.11}, there are $\mathbb{U}\!\subseteq\!(\mathbbm{i}_1\cap\mathbbm{j}_{(1)})\!\setminus\!\mathbbm{k}_{(1)}$, $\mathbb{V}\!\subseteq\!(\mathbbm{j}_{(2)}\!\cap\!\mathbbm{k}_{(3)})\!\setminus\!\mathbbm{k}_{(2)}$, and $\mathbb{W}\subseteq(\mathbbm{i}_2\cap\mathbbm{j}_{(3)})\setminus\mathbbm{k}_{(3)}$ such that $\mathbbm{g}_1\cap\mathbbm{m}_{(1)}=(\mathbbm{j}_{(1)}\cap\mathbbm{k}_{(1)})\cup((\mathbbm{g}_1\cap\mathbbm{l}_{(1)})\setminus\mathbbm{j}_{(1)})
\cup\mathbb{U}$, $\mathbbm{g}_2\!\cap\!\mathbbm{m}_{(2)}\!=\!(\mathbbm{j}_{(2)}\!\cap\!\mathbbm{k}_{(2)})\!\cup\!((\mathbbm{g}_2\!\cap\!\mathbbm{l}_{(2)})\!\setminus\!\mathbbm{j}_{(2)})
\cup\mathbb{V}$, $\mathbbm{g}_2\!\cap\!\mathbbm{m}_{(3)}\!=\!(\mathbbm{j}_{(3)}\!\cap\!\mathbbm{k}_{(3)})\!\cup\!((\mathbbm{g}_2\!\cap\!\mathbbm{l}_{(3)})\!\setminus\!\mathbbm{j}_{(3)})
\cup\mathbb{W}$, $\mathbbm{j}_{(1)}\cap\mathbbm{m}_{(1)}=(\mathbbm{j}_{(1)}\!\cap\!\mathbbm{k}_{(1)})\!\cup\!\mathbb{U}$, $\mathbbm{j}_{(2)}\!\cap\!\mathbbm{m}_{(2)}\!=\!(\mathbbm{j}_{(2)}\!\cap\!\mathbbm{k}_{(2)})\!\cup\!\mathbb{V}$, and $\mathbbm{j}_{(3)}\cap\mathbbm{m}_{(3)}=(\mathbbm{j}_{(3)}\cap\mathbbm{k}_{(3)})\cup\mathbb{W}$.
The desired lemma thus follows from Lemma \ref{L;Lemma3.20} and Equation \eqref{Eq;8}.
\end{proof}
\begin{lem}\label{L;Lemma5.14}
Assume that $\mathbf{g}, \mathbf{h}, \mathbf{i}\in \mathbb{E}$. Assume that $\mathfrak{j}\!=\!(\mathbbm{j}_{(1)}, \mathbbm{j}_{(2)}, \mathbbm{j}_{(3)})\!\in\!\mathbb{U}_{\mathbf{g}, \mathbf{h}}$ and  $\mathfrak{k}\!=\!(\mathbbm{k}_{(1)}, \mathbbm{k}_{(2)}, \mathbbm{k}_{(3)})\!\in\!\mathbb{U}_{\mathbf{h}, \mathbf{i}}$. Assume that $p\nmid k_{[\mathbf{g}, \mathbf{h}, \mathbf{g}]}k_{[\mathbf{h}, \mathbf{i}, \mathbf{h}]}k_\mathfrak{j}k_\mathfrak{k}$. Then $B_{\mathbf{g}, \mathbf{h}, \mathfrak{j}}D_{\mathbf{h}, \mathbf{i}, \mathfrak{k}}\!\neq\! O$ only if the containments $((\mathbbm{h}_1\cap\mathbbm{i}_1)^\circ\setminus\mathbbm{g}_1)\!\cup\!(\mathbbm{i}_1\cap\mathbbm{j}_{(1)})\!\subseteq\!\mathbbm{k}_{(1)}$, $({\mathbbm{k}_{(3)}}^\bullet\setminus \mathbbm{g}_2)\cup(\mathbbm{j}_{(2)}\cap\mathbbm{k}_{(3)})\!\subseteq\!\mathbbm{k}_{(2)}$, and $((\mathbbm{h}_2\cap\mathbbm{i}_2)\setminus\mathbbm{g}_2)\cup(\mathbbm{i}_2\cap\mathbbm{j}_{(3)})\subseteq\mathbbm{k}_{(3)}$ hold.
\end{lem}
\begin{proof}
Assume that the three containments do not hold together and $B_{\mathbf{g
}, \mathbf{h}, \mathfrak{j}}D_{\mathbf{h}, \mathbf{i}, \mathfrak{k}}\!\neq\!O$. Assume that $B_{\mathbf{h}, \mathbf{i}, \mathfrak{l}}\!\in\!\mathrm{Supp}_{\mathbb{B}_2}(D_{\mathbf{h}, \mathbf{i}, \mathfrak{k}})$ and $B_{\mathbf{g}, \mathbf{i}, (\mathbf{g}, \mathbf{h}, \mathbf{i}, \mathfrak{j}, \mathfrak{l})}\!\in\!\mathrm{Supp}_{\mathbb{B}_2}(B_{\mathbf{g}, \mathbf{h}, \mathfrak{j}}D_{\mathbf{h}, \mathbf{i}, \mathfrak{k}})$ by Theorem \ref{T;Theorem3.28}. So $p\nmid k_{[\mathbf{g}, \mathbf{h}, \mathbf{g}]}k_\mathfrak{j}k_\mathfrak{k}k_\mathfrak{l}$ by Equation \eqref{Eq;12} and Theorem \ref{T;Theorem3.13}. Set $\mathfrak{l}=(\mathbbm{l}_{(1)}, \mathbbm{l}_{(2)}, \mathbbm{l}_{(3)})$. Let
$m=|(\mathbbm{g}_1\cap\mathbbm{l}_{(1)})\setminus(\mathbbm{j}_{(1)}\cup\mathbbm{k}_{(1)})|+
|(\mathbbm{g}_2\cap\mathbbm{l}_{(2)})\setminus(\mathbbm{j}_{(2)}\cup\mathbbm{k}_{(2)})|+|(\mathbbm{g}_2\cap\mathbbm{l}_{(3)})\setminus(\mathbbm{j}_{(3)}\cup\mathbbm{k}_{(3)})|$.   Let $q$ be the sum of $|(\mathbbm{h}_1\cap\mathbbm{i}_1)^\circ\setminus(\mathbbm{g}_1\cup\mathbbm{k}_{(1)})|$, $|(\mathbbm{i}_1\cap\mathbbm{j}_{(1)})\setminus\mathbbm{k}_{(1)}|$, $|{\mathbbm{k}_{(3)}}^\bullet\setminus (\mathbbm{g}_2\cup\mathbbm{k}_{(2)})|$, $|(\mathbbm{j}_{(2)}\cap\mathbbm{k}_{(3)})\setminus\mathbbm{k}_{(2)}|$, $|(\mathbbm{h}_2\cap\mathbbm{i}_2)\setminus(\mathbbm{g}_2\cup
\mathbbm{k}_{(3)})|$, and $|(\mathbbm{i}_2\cap\mathbbm{j}_{(3)})\setminus\mathbbm{k}_{(3)}|$.
Therefore $q>0$. By combining Lemmas \ref{L;Lemma5.8}, \ref{L;Lemma5.9}, \ref{L;Lemma5.10}, \ref{L;Lemma5.11}, \ref{L;Lemma5.12}, \ref{L;Lemma5.13}, \ref{L;Lemma3.20}, Theorems \ref{T;Theorem3.28}, and \ref{T;Theorem3.13},
$$c_{\mathbf{g}, \mathbf{i}, (\mathbf{g}, \mathbf{h}, \mathbf{i}, \mathfrak{j}, \mathfrak{l})}(B_{\mathbf{g}, \mathbf{h}, \mathfrak{j}}D_{\mathbf{h}, \mathbf{i}, \mathfrak{k}})=\sum_{r=0}^q(\overline{-1})^{m+r}\overline{k_{[\mathbf{h}, \mathbf{i}, \mathbf{h}]}}^{-1}\overline{k_{\mathbf{g}\cap\mathfrak{l}}}^{-1}\overline{k_{[\mathbf{g}, \mathbf{h}, \mathbf{i}]}}\overline{k_{\mathfrak{j}\cap\mathfrak{l}}}\overline{k_{\mathfrak{j}\setminus \mathbf{i}}}\overline{{q\choose r}}=\overline{0},$$
which is absurd as $B_{\mathbf{g}, \mathbf{i}, (\mathbf{g}, \mathbf{h}, \mathbf{i}, \mathfrak{j}, \mathfrak{l})}\!\in\!\mathrm{Supp}_{\mathbb{B}_2}(B_{\mathbf{g}, \mathbf{h}, \mathfrak{j}}D_{\mathbf{h}, \mathbf{i}, \mathfrak{k}})$. The desired lemma follows.
\end{proof}
\begin{lem}\label{L;Lemma5.15}
Assume that $\mathbf{g}\in\mathbb{E}$ and $\mathfrak{h}, \mathfrak{i}\in \mathbb{U}_{\mathbf{g}, \mathbf{g}}$. Assume that $p\nmid k_\mathfrak{h}k_\mathfrak{i}$. Then
\[B_{\mathbf{g}, \mathbf{g}, \mathfrak{h}}D_{\mathbf{g}, \mathbf{g}, \mathfrak{i}}=D_{\mathbf{g}, \mathbf{g}, \mathfrak{i}}B_{\mathbf{g}, \mathbf{g}, \mathfrak{h}}=\begin{cases}
\overline{k_\mathfrak{h}}D_{\mathbf{g}, \mathbf{g}, \mathfrak{i}}, &\ \text{if}\ \mathfrak{h}\preceq\mathfrak{i},\\
O, &\ \text{otherwise}.
\end{cases}\]
\end{lem}
\begin{proof}
If $\mathfrak{h}\preceq\mathfrak{i}$, notice that $B_{\mathbf{g}, \mathbf{g}, \mathfrak{h}}D_{\mathbf{g}, \mathbf{g}, \mathfrak{i}}\!=\!D_{\mathbf{g}, \mathbf{g}, \mathfrak{i}}B_{\mathbf{g}, \mathbf{g}, \mathfrak{h}}=\overline{k_\mathfrak{h}}D_{\mathbf{g}, \mathbf{g}, \mathfrak{i}}$ by Equation \eqref{Eq;12} and Lemma \ref{L;Lemma4.6}.
By combining Lemmas \ref{L;Lemma4.6}, \ref{L;Lemma5.14}, and a direct computation, notice that
$B_{\mathbf{g}, \mathbf{g}, \mathfrak{h}}D_{\mathbf{g}, \mathbf{g}, \mathfrak{i}}=D_{\mathbf{g}, \mathbf{g}, \mathfrak{i}}B_{\mathbf{g}, \mathbf{g}, \mathfrak{h}}\neq O$ only if $\mathfrak{h}\preceq\mathfrak{i}$. The desired lemma thus follows.
\end{proof}
\begin{lem}\label{L;Lemma5.16}
Assume that $\mathbf{g}\in\mathbb{E}$ and $\mathfrak{h}, \mathfrak{i}\in \mathbb{U}_{\mathbf{g}, \mathbf{g}}$. Assume that $p\nmid k_\mathfrak{h}k_\mathfrak{i}$. Then
$$D_{\mathbf{g}, \mathbf{g}, \mathfrak{h}}D_{\mathbf{g}, \mathbf{g}, \mathfrak{i}}=\delta_{\mathfrak{h},\mathfrak{i}}D_{\mathbf{g}, \mathbf{g}, \mathfrak{h}}.$$
\end{lem}
\begin{proof}
By Equation \eqref{Eq;12} and Lemma \ref{L;Lemma5.15}, $D_{\mathbf{g}, \mathbf{g}, \mathfrak{h}}D_{\mathbf{g}, \mathbf{g}, \mathfrak{i}}\!=\!D_{\mathbf{g}, \mathbf{g}, \mathfrak{i}}D_{\mathbf{g}, \mathbf{g}, \mathfrak{h}}\!\neq\! O$ only if $\mathfrak{h}\preceq\mathfrak{i}\preceq\mathfrak{h}$. Hence $D_{\mathbf{g}, \mathbf{g}, \mathfrak{h}}D_{\mathbf{g}, \mathbf{g}, \mathfrak{i}}\!=\!D_{\mathbf{g}, \mathbf{g}, \mathfrak{i}}D_{\mathbf{g}, \mathbf{g}, \mathfrak{h}}\!\neq\! O$ only if $\mathfrak{h}=\mathfrak{i}$. By Equation \eqref{Eq;12} and Lemma \ref{L;Lemma5.15}, notice that $D_{\mathbf{g}, \mathbf{g}, \mathfrak{h}}D_{\mathbf{g}, \mathbf{g}, \mathfrak{h}}=D_{\mathbf{g}, \mathbf{g}, \mathfrak{h}}$. The desired lemma thus follows.
\end{proof}
We are now ready to give the first main result of this section and two corollaries.
\begin{thm}\label{T;Theorem5.17}
Assume that $\mathbf{g}\in \mathbb{E}$. Then $|\{\mathfrak{a}: \mathfrak{a}\in \mathbb{U}_{\mathbf{g}, \mathbf{g}}, p\nmid k_\mathfrak{a}\}|=2^{n_{\mathbf{g}, 1}+n_{\mathbf{g}, 3}}3^{n_{\mathbf{g}, 2}}$, where $n_{\mathbf{g}, 1}=|\{a: a\in {\mathbbm{g}_1}^\circ, p\nmid m_a-1\}|$, $n_{\mathbf{g}, 2}=|\{a: a\in {\mathbbm{g}_2}^\bullet, p\nmid (\ell_a-1)m_a\}|$, and  $n_{\mathbf{g}, 3}=|\{a: a\in \mathbbm{g}_2\setminus {\mathbbm{g}_2}^\bullet, p\nmid m_a\}|$. Furthermore, $E_\mathbf{g}^*\mathbb{T}E_\mathbf{g}^*/\mathbb{I}_\mathbf{g}\cong2^{n_{\mathbf{g}, 1}+n_{\mathbf{g}, 3}}3^{n_{\mathbf{g}, 2}}\mathrm{M}_1(\F)$ as $\F$-algebras. In particular, $\mathrm{Rad}(E_\mathbf{g}^*\mathbb{T}E_\mathbf{g}^*)=\mathbb{I}_\mathbf{g}$ and the nilpotent index of $\mathrm{Rad}(E_\mathbf{g}^*\mathbb{T}E_\mathbf{g}^*)$ equals $|\{a: a\in\mathbbm{g}_1, p\mid m_a-1\}|+|\{a: a\in\mathbbm{g}_2, p\mid (\ell_a-1)m_a\}|+1$.
\end{thm}
\begin{proof}
Set $\mathbb{U}=\{D_{\mathbf{g}, \mathbf{g}, \mathfrak{a}}: \mathfrak{a}\in\mathbb{U}_{\mathbf{g}, \mathbf{g}}, p\nmid k_\mathfrak{a}\}\cup\{B_{\mathbf{g}, \mathbf{g}, \mathfrak{a}}: \mathfrak{a}\in\mathbb{U}_{\mathbf{g}, \mathbf{g}}, p\mid k_\mathfrak{a}\}$. Then the $\F$-subalgebra $E_\mathbf{g}^*\mathbb{T}E_\mathbf{g}^*$ of $\mathbb{T}$ has an $\F$-basis $\mathbb{U}$ by combining Equation \eqref{Eq;12}, Lemmas
\ref{L;Lemma4.6}, and \ref{L;Lemma5.16}. Therefore $\{D_{\mathbf{g}, \mathbf{g}, \mathfrak{a}}+\mathbb{I}_\mathbf{g}: \mathfrak{a}\in \mathbb{U}_{\mathbf{g}, \mathbf{g}}, p\nmid k_\mathfrak{a}\}$
is an $\F$-basis of $E_\mathbf{g}^*\mathbb{T}E_\mathbf{g}^*/\mathbb{I}_\mathbf{g}$. Notice that $|\{D_{\mathbf{g}, \mathbf{g}, \mathfrak{a}}\!+\!\mathbb{I}_\mathbf{g}:\mathfrak{a}\!\in\!\mathbb{U}_{\mathbf{g}, \mathbf{g}}, p\nmid k_\mathfrak{a}\}|\!\!=\!\!|\{\mathfrak{a}: \mathfrak{a}\!\in\! \mathbb{U}_{\mathbf{g}, \mathbf{g}}, p\nmid k_\mathfrak{a}\}|\!=\!2^{n_{\mathbf{g}, 1}+n_{\mathbf{g}, 3}}3^{n_{\mathbf{g}, 2}}$ by a direct computation. The first two statements thus follow from the above discussion and Lemma \ref{L;Lemma5.16}. The desired theorem thus follows from Lemmas \ref{L;Lemma2.5} and \ref{L;Lemma5.5}.
\end{proof}
\begin{cor}\label{C;Corollary5.18}
Assume that $\mathbf{g}\in \mathbb{E}$, $\mathfrak{h}\in \mathbb{U}_{\mathbf{g}, \mathbf{g}}$, $p\nmid k_\mathfrak{h}$, $\mathbbm{Irr}(\mathbf{g}, \mathfrak{h})\!=\!\langle\{D_{\mathbf{g}, \mathbf{g}, \mathfrak{h}}+\mathbb{I}_\mathbf{g}\}\rangle_{\mathbb{T}/\mathbb{I}_\mathbf{g}}$. Then the set $\{\mathbbm{Irr}(\mathbf{g}, \mathfrak{a}): \mathfrak{a}\in\mathbb{U}_{\mathbf{g}, \mathbf{g}}, p\nmid k_\mathfrak{a}\}$ comprises a complete system of distinct representatives of all isomorphism classes of irreducible $E_\mathbf{g}^*\mathbb{T}E_\mathbf{g}^*$-modules, and the irreducible $E_\mathbf{g}^*\mathbb{T}E_\mathbf{g}^*$-modules are precisely the absolutely irreducible $E_\mathbf{g}^*\mathbb{T}E_\mathbf{g}^*$-modules.
\end{cor}
\begin{proof}
Notice that $\mathbbm{Irr}(\mathbf{g}, \mathfrak{h})$ is an irreducible $E_\mathbf{g}^*\mathbb{T}E_\mathbf{g}^*$-module by combining Lemmas \ref{L;Lemma5.15}, \ref{L;Lemma5.2}, and \ref{L;Lemma4.6}.
Assume that $\mathfrak{i}, \mathfrak{j}\in \mathbb{U}_{\mathbf{g}, \mathbf{g}}$ and $p\nmid k_\mathfrak{i}k_\mathfrak{j}$. Lemma \ref{L;Lemma5.16} thus can imply that $\mathfrak{i}=\mathfrak{j}$ if and only if $\mathbbm{Irr}(\mathbf{g}, \mathfrak{i})\cong\mathbbm{Irr}(\mathbf{g}, \mathfrak{j})$ as irreducible $E_\mathbf{g}^*\mathbb{T}E_\mathbf{g}^*$-modules. Hence the desired corollary thus follows from Theorem \ref{T;Theorem5.17} and Lemma \ref{L;Lemma2.6}.
\end{proof}
\begin{cor}\label{C;Corollary5.19}
Assume that $\mathbf{g}\in \mathbb{E}$. Then the $\F$-subalgebra $E_\mathbf{g}^*\mathbb{T}E_\mathbf{g}^*$ of $\mathbb{T}$ is a semisimple $\F$-algebra if and only if $p\nmid k_\mathbf{g}$. In particular,
the $\F$-subalgebra $E_\mathbf{g}^*\mathbb{T}E_\mathbf{g}^*$ of $\mathbb{T}$ is a semisimple $\F$-algebra if and only if $p\nmid \prod_{h\in\mathbbm{g}_1}(m_h-1)\prod_{i\in\mathbbm{g}_2}(\ell_i-1)m_i$, where a product over the empty set equals one.
\end{cor}
\begin{proof}
The desired corollary follows from Lemma \ref{L;Lemma3.19} and Theorem \ref{T;Theorem5.17}.
\end{proof}
We next introduce another lemma and the second main result of this section.
\begin{lem}\label{L;Lemma5.20}
Assume that $\mathbf{g}, \mathbf{h}\in\mathbb{E}$, $\mathbf{g}_i=1$, and $\mathbf{h}_i=2$ for any $i\in [1, n]$. Then $\mathfrak{S}$ is a $p'$-valenced scheme if and only if $p\nmid k_\mathbf{g}k_\mathbf{h}$. In particular, $\mathfrak{S}$ is a $p'$-valenced scheme if and only if $p\nmid \prod_{i=1}^n(\ell_i-1)(m_i-1)m_i$.
\end{lem}
\begin{proof}
The desired lemma follows from Lemma \ref{L;Lemma3.19} and a direct computation.
\end{proof}
\begin{thm}\label{T;Semisimplicity}
$\mathbb{T}$ is a semisimple $\F$-algebra if and only if $\mathfrak{S}$ is a $p'$-valenced scheme. In particular, $\mathbb{T}$ is a semisimple $\F$-algebra if and only if $p\nmid \prod_{g=1}^n(\ell_g-1)(m_g-1)m_g$.
\end{thm}
\begin{proof}
Assume that $\mathfrak{S}$ is a $p'$-valenced scheme. Then $p\nmid k_\mathbf{g}$ for any $\mathbf{g}\!\in\! \mathbb{E}$. Corollary \ref{C;Corollary5.19} and Lemma \ref{L;Lemma2.4} give $\mathrm{Rad}(E_\mathbf{g}^*\mathbb{T}E_\mathbf{g}^*)\!=\!\{O\}$ for any $\mathbf{g}\in \mathbb{E}$. Pick $M\in\mathrm{Rad}(\mathbb{T})$. Assume further that $M\neq O$. Lemma \ref{L;Lemma2.4} implies that $E_\mathbf{g}^*ME_\mathbf{g}^*=O$ for any $\mathbf{g}\in\mathbb{E}$. As $M\!\neq\! O$ and Equation \eqref{Eq;4} holds, $E_\mathbf{h}^*ME_\mathbf{i}^*\neq O$ for some distinct $\mathbf{h}, \mathbf{i}\in \mathbb{E}$. By Lemma \ref{L;Lemma3.1},
there is $\mathbf{j}\in\mathbb{E}$ such that $\mathbbm{j}_1=\mathbbm{i}_1\setminus\mathbbm{h}_1$ and $\mathbbm{j}_2=\mathbbm{i}_2\setminus\mathbbm{h}_2$. Hence $k_\mathbf{j}=k_{[\mathbf{h}, \mathbf{i}, \mathbf{h}]}$.

By Equation \eqref{Eq;3}, there must exist $k\in\mathbb{N}$ and pairwise distinct $\mathfrak{l}_1, \mathfrak{l}_2, \ldots, \mathfrak{l}_k\in\mathbb{U}_{\mathbf{h}, \mathbf{i}}$ such that $\mathrm{Supp}_{\mathbb{B}_2}(E_\mathbf{h}^*ME_\mathbf{i}^*)=\{B_{\mathbf{h}, \mathbf{i}, \mathfrak{l}_1}, B_{\mathbf{h}, \mathbf{i}, \mathfrak{l}_2},\ldots, B_{\mathbf{h}, \mathbf{i}, \mathfrak{l}_k}\}$. By a direct computation, notice that $(\mathbf{h}, \mathbf{i}, \mathbf{h}, \mathfrak{l}_1, \mathfrak{o}), (\mathbf{h}, \mathbf{i}, \mathbf{h}, \mathfrak{l}_2, \mathfrak{o}),\ldots, (\mathbf{h}, \mathbf{i}, \mathbf{h}, \mathfrak{l}_k, \mathfrak{o})$ are pairwise distinct in $\mathbb{U}_{\mathbf{h}, \mathbf{h}}$. By Theorem \ref{T;Theorem3.13}, $B_{\mathbf{h}, \mathbf{h}, (\mathbf{h}, \mathbf{i}, \mathbf{h}, \mathfrak{l}_1, \mathfrak{o})}, B_{\mathbf{h}, \mathbf{h}, (\mathbf{h}, \mathbf{i}, \mathbf{h}, \mathfrak{l}_2, \mathfrak{o})}, \ldots, B_{\mathbf{h}, \mathbf{h}, (\mathbf{h}, \mathbf{i}, \mathbf{h}, \mathfrak{l}_k, \mathfrak{o})}$ are pairwise distinct.  As $k_{\mathfrak{l}_1\setminus \mathbf{h}}=k_{\mathfrak{o}\setminus \mathbf{h}}=k_{\mathfrak{l}_1\cap\mathfrak{o}}=1$,
$c_{\mathbf{h}, \mathbf{h}, (\mathbf{h}, \mathbf{i}, \mathbf{h}, \mathfrak{l}_1, \mathfrak{o})}(E_\mathbf{h}^*ME_\mathbf{i}^*B_{\mathbf{i}, \mathbf{h}, \mathfrak{o}})=c_{\mathbf{h}, \mathbf{i}, \mathfrak{l}_1}(E_\mathbf{h}^*ME_\mathbf{i}^*)\overline{k_\mathbf{j}}\in{\F}^\times$ by combining Equation \eqref{Eq;3}, Theorems \ref{T;Theorem3.13}, \ref{T;Theorem3.28}. In particular, $E_\mathbf{h}^*ME_\mathbf{i}^*B_{\mathbf{i}, \mathbf{h}, \mathfrak{o}}\!\neq\! O$ by Theorem \ref{T;Theorem3.13}. As $M\in\mathrm{Rad}(\mathbb{T})$, notice that $E_\mathbf{h}^*ME_\mathbf{i}^*B_{\mathbf{i}, \mathbf{h}, \mathfrak{o}}\in \mathrm{Rad}(E_\mathbf{h}^*\mathbb{T}E_\mathbf{h}^*)\setminus\{O\}$ by Equation \eqref{Eq;3} and Lemma \ref{L;Lemma2.4}. This is a contradiction as $\mathrm{Rad}(E_\mathbf{h}^*\mathbb{T}E_\mathbf{h}^*)=\{O\}$. So $\mathbb{T}$ is a semisimple $\F$-algebra by Lemma \ref{L;Lemma2.4}. The first statement is from Lemma \ref{L;Lemma2.7}. The desired theorem thus follows from the first statement and Lemma \ref{L;Lemma5.20}.
\end{proof}
We conclude this section by giving a corollary of Theorem \ref{T;Semisimplicity} and an example.
\begin{cor}\label{C;Corollary5.22}
Assume that $\mathbf{g}, \mathbf{h}\in\mathbb{E}$, $\mathbf{g}_i=1$, and $\mathbf{h}_i=2$ for any $i\in [1, n]$. Then the following are equivalent: $\mathbb{T}$ is a semisimple $\F$-algebra; the $\F$-subalgebra $E_\mathbf{i}^*\mathbb{T}E_\mathbf{i}^*$ of $\mathbb{T}$ is a semisimple $\F$-algebra for any $\mathbf{i}\in\mathbb{E}$; the $\F$-subalgebras $E_\mathbf{g}^*\mathbb{T}E_\mathbf{g}^*$ and $E_\mathbf{h}^*\mathbb{T}E_\mathbf{h}^*$ of $\mathbb{T}$ are semisimple $\F$-algebras; the $\F$-subalgebra $\mathrm{Z}(\mathbb{T})$ of $\mathbb{T}$ is a semisimple $\F$-algebra.
\end{cor}
\begin{proof}
Notice that the first three statements are equivalent by combining Theorem \ref{T;Semisimplicity}, Corollary \ref{C;Corollary5.19}, and Lemma \ref{L;Lemma5.20}. By Lemma \ref{L;Lemma2.4}, $\mathbb{T}$ is a semisimple $\F$-algebra only if the $\F$-subalgebra
$\mathrm{Z}(\mathbb{T})$ of $\mathbb{T}$ is a semisimple $\F$-algebra. By Lemma \ref{L;Lemma3.19}, notice that $({\mathbbm{i}_1}^\circ, {\mathbbm{i}_2}^\bullet, \mathbbm{i}_2)\in\mathbb{U}_{\mathbf{i},\mathbf{i}}$ and
$k_\mathbf{i}=k_{({\mathbbm{i}_1}^\circ, {\mathbbm{i}_2}^\bullet, \mathbbm{i}_2)}$ for any $\mathbf{i}\in\mathbb{E}$. The desired corollary follows from combining
Lemma \ref{L;Lemma4.13}, Theorem \ref{T;Semisimplicity}, and the above discussion.
\end{proof}
\begin{eg}\label{E;Example5.23}
Assume that $n=1$ and $g$ denotes the 1-tuple $(g)$ for any $g\in [0,2]$. Then $E_0^*\mathbb{T}E_0^*\cong\mathrm{M}_1(\F)$ as $\F$-algebras. By Corollary \ref{C;Corollary5.19}, the $\F$-subalgebra $E_1^*\mathbb{T}E_1^*$ of $\mathbb{T}$ is a semisimple $\F$-algebra if and only if $p\nmid m_1-1$. By Corollary \ref{C;Corollary5.19} again, the $\F$-subalgebra $E_2^*\mathbb{T}E_2^*$ of $\mathbb{T}$ is a semisimple $\F$-algebra if and only if $p\nmid (\ell_1-1)m_1$. If $p\mid m_1-1$, Theorem \ref{T;Theorem5.17} implies that $E_1^*\mathbb{T}E_1^*/\mathrm{Rad}(E_1^*\mathbb{T}E_1^*)\cong \mathrm{M}_1(\F)$ as $\F$-algebras. If $p\mid (\ell_1-1)m_1$, Theorem \ref{T;Theorem5.17} gives $E_2^*\mathbb{T}E_2^*/\mathrm{Rad}(E_2^*\mathbb{T}E_2^*)\cong \mathrm{M}_1(\F)$ as $\F$-algebras. By Theorem \ref{T;Semisimplicity}, $\mathbb{T}$ is a semisimple $\F$-algebra if and only if $p\nmid (\ell_1-1)(m_1-1)m_1$.

Assume further that $n=\ell_1=2$ and $\ell_2=m_1=m_2=3$. By Theorem \ref{T;Semisimplicity} and a direct computation, notice that $\mathbb{T}$ is a semisimple $\F$-algebra if and only if $p\not\in[2, 3]$.
\end{eg}
\section{Algebraic structure of $\mathbb{T}$: Jacobson radical}
In this section, we determine $\mathrm{Rad}(\mathbb{T})$ and its nilpotent index. For this purpose, we recall Notations \ref{N;Notation3.7}, \ref{N;Notation3.8}, \ref{N;Notation3.14}, \ref{N;Notation3.15}, \ref{N;Notation3.18}, \ref{N;Notation4.1}, \ref{N;Notation5.3} and start with another notation.
\begin{nota}\label{N;Notation6.1}
Assume that $\mathbf{g}, \mathbf{h}\in \mathbb{E}$ and $\mathfrak{i}\in\mathbb{U}_{\mathbf{g}, \mathbf{h}}$. Then $\mathbb{V}_{\mathbf{g}, \mathbf{h}, \mathfrak{i}}$ is defined to be $\{a: a\in\mathbbm{g}_1\triangle\mathbbm{h}_1, p\mid m_a-1\}\cup\{a: a\in\mathbbm{g}_2\triangle\mathbbm{h}_2, p\mid (\ell_a-1)m_a\}\cup\mathbb{U}_\mathfrak{i}$. As $\mathbb{U}_{\mathbf{g}, \mathbf{h}}=\mathbb{U}_{\mathbf{h}, \mathbf{g}}$, it is clear that $\mathbb{V}_{\mathbf{g}, \mathbf{h}, \mathfrak{i}}=\mathbb{V}_{\mathbf{h}, \mathbf{g}, \mathfrak{i}}$. Define $\mathbb{I}=\langle\{B_{\mathbf{a}, \mathbf{b}, \mathfrak{c}}: \mathbf{a}, \mathbf{b}\in \mathbb{E}, \mathfrak{c}\in\mathbb{U}_{\mathbf{a}, \mathbf{b}}, \mathbb{V}_{\mathbf{a}, \mathbf{b}, \mathfrak{c}}\neq\varnothing\}\rangle_\mathbb{T}$. By Lemma \ref{L;Lemma5.1}, it is clear that $\mathbb{I}=\langle\{B_{\mathbf{a}, \mathbf{b}, \mathfrak{c}}: \mathbf{a}, \mathbf{b}\in \mathbb{E}, \mathfrak{c}\in\mathbb{U}_{\mathbf{a}, \mathbf{b}}, p\mid k_{[\mathbf{a}, \mathbf{b}, \mathbf{a}]}k_{[\mathbf{b}, \mathbf{a}, \mathbf{b}]}k_\mathfrak{c}\}\rangle_\mathbb{T}$.
\end{nota}
\begin{lem}\label{L;Lemma6.2}
Assume that $\mathbf{g}, \mathbf{h}, \mathbf{i}\in\mathbb{E}$, $\mathfrak{j}\in \mathbb{U}_{\mathbf{g}, \mathbf{h}}$, and $\mathfrak{k}\in\mathbb{U}_{\mathbf{h}, \mathbf{i}}$. If $B_{\mathbf{g}, \mathbf{h}, \mathfrak{j}}B_{\mathbf{h}, \mathbf{i}, \mathfrak{k}}\neq O$, then $\mathbb{V}_{\mathbf{g}, \mathbf{h}, \mathfrak{j}}\cup\mathbb{V}_{\mathbf{h}, \mathbf{i}, \mathfrak{k}}\subseteq\mathbb{V}_{\mathbf{g}, \mathbf{i}, (\mathbf{g}, \mathbf{h}, \mathbf{i}, \mathfrak{j}, \mathfrak{k})}$. In particular, $\mathbb{I}$ is a two-sided ideal of $\mathbb{T}$.
\end{lem}
\begin{proof}
Assume that $\mathbb{V}_{\mathbf{g}, \mathbf{h}, \mathfrak{j}}\neq\varnothing$ and $\mathbb{V}_{\mathbf{h}, \mathbf{i}, \mathfrak{k}}\neq\varnothing$. Pick $q\in\mathbb{V}_{\mathbf{g}, \mathbf{h}, \mathfrak{j}}$. If $q\in \mathbbm{g}_1\setminus(\mathbbm{h}_1\cup\mathbbm{i}_1)$, then $p\mid m_q-1$ and $q\in\mathbb{V}_{\mathbf{g}, \mathbf{i}, (\mathbf{g}, \mathbf{h}, \mathbf{i}, \mathfrak{j}, \mathfrak{k})}$. If $q\in (\mathbbm{g}_1\cap\mathbbm{i}_1)\setminus\mathbbm{h}_1$, then $q\in(\mathbbm{g}_1\cap\mathbbm{i}_1)^\circ\setminus\mathbbm{h}_1$, $p\mid m_q\!-\!1$, and $q\in\mathbb{V}_{\mathbf{g}, \mathbf{i}, (\mathbf{g}, \mathbf{h}, \mathbf{i}, \mathfrak{j}, \mathfrak{k})}$. If $q\!\in\!\mathbbm{h}_1\setminus(\mathbbm{g}_1\cup\mathbbm{i}_1)$,
then $p\mid m_q-1$, $p\mid k_{[\mathbf{g}, \mathbf{h}, \mathbf{i}]}$, and $B_{\mathbf{g}, \mathbf{h}, \mathfrak{j}}B_{\mathbf{h}, \mathbf{i}, \mathfrak{k}}=O$ by Theorem \ref{T;Theorem3.28}. If $q\!\in\!(\mathbbm{h}_1\cap\mathbbm{i}_1)\setminus\mathbbm{g}_1$, then $p\mid m_q-1$ and $q\in\mathbb{V}_{\mathbf{g}, \mathbf{i}, (\mathbf{g}, \mathbf{h}, \mathbf{i}, \mathfrak{j}, \mathfrak{k})}$. If $q\in\mathbbm{g}_2\setminus(\mathbbm{h}_2\cup\mathbbm{i}_2)$, then $p\mid(\ell_q-1)m_q$ and $q\in\mathbb{V}_{\mathbf{g}, \mathbf{i}, (\mathbf{g}, \mathbf{h}, \mathbf{i}, \mathfrak{j}, \mathfrak{k})}$. If $q\!\in\!(\mathbbm{g}_2\cap\mathbbm{i}_2)^\bullet\setminus\mathbbm{h}_2$, then $p\mid(\ell_q\!-\!1)m_q$ and $q\!\in\!\mathbb{V}_{\mathbf{g}, \mathbf{i}, (\mathbf{g}, \mathbf{h}, \mathbf{i}, \mathfrak{j}, \mathfrak{k})}$. If $q\!\in\!(\mathbbm{g}_2\cap\mathbbm{i}_2)\setminus((\mathbbm{g}_2\cap\mathbbm{i}_2)^\bullet\cup\mathbbm{h}_2)$, then $p\mid m_q$ and $q\in\mathbb{V}_{\mathbf{g}, \mathbf{i}, (\mathbf{g}, \mathbf{h}, \mathbf{i}, \mathfrak{j}, \mathfrak{k})}$. If $q\in\mathbbm{h}_2\setminus(\mathbbm{g}_2\cup\mathbbm{i}_2)$, then $p\mid (\ell_q\!-\!1)m_q$, $p\mid k_{[\mathbf{g}, \mathbf{h}, \mathbf{i}]}$, and $B_{\mathbf{g}, \mathbf{h}, \mathfrak{j}}B_{\mathbf{h}, \mathbf{i}, \mathfrak{k}}\!=\!O$ by Theorem \ref{T;Theorem3.28}. If $q\in(\mathbbm{h}_2\cap\mathbbm{i}_2)\setminus\mathbbm{g}_2$, then $p\mid(\ell_q-1)m_q$ and $q\!\in\!\mathbb{V}_{\mathbf{g}, \mathbf{i}, (\mathbf{g}, \mathbf{h}, \mathbf{i}, \mathfrak{j}, \mathfrak{k})}$. Assume that $\mathfrak{j}\!=\!(\mathbbm{j}_{(1)}, \mathbbm{j}_{(2)}, \mathbbm{j}_{(3)})$ and $\mathfrak{k}=(\mathbbm{k}_{(1)}, \mathbbm{k}_{(2)}, \mathbbm{k}_{(3)})$. If $q\!\in\!\mathbbm{j}_{(1)}\!\setminus\!\mathbbm{i}_1$, then $p\mid m_q-1$, $p\mid k_{\mathfrak{j}\setminus \mathbf{i}}$, and $B_{\mathbf{g}, \mathbf{h}, \mathfrak{j}}B_{\mathbf{h}, \mathbf{i}, \mathfrak{k}}\!=\!O$ by Theorem \ref{T;Theorem3.28}. If $q\!\in\!\mathbbm{i}_1\cap\mathbbm{j}_{(1)}$, then
$q\in\mathbbm{g}_1\!\cap\!\mathbbm{i}_1\!\cap\!(\mathbbm{j}_{(1)}\cup\mathbbm{k}_{(1)})$, $p\mid m_q-1$, and $q\!\in\!\mathbb{V}_{\mathbf{g}, \mathbf{i}, (\mathbf{g}, \mathbf{h}, \mathbf{i}, \mathfrak{j}, \mathfrak{k})}$. If $q\in\mathbbm{j}_{(2)}\setminus\mathbbm{i}_2$, then $p\mid (\ell_q\!-\!1)m_q$, $p\mid k_{\mathfrak{j}\setminus \mathbf{i}}$, and
$B_{\mathbf{g}, \mathbf{h}, \mathfrak{j}}B_{\mathbf{h}, \mathbf{i}, \mathfrak{k}}=O$ by Theorem \ref{T;Theorem3.28}. If $q\in\mathbbm{i}_2\cap\mathbbm{j}_{(2)}$, then $q\in\mathbbm{g}_2\cap\mathbbm{i}_2\cap(\mathbbm{j}_{(2)}\cup\mathbbm{k}_{(2)})$,
$p\mid (\ell_q-1)m_q$, and $q\!\in\!\mathbb{V}_{\mathbf{g}, i, (\mathbf{g}, \mathbf{h}, \mathbf{i}, \mathfrak{j}, \mathfrak{k})}$.
If $q\!\in\!\mathbbm{j}_{(3)}\!\setminus\!(\mathbbm{i}_2\!\cup\!\mathbbm{j}_{(2)})$, then $p\mid m_q$, $p\mid k_{\mathfrak{j}\setminus \mathbf{i}}$, and $B_{\mathbf{g}, \mathbf{h}, \mathfrak{j}}B_{\mathbf{h}, \mathbf{i}, \mathfrak{k}}=O$ by Theorem \ref{T;Theorem3.28}.
If $q\in(\mathbbm{j}_{(3)}\cap\mathbbm{k}_{(3)})\setminus\mathbbm{j}_{(2)}$, then $p\mid m_q$, $p\mid k_{\mathfrak{j}\cap\mathfrak{k}}$, and $B_{\mathbf{g}, \mathbf{h}, \mathfrak{j}}B_{\mathbf{h}, \mathbf{i}, \mathfrak{k}}=O$ by Theorem \ref{T;Theorem3.28}. If $q\in(\mathbbm{i}_2\cap\mathbbm{j}_{(3)})\setminus(\mathbbm{j}_{(2)}\cup\mathbbm{k}_{(3)})$, then $q\in \mathbbm{g}_2\cap\mathbbm{i}_2\cap((\mathbbm{j}_{(3)}\cup\mathbbm{k}_{(3)})\setminus(\mathbbm{j}_{(2)}\cup\mathbbm{k}_{(2)}))$, $p\mid m_q$, and $q\in\mathbb{V}_{\mathbf{g}, \mathbf{i}, (\mathbf{g}, \mathbf{h}, \mathbf{i}, \mathfrak{j}, \mathfrak{k})}$. In conclusion, $\mathbb{V}_{\mathbf{g}, \mathbf{h}, \mathfrak{j}}\subseteq\mathbb{V}_{\mathbf{g}, \mathbf{i}, (\mathbf{g}, \mathbf{h}, \mathbf{i}, \mathfrak{j}, \mathfrak{k})}$ as $B_{\mathbf{g}, \mathbf{h}, \mathfrak{j}}B_{\mathbf{h}, \mathbf{i}, \mathfrak{k}}\neq O$. As $\mathbb{V}_{\mathbf{g}, \mathbf{h}, \mathfrak{j}}=\mathbb{V}_{\mathbf{h}, \mathbf{g}, \mathfrak{j}}$, $\mathbb{V}_{\mathbf{h}, \mathbf{i}, \mathfrak{k}}=\mathbb{V}_{\mathbf{i}, \mathbf{h}, \mathfrak{k}}$, and $\mathbb{V}_{\mathbf{g}, \mathbf{i}, (\mathbf{g}, \mathbf{h}, \mathbf{i}, \mathfrak{j}, \mathfrak{k})}=\mathbb{V}_{\mathbf{i}, \mathbf{g}, (\mathbf{i}, \mathbf{h}, \mathbf{g}, \mathfrak{k}, \mathfrak{j})}$, the first statement thus follows from Lemma \ref{L;Lemma3.9}. The desired lemma thus follows from the first statement and Theorem \ref{T;Theorem3.13}.
\end{proof}
\begin{lem}\label{L;Lemma6.3}
Assume that $\mathbb{I}\neq\{O\}$. Then there is a nonzero matrix product of\ \ $2|\{a: a\!\in\! [1, n], p\mid (\ell_a-1)(m_a-1)m_a\}|$-many matrices in $\mathbb{I}$. 
\end{lem}
\begin{proof}
Set $\mathbb{U}=\{a: a\!\in\! [1, n], p\mid (\ell_a-1)(m_a-1)m_a\}$. Since $\mathbb{I}\neq\{O\}$, notice that $\mathbb{U}\neq\varnothing$. By Lemma \ref{L;Lemma3.1}, there is $\mathbf{g}\in \mathbb{E}$ such that $\mathbbm{g}_1=\{a: a\in [1,n], p\mid m_a-1\}$ and $\mathbbm{g}_2=\{a: a\in [1,n], p\nmid m_a-1, p\mid (\ell_a-1)m_a\}$. Let $h\in \mathbb{N}$ and $\mathbb{U}=\{i_1, i_2, \ldots, i_h\}$. If $j\in\mathbbm{g}_1$, define $\mathfrak{k}_j=(\{j\},\varnothing, \varnothing)$. If $j\in {\mathbbm{g}_2}^\bullet$, define $\mathfrak{k}_j=(\varnothing, \{j\}, \{j\})$. If $j\in \mathbbm{g}_2\setminus{\mathbbm{g}_2}^\bullet$, define $\mathfrak{k}_j=(\varnothing, \varnothing, \{j\})$. So $\mathfrak{k}_\ell\!\in\! \mathbb{U}_{\mathbf{g}, \mathbf{g}}$ and $\mathbb{V}_{\mathbf{g}, \mathbf{g}, \mathfrak{k}_\ell}\!\neq\! \varnothing$ for any $\ell\!\in\! \{i_1, i_2, \ldots, i_h\}$. Notice that $B_{\mathbf{g}, \mathbf{g}, \mathfrak{k}_{i_1}}, B_{\mathbf{g}, \mathbf{g}, \mathfrak{k}_{i_2}}, \ldots, B_{\mathbf{g}, \mathbf{g}, \mathfrak{k}_{i_h}}\!\in\!\mathbb{I}$ and $B_{\mathbf{g}, \mathbf{g}, \mathfrak{k}_{i_1}}B_{\mathbf{g}, \mathbf{g}, \mathfrak{k}_{i_2}}\cdots B_{\mathbf{g}, \mathbf{g}, \mathfrak{k}_{i_h}}\neq O$ by Lemma \ref{L;Lemma4.6}.

If $m\!\in\! \mathbbm{g}_1$, Lemma \ref{L;Lemma3.1} gives $\mathbf{q}\in\mathbb{E}$ satisfying $\mathbbm{q}_1=\mathbbm{g}_1\setminus\{m\}$ and $\mathbbm{q}_2=\mathbbm{g}_2$. Define $\mathbf{r}_{(\mathbf{g}, m)}\!=\!\mathbf{q}$. If $m\!\in\! \mathbbm{g}_2$, Lemma \ref{L;Lemma3.1} gives $\mathbf{s}\!\!\in\!\!\mathbb{E}$ satisfying $\mathbbm{s}_1\!=\!\mathbbm{g}_1$ and $\mathbbm{s}_2=\mathbbm{g}_2\!\setminus\!\{m\}$. Define $\mathbf{r}_{(\mathbf{g}, m)}\!\!=\!\!\mathbf{s}$. For any $t\!\in\!\{i_1, i_2, \ldots, i_h\}$, $\mathbb{V}_{\mathbf{g}, \mathbf{r}_{(\mathbf{g}, t)}, \mathfrak{o}}\!\!=\!\!\mathbb{V}_{\mathbf{r}_{(\mathbf{g}, t)}, \mathbf{g}, \mathfrak{o}}\!\neq\!\varnothing$ and $B_{\mathbf{g}, \mathbf{r}_{(\mathbf{g},t)}, \mathfrak{o}}, B_{\mathbf{r}_{(\mathbf{g}, t)}, \mathbf{g}, \mathfrak{o}}\!\!\in\!\!\mathbb{I}$. For any $t\!\in\!\{i_1, i_2, \ldots, i_h\}$, $k_{[\mathbf{g}, \mathbf{r}_{(\mathbf{g}, t)}, \mathbf{g}]}\!=\!k_{\mathfrak{o}\setminus \mathbf{g}}\!=\!k_\mathfrak{o}=1$ and $B_{\mathbf{g}, \mathbf{r}_{(\mathbf{g},t)}, \mathfrak{o}}B_{\mathbf{r}_{(\mathbf{g}, t)}, \mathbf{g}, \mathfrak{o}}\!=\!B_{\mathbf{g}, \mathbf{g}, \mathfrak{k}_t}$ by Theorem \ref{T;Theorem3.28}. So $B_{\mathbf{g}, \mathbf{r}_{(\mathbf{g},i_1)}, \mathfrak{o}}B_{\mathbf{r}_{(\mathbf{g}, i_1)}, \mathbf{g}, \mathfrak{o}}B_{\mathbf{g}, \mathbf{r}_{(\mathbf{g},i_2)}, \mathfrak{o}}B_{\mathbf{r}_{(\mathbf{g},i_2)}, \mathbf{g}, \mathfrak{o}}\cdots B_{\mathbf{g}, \mathbf{r}_{(\mathbf{g}, i_h)}, \mathfrak{o}}B_{\mathbf{r}_{(\mathbf{g}, i_h)}, \mathbf{g}, \mathfrak{o}}\neq O$. The desired lemma thus follows as the choices are clear by the above discussion.
\end{proof}
The following lemmas allow us to show that $\mathbb{I}$ is a nilpotent two-sided ideal of $\mathbb{T}$.
\begin{lem}\label{L;Lemma6.4}
Assume that $\mathbf{g}, \mathbf{h}, \mathbf{i}, \mathbf{j}\in \mathbb{E}$, $\mathfrak{k}\in\mathbb{U}_{\mathbf{g}, \mathbf{h}}$, $\mathfrak{l}\in\mathbb{U}_{\mathbf{h}, \mathbf{i}}$, and $\mathfrak{m}\in\mathbb{U}_{\mathbf{i},\mathbf{j}}$. Assume that $\mathbb{V}_{\mathbf{g}, \mathbf{h}, \mathfrak{k}}\cap\mathbb{V}_{\mathbf{h}, \mathbf{i}, \mathfrak{l}}\cap\mathbb{V}_{\mathbf{i}, \mathbf{j}, \mathfrak{m}}\neq\varnothing$. Then $B_{\mathbf{g}, \mathbf{h}, \mathfrak{k}}B_{\mathbf{h}, \mathbf{i}, \mathfrak{l}}B_{\mathbf{i},\mathbf{j}, \mathfrak{m}}=O$.
\end{lem}
\begin{proof}
Assume that $\mathfrak{k}\!\!=\!\!(\mathbbm{k}_{(1)}, \mathbbm{k}_{(2)}, \mathbbm{k}_{(3)})$, $\mathfrak{l}\!\!=\!\!(\mathbbm{l}_{(1)}, \mathbbm{l}_{(2)}, \mathbbm{l}_{(3)})$, $\mathfrak{m}\!\!=\!\!(\mathbbm{m}_{(1)}, \mathbbm{m}_{(2)}, \mathbbm{m}_{(3)})$. Pick $q\!\!\in\!\!\mathbb{V}_{\mathbf{g}, \mathbf{h}, \mathfrak{k}}\cap\mathbb{V}_{\mathbf{h}, \mathbf{i}, \mathfrak{l}}\!\cap\!\mathbb{V}_{\mathbf{i}, \mathbf{j}, \mathfrak{m}}$. If $q\!\in\!\mathbbm{h}_1\!\setminus\!(\mathbbm{g}_1\cup\mathbbm{i}_1)$, then $p\mid m_q-1$ and $p\mid k_{[\mathbf{g}, \mathbf{h}, \mathbf{i}]}$. If $q\!\in\!\mathbbm{k}_{(1)}\setminus\mathbbm{i}_1$, then $p\!\mid\! m_q\!-\!1$ and $p\mid k_{\mathfrak{k}\setminus \mathbf{i}}$. If $q\!\in\!\mathbbm{i}_1\setminus(\mathbbm{h}_1\!\cup\!\mathbbm{j}_1)$, then $p\!\mid\! m_q\!-\!1$ and $p\mid k_{[\mathbf{h},\mathbf{i},\mathbf{j}]}$. If $q\!\in\!\mathbbm{m}_{(1)}\!\setminus\!\mathbbm{h}_1$, then $p\!\mid\! m_q\!-\!1$ and $p\mid k_{\mathfrak{m}\setminus \mathbf{h}}$. If $q\in\mathbbm{h}_2\setminus(\mathbbm{g}_2\cup\mathbbm{i}_2)$, then $p\mid (\ell_q-1)m_q$ and $p\mid k_{[\mathbf{g}, \mathbf{h}, \mathbf{i}]}$. If $q\in\mathbbm{k}_{(2)}\setminus \mathbbm{i}_2$, then $p\mid (\ell_q-1)m_q$ and $p\mid k_{\mathfrak{k}\setminus \mathbf{i}}$. If $q\in \mathbbm{k}_{(3)}\setminus(\mathbbm{i}_2\cup\mathbbm{k}_{(2)})$, then $p\mid m_q$ and $p\mid k_{\mathfrak{k}\setminus \mathbf{i}}$. If $q\in\mathbbm{i}_2\setminus(\mathbbm{h}_2\cup\mathbbm{j}_2)$, then $p\mid (\ell_q-1)m_q$ and $p\mid k_{[\mathbf{h}, \mathbf{i}, \mathbf{j}]}$. If $q\in \mathbbm{m}_{(2)}\setminus\mathbbm{h}_2$, then $p\mid (\ell_q-1)m_q$ and $p\mid k_{\mathfrak{m}\setminus \mathbf{h}}$. If $q\in \mathbbm{m}_{(3)}\setminus(\mathbbm{h}_2\cup\mathbbm{m}_{(2)})$, then $p\mid m_q$ and $p\mid k_{\mathfrak{m}\setminus \mathbf{h}}$. If $q\in \mathbbm{l}_{(1)}\setminus\mathbbm{g}_1$, then $p\mid m_q-1$ and $p\mid k_{\mathfrak{l}\setminus \mathbf{g}}$. If $q\in\mathbbm{k}_{(1)}\cap\mathbbm{l}_{(1)}$, then $p\mid m_q\!-\!1$ and $p\mid k_{\mathfrak{k}\cap\mathfrak{l}}$. If $q\!\in\!\mathbbm{l}_{(2)}\!\setminus\!\mathbbm{g}_2$, then $p\mid (\ell_q\!-\!1)m_q$ and $p\mid k_{\mathfrak{l}\setminus \mathbf{g}}$.
If $q\!\in\!\mathbbm{l}_{(3)}\!\setminus\!(\mathbbm{g}_2\!\cup\!\mathbbm{l}_{(2)})$, then
$p\mid m_q$ and $p\mid k_{\mathfrak{l}\setminus \mathbf{g}}$. If $q\in (\mathbbm{k}_{(3)}\cap\mathbbm{l}_{(3)})\setminus(\mathbbm{k}_{(2)}\cap\mathbbm{l}_{(2)})$, then $p\mid m_q$ and $p\mid k_{\mathfrak{k}\cap\mathfrak{l}}$. If $q\in\mathbbm{k}_{(2)}\!\cap\!\mathbbm{l}_{(2)}$, then $p\mid (\ell_q-1)m_q$ and $p\mid k_{\mathfrak{k}\cap\mathfrak{l}}$. Hence $B_{\mathbf{g}, \mathbf{h}, \mathfrak{k}}B_{\mathbf{h}, \mathbf{i}, \mathfrak{l}}B_{\mathbf{i},\mathbf{j}, \mathfrak{m}}=O$ by the above discussion and Theorem \ref{T;Theorem3.28}. The desired lemma thus follows.
\end{proof}
\begin{lem}\label{L;Lemma6.5}
Assume that $g\in\mathbb{N}\setminus[1, 2]$ and $\mathfrak{h}_i\in\mathbb{U}_{\mathbf{j}_{\mathfrak{h}_i}, \mathbf{k}_{\mathfrak{h}_i}}$ for any $i\in [1, g]$ and some $\mathbf{j}_{\mathfrak{h}_i}, \mathbf{k}_{\mathfrak{h}_i}\in\mathbb{E}$. Assume that there are pairwise distinct $\ell, m, q\in [1, g]$ such that $\mathbb{V}_{\mathbf{j}_{\mathfrak{h}_\ell}, \mathbf{k}_{\mathfrak{h}_\ell}, \mathfrak{h}_\ell}\cap\mathbb{V}_{\mathbf{j}_{\mathfrak{h}_m}, \mathbf{k}_{\mathfrak{h}_m}, \mathfrak{h}_m}\cap\mathbb{V}_{\mathbf{j}_{\mathfrak{h}_q}, \mathbf{k}_{\mathfrak{h}_q}, \mathfrak{h}_q}\neq\varnothing$. Then
$B_{\mathbf{j}_{\mathfrak{h}_1}, \mathbf{k}_{\mathfrak{h}_1}, \mathfrak{h}_1}B_{\mathbf{j}_{\mathfrak{h}_2}, \mathbf{k}_{\mathfrak{h}_2}, \mathfrak{h}_2}\cdots B_{\mathbf{j}_{\mathfrak{h}_g}, \mathbf{k}_{\mathfrak{h}_g}, \mathfrak{h}_g}=O$.
\end{lem}
\begin{proof}
Assume that $B_{\mathbf{j}_{\mathfrak{h}_1}, \mathbf{k}_{\mathfrak{h}_1}, \mathfrak{h}_1}B_{\mathbf{j}_{\mathfrak{h}_2}, \mathbf{k}_{\mathfrak{h}_2}, \mathfrak{h}_2}\cdots B_{\mathbf{j}_{\mathfrak{h}_g}, \mathbf{k}_{\mathfrak{h}_g}, \mathfrak{h}_g}\!\neq\! O$. As $\ell, m, q$ are pairwise distinct, there is no loss to let $\ell<m<q$. By Equation \eqref{Eq;3} and Theorem \ref{T;Theorem3.28}, there are $r, s\!\in\!\F^\times$ such that $B_{\mathbf{j}_{\mathfrak{h}_1}, \mathbf{k}_{\mathfrak{h}_1}, \mathfrak{h}_1}B_{\mathbf{j}_{\mathfrak{h}_2}, \mathbf{k}_{\mathfrak{h}_2}, \mathfrak{h}_2}\cdots B_{\mathbf{j}_{\mathfrak{h}_{m-1}}, \mathbf{k}_{\mathfrak{h}_{m-1}}, \mathfrak{h}_{m-1}}\!=\!rB_{\mathbf{j}_{\mathfrak{h}_1}, \mathbf{k}_{\mathfrak{h}_{m-1}}, \mathfrak{t}}$ and $B_{\mathbf{j}_{\mathfrak{h}_{m+1}}, \mathbf{k}_{\mathfrak{h}_{m+1}}, \mathfrak{h}_{m+1}}B_{\mathbf{j}_{\mathfrak{h}_{m+2}}, \mathbf{k}_{\mathfrak{h}_{m+2}}, \mathfrak{h}_{m+2}}\cdots B_{\mathbf{j}_{\mathfrak{h}_g}, \mathbf{k}_{\mathfrak{h}_g}, \mathfrak{h}_g}=sB_{\mathbf{j}_{\mathfrak{h}_{m+1}}, \mathbf{k}_{\mathfrak{h}_g}, \mathfrak{u}}$ for $\mathfrak{t}\in\mathbb{U}_{\mathbf{j}_{\mathfrak{h}_1},\mathbf{k}_{\mathfrak{h}_{m-1}}}$ and $\mathfrak{u}\in\mathbb{U}_{\mathbf{j}_{\mathfrak{h}_{m+1}}, \mathbf{k}_{\mathfrak{h}_g}}$. So  $\mathbb{V}_{\mathbf{j}_{\mathfrak{h}_\ell}, \mathbf{k}_{\mathfrak{h}_\ell}, \mathfrak{h}_\ell}\!\subseteq\!\mathbb{V}_{\mathbf{j}_{\mathfrak{h}_1}, \mathbf{k}_{\mathfrak{h}_{m-1}}, \mathfrak{t}}$ and
$\mathbb{V}_{\mathbf{j}_{\mathfrak{h}_q}, \mathbf{k}_{\mathfrak{h}_q}, \mathfrak{h}_q}\!\subseteq\!\mathbb{V}_{\mathbf{j}_{\mathfrak{h}_{m+1}}, \mathbf{k}_{\mathfrak{h}_g}, \mathfrak{u}}$ by Lemma \ref{L;Lemma6.2}. These containments imply that $\mathbb{V}_{\mathbf{j}_{\mathfrak{h}_1}, \mathbf{k}_{\mathfrak{h}_{m-1}}, \mathfrak{t}}\!\cap\!\mathbb{V}_{\mathbf{j}_{\mathfrak{h}_m}, \mathbf{k}_{\mathfrak{h}_m}, \mathfrak{h}_m}\!\cap\!\mathbb{V}_{\mathbf{j}_{\mathfrak{h}_{m+1}}, \mathbf{k}_{\mathfrak{h}_g}, \mathfrak{u}}\!\neq\!\varnothing$.
Hence
$$
O\neq B_{\mathbf{j}_{\mathfrak{h}_1}, \mathbf{k}_{\mathfrak{h}_1}, \mathfrak{h}_1}B_{\mathbf{j}_{\mathfrak{h}_2}, \mathbf{k}_{\mathfrak{h}_2}, \mathfrak{h}_2}\cdots B_{\mathbf{j}_{\mathfrak{h}_g}, \mathbf{k}_{\mathfrak{h}_g}, \mathfrak{h}_g}=rsB_{\mathbf{j}_{\mathfrak{h}_1}, \mathbf{k}_{\mathfrak{h}_{m-1}}, \mathfrak{t}}B_{\mathbf{j}_{\mathfrak{h}_m}, \mathbf{k}_{\mathfrak{h}_m}, \mathfrak{h}_m}B_{\mathbf{j}_{\mathfrak{h}_{m+1}}, \mathbf{k}_{\mathfrak{h}_g}, \mathfrak{u}}=O
$$
by the above discussion and Lemma \ref{L;Lemma6.4}. This is a contradiction. This contradiction implies that $B_{\mathbf{j}_{\mathfrak{h}_1}, \mathbf{k}_{\mathfrak{h}_1}, \mathfrak{h}_1}B_{\mathbf{j}_{\mathfrak{h}_2}, \mathbf{k}_{\mathfrak{h}_2}, \mathfrak{h}_2}\cdots B_{\mathbf{j}_{\mathfrak{h}_g}, \mathbf{k}_{\mathfrak{h}_g}, \mathfrak{h}_g}=O$.
The desired lemma follows.
\end{proof}
\begin{lem}\label{L;Lemma6.6}
Assume that $g\in \mathbb{N}$ and $\mathbb{U}$ is a set with cardinality $g$. Assume that $\mathbb{V}_1, \mathbb{V}_2, \ldots, \mathbb{V}_{2g+1}$ are nonempty subsets of $\mathbb{U}$. Then there exist pairwise distinct $h, i, j\in[1, 2g+1]$ such that $\mathbb{V}_h\cap\mathbb{V}_i\cap\mathbb{V}_j\neq \varnothing$.
\end{lem}
\begin{proof}
Work by induction on $g$. If $g=1$, then $\mathbb{V}_1\cap\mathbb{V}_2\cap\mathbb{V}_3=\mathbb{U}\neq\varnothing$. The base case is thus checked. Assume that $g>1$ and any sequence of $(2g-1)$-many nonempty subsets of a set with cardinality $g-1$ has a nonempty intersection of three members.

By the Pigeonhole Principle, notice that there exists $k\in\mathbb{U}$ such that $k\in\mathbb{V}_\ell\cap\mathbb{V}_m$ for some distinct $\ell, m\in[1, 2g+1]$. There is no loss to require that $\ell<m$. If there exists $q\in[1, 2g+1]\setminus\{\ell, m\}$ such that $k\in\mathbb{V}_q$, it is obvious that $\mathbb{V}_\ell\cap\mathbb{V}_m\cap\mathbb{V}_q\neq\varnothing$.

Assume further that $k\notin\mathbb{V}_r$ for any $r\in [1, 2g+1]\setminus\{\ell, m\}$. Therefore  $\mathbb{U}\setminus\{k\}$ has $(2g-1)$-many nonempty subsets $\mathbb{V}_1, \mathbb{V}_2, \ldots, \mathbb{V}_{\ell-1}, \mathbb{V}_{\ell+1}, \ldots, \mathbb{V}_{m-1}, \mathbb{V}_{m+1},\ldots, \mathbb{V}_{2g+1}$. By the inductive hypothesis, there are pairwise distinct $s, t, u\in[1, 2g+1]\setminus\{\ell, m\}$ such that $\mathbb{V}_s\cap \mathbb{V}_t\cap\mathbb{V}_u\neq\varnothing$. The desired lemma follows from the above discussion.
\end{proof}
\begin{lem}\label{L;Lemma6.7}
There are not $(2|\{a: a\in [1, n], p\mid (\ell_a-1)(m_a-1)m_a\}|+1)$-many matrices in $\mathbb{I}$ such that their matrix product is nonzero. In particular, $\mathbb{I}$ is a nilpotent two-sided ideal of $\mathbb{T}$ with nilpotent index $2|\{a: a\in [1, n], p\mid (\ell_a-1)(m_a-1)m_a\}|+1$.
\end{lem}
\begin{proof}
Set $\mathbb{U}=\{a: a\in [1, n], p\mid (\ell_a-1)(m_a-1)m_a\}$ and $g=|\mathbb{U}|$. If $g=0$, then  $\mathrm{Rad}(\mathbb{T})=\mathbb{I}=\{O\}$ by Theorem \ref{T;Semisimplicity} and Lemma \ref{L;Lemma2.4}. Assume that $g>0$.
Then $\mathbb{I}\neq\{O\}$ by Theorem \ref{T;Semisimplicity} and Lemma \ref{L;Lemma3.19}. Let $\mathfrak{h}_i\in\mathbb{U}_{\mathbf{j}_{\mathfrak{h}_i}, \mathbf{k}_{\mathfrak{h}_i}}$ and $\mathbb{V}_{\mathbf{j}_{\mathfrak{h}_i}, \mathbf{k}_{\mathfrak{h}_i}, \mathfrak{h}_i}\neq\varnothing$ for any $i\in [1, 2g+1]$ and some
$\mathbf{j}_{\mathfrak{h}_i}, \mathbf{k}_{\mathfrak{h}_i}\in\mathbb{E}$. Then $\mathbb{V}_{\mathbf{j}_{\mathfrak{h}_i}, \mathbf{k}_{\mathfrak{h}_i}, \mathfrak{h}_i}\subseteq\mathbb{U}$ for any $i\in[1, 2g+1]$. By Lemma \ref{L;Lemma6.6}, $\mathbb{V}_{\mathbf{j}_{\mathfrak{h}_\ell}, \mathbf{k}_{\mathfrak{h}_\ell}, \mathfrak{h}_\ell}\cap\mathbb{V}_{\mathbf{j}_{\mathfrak{h}_m}, \mathbf{k}_{\mathfrak{h}_m}, \mathfrak{h}_m}\cap\mathbb{V}_{\mathbf{j}_{\mathfrak{h}_q}, \mathbf{k}_{\mathfrak{h}_q}, \mathfrak{h}_q}\!\neq\!\varnothing$
for pairwise distinct $\ell, m, q$ in $[1, 2g+1]$. The desired lemma follows from combining Lemmas \ref{L;Lemma6.5}, \ref{L;Lemma6.2},
\ref{L;Lemma6.3}.
\end{proof}
We are now ready to list the main result of this section and an additional corollary.
\begin{thm}\label{T;Jacobson}
Assume that $M\in\mathrm{Rad}(\mathbb{T})$. Then $M\in\mathbb{I}$. In particular, $\mathrm{Rad}(\mathbb{T})=\mathbb{I}$ and the nilpotent index of $\mathrm{Rad}(\mathbb{T})$ is $2|\{a: a\in [1, n], p\mid (\ell_a-1)(m_a-1)m_a\}|+1$.
\end{thm}
\begin{proof}
Assume that $M\in\mathrm{Rad}(\mathbb{T})\setminus \mathbb{I}$. By Equation \eqref{Eq;4}, there are $\mathbf{g}, \mathbf{h}\in \mathbb{E}$ such that $E_\mathbf{g}^*ME_\mathbf{h}^*\!\in\!\mathrm{Rad}(\mathbb{T})\setminus \mathbb{I}$. By Lemma \ref{L;Lemma6.7} and Theorem \ref{T;Theorem3.13}, there is no loss to assume that $\mathrm{Supp}_{\mathbb{B}_2}(E_\mathbf{g}^*ME_\mathbf{h}^*)\cap\mathbb{I}=\varnothing$. By Equation \eqref{Eq;3} and Theorem \ref{T;Theorem3.13}, there are $i\in\mathbb{N}$ and pairwise distinct $\mathfrak{j}_1, \mathfrak{j}_2, \ldots, \mathfrak{j}_i\!\in\!\mathbb{U}_{\mathbf{g}, \mathbf{h}}$ such that $p\nmid k_{[\mathbf{g}, \mathbf{h}, \mathbf{g}]}k_{[\mathbf{h}, \mathbf{g}, \mathbf{h}]}k_{\mathfrak{j}_k}$ and $\mathrm{Supp}_{\mathbb{B}_2}(E_\mathbf{g}^*ME_\mathbf{h}^*)=\{B_{\mathbf{g}, \mathbf{h}, \mathfrak{j}_1}, B_{\mathbf{g}, \mathbf{h}, \mathfrak{j}_2}, \ldots, B_{\mathbf{g}, \mathbf{h}, \mathfrak{j}_i}\}$ for any $k\in [1, i]$. According to Theorem \ref{T;Theorem3.13}, it is obvious to see that $B_{\mathbf{g}, \mathbf{h}, \mathfrak{j}_1}, B_{\mathbf{g}, \mathbf{h}, \mathfrak{j}_2}, \ldots, B_{\mathbf{g}, \mathbf{h}, \mathfrak{j}_i}$ are pairwise distinct.

The combination of Equation \eqref{Eq;3}, Theorem \ref{T;Theorem5.17}, and Lemma \ref{L;Lemma2.4} thus implies that $E_\mathbf{g}^*ME_\mathbf{h}^*B_{\mathbf{h}, \mathbf{g}, \mathfrak{o}}\!\in\!\mathrm{Rad}(\mathbb{T})\cap E_\mathbf{g}^*\mathbb{T}E_\mathbf{g}^*\!\subseteq\!\mathrm{Rad}(E_\mathbf{g}^*\mathbb{T}E_\mathbf{g}^*)$. As $\mathfrak{j}_1, \mathfrak{j}_2, \ldots, \mathfrak{j}_i$ are pairwise distinct, $(\mathbf{g}, \mathbf{h}, \mathbf{g}, \mathfrak{j}_1, \mathfrak{o}), (\mathbf{g}, \mathbf{h}, \mathbf{g}, \mathfrak{j}_2, \mathfrak{o}), \ldots, (\mathbf{g}, \mathbf{h}, \mathbf{g}, \mathfrak{j}_i, \mathfrak{o})$ are also pairwise distinct by a direct computation. So $B_{\mathbf{g}, \mathbf{g}, (\mathbf{g}, \mathbf{h}, \mathbf{g}, \mathfrak{j}_1, \mathfrak{o})}, B_{\mathbf{g}, \mathbf{g}, (\mathbf{g}, \mathbf{h}, \mathbf{g}, \mathfrak{j}_2, \mathfrak{o})}, \ldots, B_{\mathbf{g}, \mathbf{g}, (\mathbf{g}, \mathbf{h}, \mathbf{g}, \mathfrak{j}_i, \mathfrak{o})}$ are pairwise distinct and $\mathrm{Supp}_{\mathbb{B}_2}(E_\mathbf{g}^*ME_\mathbf{h}^*B_{\mathbf{h}, \mathbf{g}, \mathfrak{o}})\!\!\subseteq\!\!\{B_{\mathbf{g}, \mathbf{g}, (\mathbf{g}, \mathbf{h}, \mathbf{g}, \mathfrak{j}_1, \mathfrak{o})}, B_{\mathbf{g}, \mathbf{g}, (\mathbf{g}, \mathbf{h}, \mathbf{g}, \mathfrak{j}_2, \mathfrak{o})}, \ldots, B_{\mathbf{g}, \mathbf{g}, (\mathbf{g}, \mathbf{h}, \mathbf{g}, \mathfrak{j}_i, \mathfrak{o})}\!\}$ by Theorems \ref{T;Theorem3.13} and \ref{T;Theorem3.28}. Lemmas \ref{L;Lemma3.20} and \ref{L;Lemma3.19} imply that $k_{(\mathbf{g}, \mathbf{h}, \mathbf{g}, \mathfrak{j}_k, \mathfrak{o})}=k_{[\mathbf{h}, \mathbf{g}, \mathbf{h}]}k_{\mathfrak{j}_k}$ for any $k\in [1, i]$. In particular, $p\nmid k_{(\mathbf{g}, \mathbf{h}, \mathbf{g}, \mathfrak{j}_k, \mathfrak{o})}$ for any $k\!\in\![1, i]$. This thus forces that $E_\mathbf{g}^*ME_\mathbf{h}^*B_{\mathbf{h}, \mathbf{g}, \mathfrak{o}}=O$ by Theorems \ref{T;Theorem3.13} and \ref{T;Theorem5.17}. According to Theorem \ref{T;Theorem3.28}, notice that $c_{\mathbf{g}, \mathbf{g}, (\mathbf{g}, \mathbf{h}, \mathbf{g}, \mathfrak{j}_1, \mathfrak{o})}(E_\mathbf{g}^*ME_\mathbf{h}^*B_{\mathbf{h}, \mathbf{g}, \mathfrak{o}})\!=\!c_{\mathbf{g}, \mathbf{h}, \mathfrak{j}_1}(E_\mathbf{g}^*ME_\mathbf{h}^*)\overline{k_{[\mathbf{g}, \mathbf{h}, \mathbf{g}]}}\!\in\!\F^\times$. So $E_\mathbf{g}^*ME_\mathbf{h}^*B_{\mathbf{h}, \mathbf{g}, \mathfrak{o}}\!\neq\! O$ by Theorem \ref{T;Theorem3.13}. The contradiction $O\!=\!E_\mathbf{g}^*ME_\mathbf{h}^*B_{\mathbf{h}, \mathbf{g}, \mathfrak{o}}\!\neq\! O$ shows that $M\!\in\!\mathbb{I}$. As $M$ is chosen from $\mathrm{Rad}(\mathbb{T})$ arbitrarily, the desired theorem follows from Lemma \ref{L;Lemma6.7}.
\end{proof}
\begin{cor}\label{C;Corollary6.9}
$\mathrm{Rad}(\mathrm{Z}(\mathbb{T}))$ has an $\F$-basis $\{C_\mathfrak{a}: \exists\ \mathbf{b}\in \mathbb{E}, \mathfrak{a}\in\mathbb{U}_{\mathbf{b}, \mathbf{b}}, p\mid k_\mathfrak{a}\}$.
\end{cor}
\begin{proof}
Notice that $\{C_\mathfrak{a}: \exists\ \mathbf{b}\in\mathbb{E}, \mathfrak{a}\in\mathbb{U}_{\mathbf{b}, \mathbf{b}}, p\mid k_\mathfrak{a}\}\subseteq\mathrm{Rad}(\mathrm{Z}(\mathbb{T}))$ by Lemma \ref{L;Lemma4.13}.
For any $\mathbf{g}\!\in\! \mathbb{E}$, $\mathfrak{h}\in\mathbb{U}_{\mathbf{g},\mathbf{g}}$, and $p\nmid k_\mathfrak{h}$,
$\mathrm{Supp}_{\mathbb{B}_2}(C_\mathfrak{h})\!\subseteq\!\{B_{\mathbf{a}, \mathbf{a}, \mathfrak{b}}\!: \mathbf{a}\!\in\! \mathbb{E}, \mathfrak{b}\!\in\!\mathbb{U}_{\mathbf{a}, \mathbf{a}}, p\nmid k_\mathfrak{b}\}$ by Lemma \ref{L;Lemma3.20}. As $\mathrm{Rad}(\mathrm{Z}(\mathbb{T}))=\mathrm{Z}(\mathbb{T})\cap\mathrm{Rad}(\mathbb{T})$, Theorems \ref{T;Jacobson} and \ref{T;Theorem3.13} yield that any nonzero $\F$-linear combination of the matrices in $\{C_\mathfrak{a}: \exists\ \mathbf{b}\in \mathbb{E}, \mathfrak{a}\in\mathbb{U}_{\mathbf{b}, \mathbf{b}}, p\nmid k_\mathfrak{a}\}$ is not in $\mathrm{Rad}(\mathrm{Z}(\mathbb{T}))$. The desired corollary thus follows from Theorem \ref{T;Center}.
\end{proof}
We are now ready to end this section by listing an example of Theorem \ref{T;Jacobson}.
\begin{eg}\label{E;Example6.10}
Assume that $n=1$. If $p\mid(\ell_1-1)(m_1-1)m_1$, Theorem \ref{T;Jacobson} shows that the nilpotent index of $\mathrm{Rad}(\mathbb{T})$ is three. If $p\nmid(\ell_1-1)(m_1-1)m_1$, the nilpotent index of $\mathrm{Rad}(\mathbb{T})$ equals one by Lemma \ref{L;Lemma2.4}. Assume further that $n=\ell_1=2$ and $\ell_2=m_1=m_2=3$. If $p\in [2,3]$, Theorem \ref{T;Jacobson} shows that the nilpotent index of $\mathrm{Rad}(\mathbb{T})$ is five. Otherwise, the nilpotent index of $\mathrm{Rad}(\mathbb{T})$ equals one by Lemma  \ref{L;Lemma2.4}.
\end{eg}
\section{Algebraic structure of $\mathbb{T}$: Quotient $\F$-algebra}
In this section, we present an $\F$-basis for $\mathbb{T}/\mathrm{Rad}(\mathbb{T})$ and determine the structure constants of this $\F$-basis in $\mathbb{T}/\mathrm{Rad}(\mathbb{T})$. For this purpose, we recall Notations \ref{N;Notation3.7}, \ref{N;Notation3.8}, \ref{N;Notation3.14}, \ref{N;Notation3.15}, \ref{N;Notation3.18}, \ref{N;Notation5.6}, \ref{N;Notation5.7}. By Notation \ref{N;Notation6.1} and Theorem \ref{T;Jacobson}, we recall that $\mathrm{Rad}(\mathbb{T})\!=\!\langle\{B_{\mathbf{a}, \mathbf{b}, \mathfrak{c}}: \mathbf{a}, \mathbf{b}\!\in\! \mathbb{E}, \mathfrak{c}\!\in\!\mathbb{U}_{\mathbf{a}, \mathbf{b}}, p\mid k_{[\mathbf{a}, \mathbf{b}, \mathbf{a}]}k_{[\mathbf{b}, \mathbf{a}, \mathbf{b}]}k_\mathfrak{c}\}\rangle_{\mathbb{T}}$. We display three lemmas.
\begin{lem}\label{L;Lemma7.1}
Assume that $\mathbf{g}, \mathbf{h}\in \mathbb{E}$, $\mathfrak{i}\in\mathbb{U}_{\mathbf{g}, \mathbf{h}}$, and $p\nmid k_{[\mathbf{g}, \mathbf{h}, \mathbf{g}]}k_{[\mathbf{h}, \mathbf{g}, \mathbf{h}]}$. Then 
$$\overline{k_{[\mathbf{g}, \mathbf{h}, \mathbf{g}]}}D_{\mathbf{g}, \mathbf{h}, \mathfrak{i}}^T=\overline{k_{[\mathbf{h}, \mathbf{g}, \mathbf{h}]}}D_{\mathbf{h}, \mathbf{g}, \mathfrak{i}}.$$
\end{lem}
\begin{proof}
As $(\mathbf{g}, \mathbf{h}; \mathfrak{i})=(\mathbf{h}, \mathbf{g}; \mathfrak{i})$, $n_{\mathbf{g}, \mathbf{h}, \mathfrak{i}}=n_{\mathbf{h}, \mathbf{g}, \mathfrak{i}}$, and $\mathbb{U}_{\mathbf{g}, \mathbf{h}, \mathfrak{i}, j}\!=\!\mathbb{U}_{\mathbf{h}, \mathbf{g}, \mathfrak{i}, j}$ for any $j\in[0, n_{\mathbf{g}, \mathbf{h}, \mathfrak{i}}]$,
the desired lemma thus follows from Equation \eqref{Eq;12} and Lemma \ref{L;Lemma3.9}.
\end{proof}
\begin{lem}\label{L;Lemma7.2}
Assume that $\mathbf{g}, \mathbf{h}, \mathbf{i}, \mathbf{j}\in \mathbb{E}$, $\mathfrak{k}\in \mathbb{U}_{\mathbf{g}, \mathbf{h}}$, $\mathfrak{l}\in \mathbb{U}_{\mathbf{i}, \mathbf{j}}$. If $p\nmid k_{[\mathbf{g}, \mathbf{h}, \mathbf{g}]}k_{[\mathbf{i}, \mathbf{j}, \mathbf{i}]}k_{\mathfrak{k}}k_{\mathfrak{l}}$, then $D_{\mathbf{g}, \mathbf{h}, \mathfrak{k}}=D_{\mathbf{i}, \mathbf{j}, \mathfrak{l}}$ if and only if $\mathbf{g}=\mathbf{i}$, $\mathbf{h}=\mathbf{j}$, and $\mathfrak{k}=\mathfrak{l}$.
\end{lem}
\begin{proof}
One direction is obvious. Assume that $D_{\mathbf{g}, \mathbf{h}, \mathfrak{k}}\!=\!D_{\mathbf{i}, \mathbf{j}, \mathfrak{l}}$. Then $O\notin\{D_{\mathbf{g}, \mathbf{h}, \mathfrak{k}}, D_{\mathbf{i}, \mathbf{j}, \mathfrak{l}}\}$ as $p\nmid k_{\mathfrak{k}}k_{\mathfrak{l}}$. Notice that $\mathbf{g}=\mathbf{i}$ and $\mathbf{h}=\mathbf{j}$ by Equation \eqref{Eq;3}. Notice that $\mathbb{U}_{\mathbf{g}, \mathbf{h}, \mathfrak{k}, 0}=\{\mathfrak{k}\}$, $\mathbb{U}_{\mathbf{g}, \mathbf{h}, \mathfrak{l}, 0}=\{\mathfrak{l}\}$, $B_{\mathbf{g}, \mathbf{h}, \mathfrak{k}}\in\mathrm{Supp}_{\mathbb{B}_2}(D_{\mathbf{g}, \mathbf{h}, \mathfrak{l}})$, $B_{\mathbf{g}, \mathbf{h}, \mathfrak{l}}\in\mathrm{Supp}_{\mathbb{B}_2}(D_{\mathbf{g}, \mathbf{h}, \mathfrak{k}})$, and $\mathfrak{k}\preceq\mathfrak{l}\preceq\mathfrak{k}$ by Equation \eqref{Eq;12} and Theorem \ref{T;Theorem3.13}. Hence $\mathfrak{k}=\mathfrak{l}$. The desired lemma thus follows.
\end{proof}
\begin{lem}\label{L;Lemma7.3}
$\mathbb{T}$ has an $\F$-linearly independent subset $$\{D_{\mathbf{a}, \mathbf{b}, \mathfrak{c}}: \mathbf{a}, \mathbf{b}\in \mathbb{E}, \mathfrak{c}\in\mathbb{U}_{\mathbf{a}, \mathbf{b}}, p\nmid k_{[\mathbf{a}, \mathbf{b}, \mathbf{a}]}k_{[\mathbf{b}, \mathbf{a}, \mathbf{b}]}k_\mathfrak{c}\}.$$
\end{lem}
\begin{proof}
Let $\mathbb{U}$ be the above displayed set. Then $D_{0_\mathfrak{S}, 0_\mathfrak{S}, \mathfrak{o}}\in \mathbb{U}$. Let $L$ be a nonzero $\F$-linear combination of the matrices in $\mathbb{U}$. Assume that $L=O$. If $M\in\mathbb{U}$, let $c_M$ be the coefficient of $M$ in $L$. So there is $N\!\in\!\mathbb{U}$ such that $c_N\in\F^\times$. By Equation \eqref{Eq;3}, $N=E_\mathbf{g}^*NE_\mathbf{h}^*$ for some $\mathbf{g}, \mathbf{h}\in \mathbb{E}$. So $\mathbb{V}=\{A: A\in\mathbb{U}, c_A\in\F^\times, A=E_\mathbf{g}^*AE_\mathbf{h}^*\}\neq \varnothing$. According to Lemma \ref{L;Lemma7.2}, there exist $i\in\mathbb{N}$ and pairwise distinct $\mathfrak{j}_1, \mathfrak{j}_2, \ldots, \mathfrak{j}_i\in\mathbb{U}_{\mathbf{g}, \mathbf{h}}$ such that $\mathbb{V}\!=\!\{D_{\mathbf{g}, \mathbf{h}, \mathfrak{j}_1}, D_{\mathbf{g}, \mathbf{h}, \mathfrak{j}_2}, \ldots, D_{\mathbf{g}, \mathbf{h}, \mathfrak{j}_i}\}$. We distinguish the cases $i=1$ and $i>1$.

If $i=1$, there is $c\in\F^\times$ such that $E_\mathbf{g}^*LE_\mathbf{h}^*=cD_{\mathbf{g}, \mathbf{h}, \mathfrak{j}_1}=O$. This is a contradiction as $p\nmid k_{\mathfrak{j}_1}$ and $D_{\mathbf{g}, \mathbf{h}, \mathfrak{j}_1}\neq O$. Assume further that $i>1$. There is no loss to assume that $\mathfrak{j}_1$ is minimal in $\{\mathfrak{j}_1, \mathfrak{j}_2, \ldots, \mathfrak{j}_i\}$ with respect to the partial order $\preceq$.
As $E_\mathbf{g}^*LE_\mathbf{h}^*=O$, $D_{\mathbf{g}, \mathbf{h}, \mathfrak{j}_1}$ is also an $\F$-linear combination of the matrices in $\{D_{\mathbf{g}, \mathbf{h}, \mathfrak{j}_2}, D_{\mathbf{g}, \mathbf{h}, \mathfrak{j}_3}, \ldots, D_{\mathbf{g}, \mathbf{h}, \mathfrak{j}_i}\}$. As $\mathbb{U}_{\mathbf{g}, \mathbf{h}, \mathfrak{j}_1, 0}=\{\mathfrak{j}_1\}$ and Theorem \ref{T;Theorem3.13} holds, this $\F$-linear combination of $D_{\mathbf{g}, \mathbf{h}, \mathfrak{j}_1}$ shows that $\mathfrak{j}\preceq\mathfrak{j}_1$ for some $\mathfrak{j}\in\{\mathfrak{j}_2, \mathfrak{j}_3, \ldots,\mathfrak{j}_i\}$. Hence $\mathfrak{j}_1\in\{\mathfrak{j}_2, \mathfrak{j}_3, \ldots,\mathfrak{j}_i\}$ by the choice of $\mathfrak{j}_1$. This is also a contradiction. Therefore $L\neq O$. The desired lemma thus follows.
\end{proof}
We are now ready to introduce the first main result of this section.
\begin{thm}\label{T;Theorem7.4}
$\mathbb{T}/\mathrm{Rad}(\mathbb{T})$ has an $\F$-basis $$\{D_{\mathbf{a}, \mathbf{b}, \mathfrak{c}}+\mathrm{Rad}(\mathbb{T}): \mathbf{a}, \mathbf{b}\in\mathbb{E}, \mathfrak{c}\in\mathbb{U}_{\mathbf{a}, \mathbf{b}}, p\nmid k_{[\mathbf{a}, \mathbf{b}, \mathbf{a}]}k_{[\mathbf{b}, \mathbf{a}, \mathbf{b}]}k_\mathfrak{c}\}.$$
\end{thm}
\begin{proof}
Let $\mathbb{U}$ be the union of $\{D_{\mathbf{a}, \mathbf{b}, \mathfrak{c}}: \mathbf{a}, \mathbf{b}\in \mathbb{E}, \mathfrak{c}\in\mathbb{U}_{\mathbf{a}, \mathbf{b}}, p\nmid k_{[\mathbf{a}, \mathbf{b}, \mathbf{a}]}k_{[\mathbf{b}, \mathbf{a}, \mathbf{b}]}k_\mathfrak{c}\}$ and $\{B_{\mathbf{a}, \mathbf{b}, \mathfrak{c}}: \mathbf{a}, \mathbf{b}\in \mathbb{E}, \mathfrak{c}\in\mathbb{U}_{\mathbf{a}, \mathbf{b}},  p\mid k_{[\mathbf{a}, \mathbf{b}, \mathbf{a}]}k_{[\mathbf{b}, \mathbf{a}, \mathbf{b}]}k_\mathfrak{c}\}$. According to Theorem \ref{T;Theorem3.13} and Lemma \ref{L;Lemma7.2}, notice that $|\mathbb{U}|\!=\!|\{(\mathbf{a}, \mathbf{b}, \mathfrak{c}): \mathbf{a}, \mathbf{b}\in \mathbb{E}, \mathfrak{c}\in\mathbb{U}_{\mathbf{a}, \mathbf{b}}\}|$. Therefore $\mathbb{T}$ has an $\F$-basis $\mathbb{U}$ by Lemma \ref{L;Lemma7.3} and Theorem \ref{T;Theorem3.13}. The desired theorem thus follows.
\end{proof}
For the remaining main results of this section, the following lemmas are necessary.
\begin{lem}\label{L;Lemma7.5}
Assume that $\mathbf{g}, \mathbf{h}, \mathbf{i}\in \mathbb{E}$. Assume that $\mathfrak{j}=(\mathbbm{j}_{(1)}, \mathbbm{j}_{(2)}, \mathbbm{j}_{(3)})\in\mathbb{U}_{\mathbf{g}, \mathbf{h}}$ and $\mathfrak{k}=(\mathbbm{k}_{(1)}, \mathbbm{k}_{(2)}, \mathbbm{k}_{(3)})\in\mathbb{U}_{\mathbf{h}, \mathbf{i}}$. If the containments $((\mathbbm{h}_1\cap\mathbbm{i}_1)^\circ\setminus\mathbbm{g}_1)\cup(\mathbbm{i}_1\cap\mathbbm{j}_{(1)})\subseteq\mathbbm{k}_{(1)}$,
$({\mathbbm{k}_{(3)}}^\bullet\setminus \mathbbm{g}_2)\!\cup\!(\mathbbm{j}_{(2)}\cap\mathbbm{k}_{(3)})\!\!\subseteq\!\!\mathbbm{k}_{(2)}$, $((\mathbbm{h}_2\!\cap\!\mathbbm{i}_2)\!\setminus\!\mathbbm{g}_2)\!\cup\!(\mathbbm{i}_2\!\cap\!\mathbbm{j}_{(3)})\!\subseteq\!\mathbbm{k}_{(3)}$ hold, then $n_{\mathbf{g}, \mathbf{i}, (\mathbf{g}, \mathbf{h}, \mathbf{i}, \mathfrak{j}, \mathfrak{k})}\!=\!n_{\mathbf{h}, \mathbf{i}, \mathfrak{k}}$.
\end{lem}
\begin{proof}
As $((\mathbbm{h}_1\cap\mathbbm{i}_1)^\circ\setminus\mathbbm{g}_1)\cup(\mathbbm{i}_1\cap\mathbbm{j}_{(1)})\subseteq\mathbbm{k}_{(1)}$, $(\mathbbm{h}_1\cap\mathbbm{i}_1)^\circ\setminus \mathbbm{k}_{(1)}=((\mathbbm{g}_1\cap\mathbbm{i}_1)^\circ\cap \mathbbm{h}_1)\setminus \mathbbm{k}_{(1)}$ and $(\mathbbm{g}_1\cap\mathbbm{i}_1)^\circ\setminus(((\mathbbm{g}_1\cap\mathbbm{i}_1)^\circ\setminus\mathbbm{h}_1)\cup(\mathbbm{g}_1\cap\mathbbm{i}_1\cap(\mathbbm{j}_{(1)}
\cup\mathbbm{k}_{(1)})))=((\mathbbm{g}_1\cap\mathbbm{i}_1)^\circ\cap \mathbbm{h}_1)\setminus \mathbbm{k}_{(1)}$.

As $({\mathbbm{k}_{(3)}}^\bullet\setminus \mathbbm{g}_2)\cup(\mathbbm{j}_{(2)}\cap\mathbbm{k}_{(3)})\subseteq\mathbbm{k}_{(2)}$ and $\mathbbm{i}_2\cap\mathbbm{j}_{(3)}\subseteq\mathbbm{k}_{(3)}$, ${\mathbbm{k}_{(3)}}^\bullet\setminus\mathbbm{k}_{(2)}=(\mathbbm{g}_2\cap{\mathbbm{k}_{(3)}}^\bullet)\setminus\mathbbm{k}_{(2)}$, $((\mathbbm{g}_2\cap\mathbbm{i}_2)^\bullet\setminus\mathbbm{h}_2)\cup(\mathbbm{g}_2\cap\mathbbm{i}_2
\cap({\mathbbm{j}_{(3)}}^\bullet\cup{\mathbbm{k}_{(3)}}^\bullet))=((\mathbbm{g}_2\cap\mathbbm{i}_2)^\bullet\setminus\mathbbm{h}_2)\cup(\mathbbm{g}_2
\cap{\mathbbm{k}_{(3)}}^\bullet)$, and $((\mathbbm{g}_2\cap\mathbbm{i}_2)^\bullet\setminus\mathbbm{h}_2)\cup(\mathbbm{g}_2\cap\mathbbm{i}_2
\cap({\mathbbm{j}_{(2)}}\cup{\mathbbm{k}_{(2)}}))=((\mathbbm{g}_2\cap\mathbbm{i}_2)^\bullet\setminus\mathbbm{h}_2)\cup(\mathbbm{g}_2
\cap{\mathbbm{k}_{(2)}})$.

As $((\mathbbm{h}_2\cap\mathbbm{i}_2)\setminus\mathbbm{g}_2)\cup(\mathbbm{i}_2\cap\mathbbm{j}_{(3)})\subseteq\mathbbm{k}_{(3)}$, $(\mathbbm{h}_2\cap\mathbbm{i}_2)\setminus\mathbbm{k}_{(3)}=(\mathbbm{g}_2\cap\mathbbm{h}_2\cap\mathbbm{i}_2)\setminus \mathbbm{k}_{(3)}$ and $(\mathbbm{g}_2\cap\mathbbm{i}_2)\setminus(((\mathbbm{g}_2\cap\mathbbm{i}_2)\setminus\mathbbm{h}_2)\cup(\mathbbm{g}_2\cap\mathbbm{i}_2\cap(\mathbbm{j}_{(3)}
\cup\mathbbm{k}_{(3)})))=(\mathbbm{g}_2\cap\mathbbm{h}_2\cap\mathbbm{i}_2)\setminus \mathbbm{k}_{(3)}$. The desired lemma follows since $|(\mathbf{g}, \mathbf{i}; (\mathbf{g}, \mathbf{h}, \mathbf{i}, \mathfrak{j}, \mathfrak{k}))|-|(\mathbf{g}, \mathbf{h}, \mathbf{i}, \mathfrak{j}, \mathfrak{k})|=|(\mathbf{h}, \mathbf{i}; \mathfrak{k})|-|\mathfrak{k}|$.
\end{proof}
\begin{lem}\label{L;Lemma7.6}
Assume that $\mathbf{g}, \mathbf{h}, \mathbf{i}\in\mathbb{E}$, $\mathfrak{j}\in\mathbb{U}_{\mathbf{g}, \mathbf{h}}$,
$\mathbf{k}\in\mathbb{U}_{\mathbf{h}, \mathbf{i}}$, and $p\nmid k_{[\mathbf{h}, \mathbf{i}, \mathbf{h}]}k_\mathfrak{j}k_\mathfrak{k}$.
Then $p\nmid k_{(\mathbf{g}, \mathbf{h}, \mathbf{i}, \mathfrak{j}, \mathfrak{k})}$.
\end{lem}
\begin{proof}
Assume that $\mathfrak{j}=(\mathbbm{j}_{(1)}, \mathbbm{j}_{(2)}, \mathbbm{j}_{(3)})$ and $\mathfrak{k}=(\mathbbm{k}_{(1)}, \mathbbm{k}_{(2)}, \mathbbm{k}_{(3)})$.
Set $\mathfrak{l}=(\mathbbm{l}_{(1)}, \mathbbm{l}_{(2)}, \mathbbm{l}_{(3)})$, $\mathbbm{l}_{(1)}=\mathbbm{g}_1\cap\mathbbm{i}_1\cap(\mathbbm{j}_{(1)}\cup\mathbbm{k}_{(1)})$,
$\mathbbm{l}_{(2)}=\mathbbm{g}_2\cap\mathbbm{i}_2\cap(\mathbbm{j}_{(2)}\cup\mathbbm{k}_{(2)})$, and $\mathbbm{l}_{(3)}=\mathbbm{g}_2\cap\mathbbm{i}_2\cap(\mathbbm{j}_{(3)}\cup\mathbbm{k}_{(3)})$. As $p\nmid k_\mathfrak{j}k_\mathfrak{k}$, notice that $p\nmid (m_q-1)(\ell_r-1)m_rm_s$ for any $q\in\mathbbm{l}_{(1)}$, $r\in\mathbbm{l}_{(2)}$, and $s\in\mathbbm{l}_{(3)}\setminus\mathbbm{l}_{(2)}$. Therefore $p\nmid k_\mathfrak{l}$. As $p\nmid k_{[\mathbf{h}, \mathbf{i}, \mathbf{h}]}$, notice that $p\nmid (m_q-1)(\ell_r-1)m_r$ for any $q\in (\mathbbm{g}_1\cap\mathbbm{i}_1)\setminus \mathbbm{h}_1$ and $r\in(\mathbbm{g}_2\cap\mathbbm{i}_2)\setminus \mathbbm{h}_2$. So $p\nmid k_{(\mathbf{g}, \mathbf{h}, \mathbf{i})}$. Since $k_{(\mathbf{g}, \mathbf{h}, \mathbf{i}, \mathfrak{j}, \mathfrak{k})}=k_{(\mathbf{g}, \mathbf{h}, \mathbf{i})}k_\mathfrak{l}$ by Lemmas \ref{L;Lemma3.20} and \ref{L;Lemma3.19}, the desired lemma follows from the above discussion.
\end{proof}
\begin{lem}\label{L;Lemma7.7}
Assume that $\mathbf{g}, \mathbf{h}, \mathbf{i}\in \mathbb{E}$. Assume that $\mathfrak{j}=(\mathbbm{j}_{(1)}, \mathbbm{j}_{(2)}, \mathbbm{j}_{(3)})\in\mathbb{U}_{\mathbf{g}, \mathbf{h}}$ and $\mathfrak{k}=(\mathbbm{k}_{(1)}, \mathbbm{k}_{(2)}, \mathbbm{k}_{(3)})\in\mathbb{U}_{\mathbf{h}, \mathbf{i}}$. Assume that $\ell\in[0, n_{\mathbf{h}, \mathbf{i}, \mathfrak{k}}]$, $\mathfrak{m}\in\mathbb{U}_{\mathbf{h}, \mathbf{i}, \mathfrak{k}, \ell}$, and $p\nmid k_{[\mathbf{h}, \mathbf{i}, \mathbf{h}]}k_\mathfrak{j}$. Assume that $((\mathbbm{h}_1\cap\mathbbm{i}_1)^\circ\setminus\mathbbm{g}_1)\cup(\mathbbm{i}_1\cap\mathbbm{j}_{(1)})\subseteq\mathbbm{k}_{(1)}$, $({\mathbbm{k}_{(3)}}^\bullet\setminus \mathbbm{g}_2)\cup(\mathbbm{j}_{(2)}\cap\mathbbm{k}_{(3)})\subseteq\mathbbm{k}_{(2)}$, and $((\mathbbm{h}_2\cap\mathbbm{i}_2)\setminus\mathbbm{g}_2)\cup(\mathbbm{i}_2\cap\mathbbm{j}_{(3)})\subseteq\mathbbm{k}_{(3)}$. Then $(\mathbf{g}, \mathbf{h}, \mathbf{i}, \mathfrak{j}, \mathfrak{m})\in\mathbb{U}_{\mathbf{g}, \mathbf{i}, (\mathbf{g}, \mathbf{h}, \mathbf{i}, \mathfrak{j}, \mathfrak{k}), \ell}$.
\end{lem}
\begin{proof}
As $\mathfrak{m}\!\in\!\mathbb{U}_{\mathbf{h}, \mathbf{i}, \mathfrak{k}, \ell}$, $\mathfrak{k}\preceq \mathfrak{m}\!\preceq\!(\mathbf{h}, \mathbf{i}; \mathfrak{k})$, $|\mathfrak{m}|-|\mathfrak{k}|=\ell$, $p\nmid k_\mathfrak{m}$. By a direct computation, $(\mathbf{g}, \mathbf{h}, \mathbf{i}, \mathfrak{j}, \mathfrak{k})\!\!\preceq\!\!(\mathbf{g}, \mathbf{h}, \mathbf{i}, \mathfrak{j}, \mathfrak{m})\!\preceq\!(\mathbf{g}, \mathbf{i}; (\mathbf{g}, \mathbf{h}, \mathbf{i}, \mathfrak{j}, \mathfrak{k}))$. By Lemma \ref{L;Lemma7.5}, $\mathbb{U}_{\mathbf{g}, \mathbf{i}, (\mathbf{g}, \mathbf{h}, \mathbf{i}, \mathfrak{j}, \mathfrak{k}), \ell}$ is defined. Then $|(\mathbf{g}, \mathbf{h}, \mathbf{i}, \mathfrak{j}, \mathfrak{m})|-|(\mathbf{g}, \mathbf{h}, \mathbf{i}, \mathfrak{j}, \mathfrak{k})|=|\mathfrak{m}|-|\mathfrak{k}|=\ell$ by a direct computation and the containments $((\mathbbm{h}_1\cap\mathbbm{i}_1)^\circ\setminus\mathbbm{g}_1)\cup(\mathbbm{i}_1\cap\mathbbm{j}_{(1)})\subseteq\mathbbm{k}_{(1)}$,
$({\mathbbm{k}_{(3)}}^\bullet\setminus \mathbbm{g}_2)\cup(\mathbbm{j}_{(2)}\cap\mathbbm{k}_{(3)})\subseteq\mathbbm{k}_{(2)}$, and
$((\mathbbm{h}_2\cap\mathbbm{i}_2)\setminus\mathbbm{g}_2)\cup(\mathbbm{i}_2\cap\mathbbm{j}_{(3)})\subseteq\mathbbm{k}_{(3)}$. As $p\nmid k_{[\mathbf{h}, \mathbf{i}, \mathbf{h}]}k_\mathfrak{j}$, $p\nmid k_{[\mathbf{h}, \mathbf{i}, \mathbf{h}]}k_\mathfrak{j}k_\mathfrak{m}$ and $p\nmid k_{(\mathbf{g}, \mathbf{h}, \mathbf{i}, \mathfrak{j}, \mathfrak{m})}$ by Lemma \ref{L;Lemma7.6}. The desired lemma thus follows from the above discussion.
\end{proof}
\begin{lem}\label{L;Lemma7.8}
Assume that $\mathbf{g}, \mathbf{h}, \mathbf{i}\!\in\! \mathbb{E}$. Assume that $\mathfrak{j}=(\mathbbm{j}_{(1)}, \mathbbm{j}_{(2)}, \mathbbm{j}_{(3)})\in\mathbb{U}_{\mathbf{g}, \mathbf{h}}$ and $\mathfrak{k}=(\mathbbm{k}_{(1)}, \mathbbm{k}_{(2)}, \mathbbm{k}_{(3)})\in\mathbb{U}_{\mathbf{h}, \mathbf{i}}$. Assume that $\ell\in[0, n_{\mathbf{h}, \mathbf{i}, \mathfrak{k}}]$, $\mathfrak{m}, \mathfrak{q}\in\mathbb{U}_{\mathbf{h}, \mathbf{i}, \mathfrak{k}, \ell}$, and $p\nmid k_{[\mathbf{h}, \mathbf{i}, \mathbf{h}]}k_\mathfrak{j}$. Assume that $((\mathbbm{h}_1\cap\mathbbm{i}_1)^\circ\setminus\mathbbm{g}_1)\cup(\mathbbm{i}_1\cap\mathbbm{j}_{(1)})\subseteq\mathbbm{k}_{(1)}$,
$({\mathbbm{k}_{(3)}}^\bullet\setminus \mathbbm{g}_2)\cup(\mathbbm{j}_{(2)}\cap\mathbbm{k}_{(3)})\subseteq\mathbbm{k}_{(2)}$, and
$((\mathbbm{h}_2\cap\mathbbm{i}_2)\setminus\mathbbm{g}_2)\cup(\mathbbm{i}_2\cap\mathbbm{j}_{(3)})\!\subseteq\!\mathbbm{k}_{(3)}$. Then $\mathfrak{m}=\mathfrak{q}$ if and only if $(\mathbf{g}, \mathbf{h}, \mathbf{i}, \mathfrak{j}, \mathfrak{m})=(\mathbf{g}, \mathbf{h}, \mathbf{i}, \mathfrak{j}, \mathfrak{q})$. Moreover, the map that sends $\mathfrak{r}$ to $(\mathbf{g}, \mathbf{h}, \mathbf{i}, \mathfrak{j}, \mathfrak{r})$ is injective from $\mathbb{U}_{\mathbf{h}, \mathbf{i}, \mathfrak{k}, \ell}$ to $\mathbb{U}_{\mathbf{g}, \mathbf{i}, (\mathbf{g}, \mathbf{h}, \mathbf{i}, \mathfrak{j}, \mathfrak{k}), \ell}$.
\end{lem}
\begin{proof}
Assume that $\mathfrak{m}=(\mathbbm{m}_{(1)}, \mathbbm{m}_{(2)}, \mathbbm{m}_{(3)})$ and $\mathfrak{q}=(\mathbbm{q}_{(1)}, \mathbbm{q}_{(2)}, \mathbbm{q}_{(3)})$. Assume that $(\mathbf{g}, \mathbf{h}, \mathbf{i}, \mathfrak{j}, \mathfrak{m})=(\mathbf{g}, \mathbf{h}, \mathbf{i}, \mathfrak{j}, \mathfrak{q})$. Notice that $\mathbbm{g}_1\cap\mathbbm{i}_1\cap(\mathbbm{j}_{(1)}\cup\mathbbm{m}_{(1)})=\mathbbm{g}_1\cap\mathbbm{i}_1\cap(\mathbbm{j}_{(1)}\cup\mathbbm{q}_{(1)})$, $\mathbbm{g}_2\cap\mathbbm{i}_2\cap(\mathbbm{j}_{(2)}\cup\mathbbm{m}_{(2)})\!=\!\mathbbm{g}_2\cap\mathbbm{i}_2\cap(\mathbbm{j}_{(2)}\cup\mathbbm{q}_{(2)})$, and
$\mathbbm{g}_2\cap\mathbbm{i}_2\cap(\mathbbm{j}_{(3)}\cup
\mathbbm{m}_{(3)})\!=\!\mathbbm{g}_2\cap\mathbbm{i}_2\cap(\mathbbm{j}_{(3)}\cup\mathbbm{q}_{(3)})$.

As $\mathbbm{i}_1\cap\mathbbm{j}_{(1)}\subseteq\mathbbm{k}_{(1)}$, $\mathbbm{j}_{(2)}\cap\mathbbm{k}_{(3)}\subseteq\mathbbm{k}_{(2)}$, $\mathbbm{i}_2\cap\mathbbm{j}_{(3)}\subseteq\mathbbm{k}_{(3)}$, $\mathfrak{k}\preceq\mathfrak{m}$, and $\mathfrak{k}\preceq\mathfrak{q}$, notice that $\mathbbm{g}_1\cap\mathbbm{m}_{(1)}=\mathbbm{g}_1\cap\mathbbm{q}_{(1)}$, $\mathbbm{g}_2\cap\mathbbm{m}_{(2)}=\mathbbm{g}_2\cap\mathbbm{q}_{(2)}$, and $\mathbbm{g}_2\cap\mathbbm{m}_{(3)}=\mathbbm{g}_2\cap\mathbbm{q}_{(3)}$. Then  $\mathbbm{m}_{(1)}\!\setminus\!\mathbbm{g}_1\!\!=\!\!\mathbbm{k}_{(1)}\!\setminus\!\mathbbm{g}_1\!\!=\!\!\mathbbm{q}_{(1)}\!\setminus\!\mathbbm{g}_1$, $\mathbbm{m}_{(2)}\!\setminus\!\mathbbm{g}_2\!=\!\mathbbm{k}_{(2)}\!\setminus\!\mathbbm{g}_2\!=\!\mathbbm{q}_{(2)}\!\setminus\!\mathbbm{g}_2$,
$\mathbbm{m}_{(3)}\!\setminus\!\mathbbm{g}_2\!=\!\mathbbm{k}_{(3)}\!\setminus\!\mathbbm{g}_2\!=\!\mathbbm{q}_{(3)}\!\setminus\!\mathbbm{g}_2$ as
$(\mathbbm{h}_1\cap\mathbbm{i}_1)^\circ\setminus \mathbbm{g}_1\subseteq\mathbbm{k}_{(1)}$, ${\mathbbm{k}_{(3)}}^\bullet\setminus \mathbbm{g}_2\subseteq\mathbbm{k}_{(2)}$, and $(\mathbbm{h}_2\cap\mathbbm{i}_2)\setminus\mathbbm{g}_2\subseteq\mathbbm{k}_{(3)}$. Hence $\mathfrak{m}=\mathfrak{q}$. The first statement thus follows. The desired lemma thus follows from Lemma \ref{L;Lemma7.7}.
\end{proof}
\begin{lem}\label{L;Lemma7.9}
Assume that $\mathbf{g}, \mathbf{h}, \mathbf{i}\in\mathbb{E}$. Assume that $\mathfrak{j}=(\mathbbm{j}_{(1)}, \mathbbm{j}_{(2)}, \mathbbm{j}_{(3)})\in\mathbb{U}_{\mathbf{g}, \mathbf{h}}$ and $\mathfrak{k}\!=\!(\mathbbm{k}_{(1)}, \mathbbm{k}_{(2)}, \mathbbm{k}_{(3)})\in\mathbb{U}_{\mathbf{h}, \mathbf{i}}$. Assume that $\ell\in[0, n_{\mathbf{h}, \mathbf{i}, \mathfrak{k}}]$,  $p\nmid k_{[\mathbf{h}, \mathbf{i}, \mathbf{h}]}k_\mathfrak{j}k_\mathfrak{k}$, $\mathfrak{m}\!\in\!\mathbb{U}_{\mathbf{g}, \mathbf{i}, (\mathbf{g}, \mathbf{h}, \mathbf{i}, \mathfrak{j}, \mathfrak{k}), \ell}$. Assume that $((\mathbbm{h}_1\cap\mathbbm{i}_1)^\circ\setminus\mathbbm{g}_1)\cup(\mathbbm{i}_1\cap\mathbbm{j}_{(1)})\subseteq\mathbbm{k}_{(1)}$,
$({\mathbbm{k}_{(3)}}^\bullet\setminus \mathbbm{g}_2)\cup(\mathbbm{j}_{(2)}\cap\mathbbm{k}_{(3)})\subseteq\mathbbm{k}_{(2)}$, and $((\mathbbm{h}_2\cap\mathbbm{i}_2)\setminus\mathbbm{g}_2)\cup(\mathbbm{i}_2\cap\mathbbm{j}_{(3)})\subseteq\mathbbm{k}_{(3)}$. Then there is $\mathfrak{q}\in\mathbb{U}_{\mathbf{h}, \mathbf{i}, \mathfrak{k}, \ell}$ such that $\mathfrak{m}=(\mathbf{g}, \mathbf{h}, \mathbf{i}, \mathfrak{j}, \mathfrak{q})$. Moreover, the map that sends $\mathfrak{r}$ to $(\mathbf{g}, \mathbf{h}, \mathbf{i}, \mathfrak{j}, \mathfrak{r})$ is bijective from $\mathbb{U}_{\mathbf{h}, \mathbf{i}, \mathfrak{k}, \ell}$ to $\mathbb{U}_{\mathbf{g}, \mathbf{i}, (\mathbf{g}, \mathbf{h}, \mathbf{i}, \mathfrak{j}, \mathfrak{k}), \ell}$.
\end{lem}
\begin{proof}
Assume that $\mathfrak{m}=(\mathbbm{m}_{(1)}, \mathbbm{m}_{(2)}, \mathbbm{m}_{(3)})$ and $\mathfrak{q}=(\mathbbm{q}_{(1)}, \mathbbm{q}_{(2)}, \mathbbm{q}_{(3)})$. Assume that $\mathbbm{q}_{(1)}\!=\!(\mathbbm{k}_{(1)}\!\setminus\!\mathbbm{g}_1)\cup(\mathbbm{h}_1\cap\mathbbm{m}_{(1)})$, $\mathbbm{q}_{(2)}\!=\!(\mathbbm{k}_{(2)}\!\setminus\!\mathbbm{g}_2)\cup(\mathbbm{h}_2\cap\mathbbm{m}_{(2)})$,
$\mathbbm{q}_{(3)}\!=\!(\mathbbm{k}_{(3)}\!\setminus\!\mathbbm{g}_2)\cup(\mathbbm{h}_2\cap\mathbbm{m}_{(3)})$.

As $\mathfrak{m}\in\mathbb{U}_{\mathbf{g}, \mathbf{i}, (\mathbf{g}, \mathbf{h}, \mathbf{i}, \mathfrak{j}, \mathfrak{k}), \ell}$ and $\mathfrak{k}\in\mathbb{U}_{\mathbf{h}, \mathbf{i}}$, notice that $(\mathbf{g}, \mathbf{h}, \mathbf{i}, \mathfrak{j}, \mathfrak{k})\preceq\mathfrak{m}\!\preceq\!(\mathbf{g}, \mathbf{i}; (\mathbf{g}, \mathbf{h}, \mathbf{i}, \mathfrak{j}, \mathfrak{k}))$, $|\mathfrak{m}|-|(\mathbf{g}, \mathbf{h}, \mathbf{i}, \mathfrak{j}, \mathfrak{k})|=\ell$, $p\nmid k_\mathfrak{m}$,
$\mathfrak{q}\in\mathbb{U}_{\mathbf{h}, \mathbf{i}}$, and $\mathfrak{m}=(\mathbf{g}, \mathbf{h}, \mathbf{i}, \mathfrak{j}, \mathfrak{q})$ by a direct computation. As $\mathbbm{i}_1\cap\mathbbm{j}_{(1)}\subseteq\mathbbm{k}_{(1)}$, $\mathbbm{j}_{(2)}\cap\mathbbm{k}_{(3)}\subseteq\mathbbm{k}_{(2)}$, and $\mathbbm{i}_2\cap\mathbbm{j}_{(3)}\subseteq\mathbbm{k}_{(3)}$,
notice that $\mathfrak{k}\preceq\mathfrak{q}\preceq(\mathbf{h}, \mathbf{i}; \mathfrak{k})$, $\mathbbm{q}_{(1)}\setminus\mathbbm{k}_{(1)}=(\mathbbm{h}_1\cap\mathbbm{m}_{(1)})\setminus\mathbbm{k}_{(1)}=\mathbbm{m}_{(1)}\setminus(((\mathbbm{g}_1
\cap\mathbbm{i}_1)^\circ\setminus\mathbbm{h}_1)\cup(\mathbbm{g}_1\cap\mathbbm{i}_1\cap(\mathbbm{j}_{(1)}\cup\mathbbm{k}_{(1)})))$,
$\mathbbm{q}_{(2)}\setminus\mathbbm{k}_{(2)}=(\mathbbm{h}_2\cap\mathbbm{m}_{(2)})\setminus\mathbbm{k}_{(2)}=\mathbbm{m}_{(2)}\setminus(((\mathbbm{g}_2
\cap\mathbbm{i}_2)^\bullet\setminus\mathbbm{h}_2)\cup(\mathbbm{g}_2\cap\mathbbm{i}_2\cap(\mathbbm{j}_{(2)}\cup\mathbbm{k}_{(2)})))$,
and $\mathbbm{q}_{(3)}\setminus\mathbbm{k}_{(3)}=(\mathbbm{h}_2\cap\mathbbm{m}_{(3)})\setminus\mathbbm{k}_{(3)}=\mathbbm{m}_{(3)}\setminus(((\mathbbm{g}_2
\cap\mathbbm{i}_2)\setminus\mathbbm{h}_2)\cup(\mathbbm{g}_2\cap\mathbbm{i}_2\cap(\mathbbm{j}_{(3)}\cup\mathbbm{k}_{(3)})))$. Hence
$|\mathfrak{q}|-|\mathfrak{k}|=|\mathfrak{m}|-|(\mathbf{g}, \mathbf{h}, \mathbf{i}, \mathfrak{j}, \mathfrak{k})|=\ell$. As $p\nmid k_\mathfrak{k}k_\mathfrak{m}$, notice that
$p\nmid (m_s-1)(\ell_t-1)m_tm_u$ for any $s\in\mathbbm{q}_1$, $t\in\mathbbm{q}_2$, and $u\in \mathbbm{q}_3\setminus\mathbbm{q}_2$. Therefore $p\nmid k_\mathfrak{q}$. The first statement thus follows.
The desired lemma follows from the first statement and Lemma \ref{L;Lemma7.8}.
\end{proof}
\begin{lem}\label{L;Lemma7.10}
Assume that $\mathbf{g}, \mathbf{h}, \mathbf{i}\!\in\! \mathbb{E}$. Assume that $\mathfrak{j}=(\mathbbm{j}_{(1)}, \mathbbm{j}_{(2)}, \mathbbm{j}_{(3)})\!\in\!\mathbb{U}_{\mathbf{g}, \mathbf{h}}$ and $\mathfrak{k}\!=\!(\mathbbm{k}_{(1)}, \mathbbm{k}_{(2)}, \mathbbm{k}_{(3)}), \mathfrak{l}\in\mathbb{U}_{\mathbf{h}, \mathbf{i}}$. Assume that $\mathfrak{k}\preceq\mathfrak{l}$,
$\mathbbm{i}_1\cap\mathbbm{j}_{(1)}\subseteq\mathbbm{k}_{(1)}$, $\mathbbm{j}_{(2)}\cap\mathbbm{k}_{(3)}\!\subseteq\!\mathbbm{k}_{(2)}$, and
$\mathbbm{i}_2\cap\mathbbm{j}_{(3)}\subseteq\mathbbm{k}_{(3)}$. Then $k_\mathfrak{j}=k_{\mathfrak{j}\setminus \mathbf{i}}k_{\mathfrak{j}\cap\mathfrak{l}}$ and
$k_{(\mathbf{g}, \mathbf{h}, \mathbf{i})}k_\mathfrak{l}=k_{\mathfrak{l}\setminus \mathbf{g}}k_{(\mathbf{g}, \mathbf{h}, \mathbf{i}, \mathfrak{j}, \mathfrak{l})}$.
\end{lem}
\begin{proof}
Assume that $\mathfrak{l}\!=\!(\mathbbm{l}_{(1)}, \mathbbm{l}_{(2)}, \mathbbm{l}_{(3)})$, $\mathfrak{m}\!=\!(\mathbbm{m}_{(1)}, \mathbbm{m}_{(2)}, \mathbbm{m}_{(3)})$, $\mathbbm{m}_{(1)}\!\!=\!\!\mathbbm{g}_1\cap\mathbbm{i}_1\cap(\mathbbm{j}_{(1)}\cup\mathbbm{l}_{(1)})$,  $\mathbbm{m}_{(2)}\!\!=\!\!\mathbbm{g}_2\cap\mathbbm{i}_2\cap(\mathbbm{j}_{(2)}\cup\mathbbm{l}_{(2)})$, and
$\mathbbm{m}_{(3)}\!\!=\!\!\mathbbm{g}_2\cap\mathbbm{i}_2\cap(\mathbbm{j}_{(3)}\cup\mathbbm{l}_{(3)})$.
As $\mathfrak{k}\preceq\mathfrak{l}$, $\mathbbm{i}_1\cap\mathbbm{j}_{(1)}\subseteq\mathbbm{k}_{(1)}$, $\mathbbm{j}_{(2)}\cap\mathbbm{k}_{(3)}\subseteq\mathbbm{k}_{(2)}$, $\mathbbm{i}_2\cap\mathbbm{j}_{(3)}\subseteq\mathbbm{k}_{(3)}$, notice that $\mathbbm{j}_{(1)}\cap\mathbbm{l}_{(1)}=\mathbbm{i}_1\cap\mathbbm{j}_{(1)}$, $\mathbbm{j}_{(2)}\cap\mathbbm{l}_{(2)}=\mathbbm{i}_2\cap\mathbbm{j}_{(2)}$,
$\mathbbm{j}_{(3)}\cap\mathbbm{l}_{(3)}=\mathbbm{i}_2\cap\mathbbm{j}_{(3)}$, $\mathbbm{g}_1\cap\mathbbm{i}_1\cap(\mathbbm{j}_{(1)}\cup \mathbbm{l}_{(1)})=\mathbbm{g}_1\cap\mathbbm{l}_{(1)}$, $\mathbbm{g}_2\cap\mathbbm{i}_2\cap(\mathbbm{j}_{(2)}\cup \mathbbm{l}_{(2)})=\mathbbm{g}_2\cap\mathbbm{l}_{(2)}$, and $\mathbbm{g}_2\cap\mathbbm{i}_2\cap(\mathbbm{j}_{(3)}\cup \mathbbm{l}_{(3)})=\mathbbm{g}_2\cap\mathbbm{l}_{(3)}$.
As $k_\mathfrak{j}=k_{\mathfrak{j}\setminus \mathbf{i}}k_{\mathbf{i}\cap \mathfrak{j}}$, $k_\mathfrak{l}=k_{\mathfrak{l}\setminus \mathbf{g}}k_{\mathbf{g}\cap\mathfrak{l}}$,
$k_{(\mathbf{g}, \mathbf{h}, \mathbf{i}, \mathfrak{j}, \mathfrak{l})}=k_{(\mathbf{g}, \mathbf{h}, \mathbf{i})}k_\mathfrak{m}$ by Lemmas \ref{L;Lemma3.20} and \ref{L;Lemma3.19}, the desired lemma follows from the above discussion.
\end{proof}
\begin{lem}\label{L;Lemma7.11}
Assume that $\mathbf{g}, \mathbf{h}, \mathbf{i}\in \mathbb{E}$. Then $k_{[\mathbf{g}, \mathbf{h}, \mathbf{g}]}\!=\!k_{(\mathbf{h}, \mathbf{g}, \mathbf{i})}k_{[\mathbf{g}, \mathbf{h}, \mathbf{i}]}$. In particular, $k_{[\mathbf{g}, \mathbf{h}, \mathbf{g}]}k_{[\mathbf{h}, \mathbf{i}, \mathbf{h}]}=k_{(\mathbf{g}, \mathbf{h}, \mathbf{i})}k_{[\mathbf{g}, \mathbf{h}, \mathbf{i}]}k_{[\mathbf{g},\mathbf{i},\mathbf{g}]}$.
\end{lem}
\begin{proof}
As $\mathbbm{h}_1\setminus \mathbbm{g}_1\!=\!(\mathbbm{h}_1\!\setminus\!(\mathbbm{g}_1\!\cup\!\mathbbm{i}_1))\!\cup\!((\mathbbm{h}_1\cap\mathbbm{i}_1)\setminus\mathbbm{g}_1)$ and $\mathbbm{h}_2\!\setminus\! \mathbbm{g}_2\!=\!(\mathbbm{h}_2\!\setminus\!(\mathbbm{g}_2\!\cup\!\mathbbm{i}_2))\cup((\mathbbm{h}_2\cap\mathbbm{i}_2)\setminus\mathbbm{g}_2)$, the first statement follows from Lemma \ref{L;Lemma3.20}. According to the first statement, notice that
$k_{[\mathbf{h},\mathbf{i},\mathbf{h}]}=k_{(\mathbf{g}, \mathbf{h}, \mathbf{i})}k_{[\mathbf{g}, \mathbf{i}, \mathbf{h}]}$ and $k_{[\mathbf{g},\mathbf{i},\mathbf{g}]}=k_{(\mathbf{h}, \mathbf{g}, \mathbf{i})}k_{[\mathbf{g}, \mathbf{i}, \mathbf{h}]}$.
The desired lemma follows.
\end{proof}
\begin{lem}\label{L;Lemma7.12}
Assume that $\mathbf{g}, \mathbf{h}, \mathbf{i}\in \mathbb{E}$. Assume that $\mathfrak{j}=(\mathbbm{j}_{(1)}, \mathbbm{j}_{(2)}, \mathbbm{j}_{(3)})\in\mathbb{U}_{\mathbf{g}, \mathbf{h}}$, $\mathfrak{k}=(\mathbbm{k}_{(1)}, \mathbbm{k}_{(2)}, \mathbbm{k}_{(3)})\in\mathbb{U}_{\mathbf{h}, \mathbf{i}}$, $p\nmid k_{[\mathbf{g}, \mathbf{i}, \mathbf{g}]}k_{[\mathbf{h}, \mathbf{i}, \mathbf{h}]}k_\mathfrak{j}k_\mathfrak{k}$, $((\mathbbm{h}_1\cap\mathbbm{i}_1)^\circ\setminus\mathbbm{g}_1)
\cup(\mathbbm{i}_1\cap\mathbbm{j}_{(1)})\subseteq\mathbbm{k}_{(1)}$,
$({\mathbbm{k}_{(3)}}^\bullet\setminus \mathbbm{g}_2)\cup(\mathbbm{j}_{(2)}\cap\mathbbm{k}_{(3)})\subseteq\mathbbm{k}_{(2)}$, and
$((\mathbbm{h}_2\cap\mathbbm{i}_2)\setminus\mathbbm{g}_2)
\cup(\mathbbm{i}_2\cap\mathbbm{j}_{(3)})\subseteq\mathbbm{k}_{(3)}$. Then $$p\nmid k_{(\mathbf{g}, \mathbf{h}, \mathbf{i}, \mathfrak{j}, \mathfrak{k})}\ \text{and}\ B_{\mathbf{g}, \mathbf{h}, \mathfrak{j}}D_{\mathbf{h}, \mathbf{i}, \mathfrak{k}}=\overline{k_{[\mathbf{g}, \mathbf{h}, \mathbf{g}]}}\overline{k_\mathfrak{j}}D_{\mathbf{g}, \mathbf{i}, (\mathbf{g}, \mathbf{h}, \mathbf{i}, \mathfrak{j}, \mathfrak{k})}.$$
\end{lem}
\begin{proof}
As $p\nmid k_{[\mathbf{g}, \mathbf{i}, \mathbf{g}]}k_{[\mathbf{h}, \mathbf{i}, \mathbf{h}]}k_\mathfrak{j}k_\mathfrak{k}$, notice that $D_{\mathbf{h}, \mathbf{i}, \mathfrak{k}}\neq O$ and $p\nmid k_{(\mathbf{g}, \mathbf{h}, \mathbf{i}, \mathfrak{j}, \mathfrak{k})}$ by Lemma \ref{L;Lemma7.6}. The combination of Equation \eqref{Eq;12}, Theorem \ref{T;Theorem3.28}, Lemmas \ref{L;Lemma7.5}, \ref{L;Lemma7.9}, \ref{L;Lemma7.10}, \ref{L;Lemma7.11} gives
\begin{align*}
B_{\mathbf{g}, \mathbf{h}, \mathfrak{j}}D_{\mathbf{h}, \mathbf{i}, \mathfrak{k}}=\overline{k_{[\mathbf{g}, \mathbf{h}, \mathbf{g}]}}\overline{k_\mathfrak{j}}\!\!\sum_{\ell=0}^{n_{\mathbf{g}, \mathbf{i}, (\mathbf{g}, \mathbf{h}, \mathbf{i}, \mathfrak{j}, \mathfrak{k})}}\!\!\!\!\sum_{(\mathbf{g}, \mathbf{h}, \mathbf{i}, \mathfrak{j}, \mathfrak{m})\in\mathbb{U}_{\mathbf{g}, \mathbf{i}, (\mathbf{g}, \mathbf{h}, \mathbf{i}, \mathfrak{j}, \mathfrak{k}), \ell}}\!\!\!\!(\overline{-1})^\ell\overline{k_{[\mathbf{g}, \mathbf{i}, \mathbf{g}]}}^{-1}\overline{k_{(\mathbf{g}, \mathbf{h}, \mathbf{i}, \mathfrak{j}, \mathfrak{m})}}^{-1}B_{\mathbf{g}, \mathbf{i}, (\mathbf{g}, \mathbf{h}, \mathbf{i}, \mathfrak{j}, \mathfrak{m})},
\end{align*}
which implies that $B_{\mathbf{g}, \mathbf{h}, \mathfrak{j}}D_{\mathbf{h}, \mathbf{i}, \mathfrak{k}}\!=\!\overline{k_{[\mathbf{g}, \mathbf{h}, \mathbf{g}]}}\overline{k_\mathfrak{j}}D_{\mathbf{g}, \mathbf{i}, (\mathbf{g}, \mathbf{h}, \mathbf{i}, \mathfrak{j}, \mathfrak{k})}$. The desired lemma follows.
\end{proof}
\begin{lem}\label{L;Lemma7.13}
Assume that $\mathbf{g}, \mathbf{h}, \mathbf{i}\in \mathbb{E}$. Assume that $\mathfrak{j}\!=\!(\mathbbm{j}_{(1)}, \mathbbm{j}_{(2)}, \mathbbm{j}_{(3)})\!\in\!\mathbb{U}_{\mathbf{g}, \mathbf{h}}$ and $\mathfrak{k}\!=\!(\mathbbm{k}_{(1)}, \mathbbm{k}_{(2)}, \mathbbm{k}_{(3)})\!\in\!\mathbb{U}_{\mathbf{h}, \mathbf{i}}$. Then the containments $((\mathbbm{g}_1\cap\mathbbm{h}_1)^\circ\!\setminus\!\mathbbm{i}_1)\cup(\mathbbm{g}_1\cap\mathbbm{k}_{(1)})
\!\subseteq\!\mathbbm{j}_{(1)}$, $({\mathbbm{j}_{(3)}}^\bullet\!\setminus\! \mathbbm{i}_2)\cup(\mathbbm{j}_{(3)}\cap\mathbbm{k}_{(2)})\!\subseteq\!\mathbbm{j}_{(2)}$,
$((\mathbbm{g}_2\!\cap\!\mathbbm{h}_2)\!\setminus\!\mathbbm{i}_2)\!\cup\!(\mathbbm{g}_2\!\cap\!\mathbbm{k}_{(3)})
\!\!\subseteq\!\!\mathbbm{j}_{(3)}$, $((\mathbbm{h}_1\!\cap\!\mathbbm{i}_1)^\circ\setminus\mathbbm{g}_1)\cup(\mathbbm{i}_1\cap\mathbbm{j}_{(1)})
\!\!\subseteq\!\!\mathbbm{k}_{(1)}$, $({\mathbbm{k}_{(3)}}^\bullet\!\setminus\! \mathbbm{g}_2)\!\cup\!(\mathbbm{j}_{(2)}\!\cap\!\mathbbm{k}_{(3)})\!\subseteq\!\mathbbm{k}_{(2)}$,
$((\mathbbm{h}_2\cap\mathbbm{i}_2)\setminus\mathbbm{g}_2)\cup(\mathbbm{i}_2\cap\mathbbm{j}_{(3)})
\subseteq\mathbbm{k}_{(3)}$ hold if and only if $(\mathbbm{g}_1\!\cap\!\mathbbm{h}_1)^\circ\!\setminus\!\mathbbm{j}_{(1)}\!=\!
(\mathbbm{h}_1\!\cap\!\mathbbm{i}_1)^\circ\!\setminus\!\mathbbm{k}_{(1)}$,
$(\mathbbm{g}_2\!\cap\!\mathbbm{h}_2)^\bullet\!\setminus\!\mathbbm{j}_{(2)}\!=\!
(\mathbbm{h}_2\!\cap\!\mathbbm{i}_2)^\bullet\!\setminus\!\mathbbm{k}_{(2)}$,
$(\mathbbm{g}_2\!\cap\!\mathbbm{h}_2)\!\setminus\!\mathbbm{j}_{(3)}\!=\!
(\mathbbm{h}_2\!\cap\!\mathbbm{i}_2)\!\setminus\!\mathbbm{k}_{(3)}$.
\end{lem}
\begin{proof}
Assume that the displayed containments hold together. According to a direct computation, notice that $(\mathbbm{g}_1\cap\mathbbm{h}_1)^\circ\setminus\mathbbm{j}_{(1)}=
(\mathbbm{g}_1\cap\mathbbm{h}_1\cap\mathbbm{i}_1)^\circ\!\setminus\!(\mathbbm{j}_{(1)}\!\cup\!\mathbbm{k}_{(1)})=
(\mathbbm{h}_1\!\cap\!\mathbbm{i}_1)^\circ\!\setminus\!\mathbbm{k}_{(1)}$, $(\mathbbm{g}_2\!\cap\!\mathbbm{h}_2)^\bullet\!\setminus\!\mathbbm{j}_{(2)}\!=\!
((\mathbbm{g}_2\cap\mathbbm{h}_2\cap\mathbbm{i}_2)^\bullet\setminus(\mathbbm{j}_{(3)}\cup\mathbbm{k}_{(3)}))\cup(((\mathbbm{g}_2\cap\mathbbm{h}_2\cap\mathbbm{i}_2)^\bullet\cap\mathbbm{j}_{(3)}
\cap\mathbbm{k}_{(3)})\setminus(\mathbbm{j}_{(2)}\cup\mathbbm{k}_{(2)}))$, $(\mathbbm{h}_2\!\cap\!\mathbbm{i}_2)^\bullet\!\setminus\!\mathbbm{k}_{(2)}\!=\!
((\mathbbm{g}_2\cap\mathbbm{h}_2\cap\mathbbm{i}_2)^\bullet\setminus(\mathbbm{j}_{(3)}\cup\mathbbm{k}_{(3)}))\cup(((\mathbbm{g}_2\cap\mathbbm{h}_2\cap\mathbbm{i}_2)^\bullet\cap\mathbbm{j}_{(3)}
\cap\mathbbm{k}_{(3)})\setminus(\mathbbm{j}_{(2)}\cup\mathbbm{k}_{(2)}))$, and $(\mathbbm{g}_2\!\cap\!\mathbbm{h}_2)\!\setminus\!\mathbbm{j}_{(3)}\!=\!(\mathbbm{g}_2\cap\mathbbm{h}_2\cap\mathbbm{i}_2)\!\setminus\!(\mathbbm{j}_{(3)}\!\cup\!\mathbbm{k}_{(3)})=
(\mathbbm{h}_2\!\cap\!\mathbbm{i}_2)\!\setminus\!\mathbbm{k}_{(3)}$. One direction is checked.

For the other direction, the equality $(\mathbbm{g}_1\!\cap\!\mathbbm{h}_1)^\circ\!\setminus\!\mathbbm{j}_{(1)}\!=\!
(\mathbbm{h}_1\!\cap\!\mathbbm{i}_1)^\circ\!\setminus\!\mathbbm{k}_{(1)}$ and a direct computation give $((\mathbbm{g}_1\cap\mathbbm{h}_1)^\circ\!\setminus\!\mathbbm{i}_1)\cup(\mathbbm{g}_1\cap\mathbbm{k}_{(1)})
\!\subseteq\!\mathbbm{j}_{(1)}$ and $((\mathbbm{h}_1\!\cap\!\mathbbm{i}_1)^\circ\setminus\mathbbm{g}_1)\cup(\mathbbm{i}_1\cap\mathbbm{j}_{(1)})
\!\!\subseteq\!\!\mathbbm{k}_{(1)}$. Notice that $((\mathbbm{g}_2\cap\mathbbm{h}_2)^\bullet\setminus \mathbbm{i}_2)\cup(\mathbbm{g}_2\cap\mathbbm{k}_{(2)})\subseteq\mathbbm{j}_{(2)}$ and
$((\mathbbm{h}_2\cap\mathbbm{i}_2)^\bullet\setminus \mathbbm{g}_2)\cup(\mathbbm{i}_2\cap\mathbbm{j}_{(2)})\subseteq\mathbbm{k}_{(2)}$ by the equality
$(\mathbbm{g}_2\!\cap\!\mathbbm{h}_2)^\bullet\!\setminus\!\mathbbm{j}_{(2)}\!=\!
(\mathbbm{h}_2\!\cap\!\mathbbm{i}_2)^\bullet\!\setminus\!\mathbbm{k}_{(2)}$ and a direct computation. Therefore $({\mathbbm{j}_{(3)}}^\bullet\!\setminus\! \mathbbm{i}_2)\cup(\mathbbm{j}_{(3)}\cap\mathbbm{k}_{(2)})\!\subseteq\!\mathbbm{j}_{(2)}$ and
$({\mathbbm{k}_{(3)}}^\bullet\!\setminus\! \mathbbm{g}_2)\!\cup\!(\mathbbm{j}_{(2)}\!\cap\!\mathbbm{k}_{(3)})\!\subseteq\!\mathbbm{k}_{(2)}$.
By a direct computation and the equality $(\mathbbm{g}_2\!\cap\!\mathbbm{h}_2)\!\setminus\!\mathbbm{j}_{(3)}\!=\!
(\mathbbm{h}_2\!\cap\!\mathbbm{i}_2)\!\setminus\!\mathbbm{k}_{(3)}$, notice that $((\mathbbm{g}_2\!\cap\!\mathbbm{h}_2)\!\setminus\!\mathbbm{i}_2)\!\cup\!(\mathbbm{g}_2\!\cap\!\mathbbm{k}_{(3)})
\!\!\subseteq\!\!\mathbbm{j}_{(3)}$ and $((\mathbbm{h}_2\cap\mathbbm{i}_2)\setminus\mathbbm{g}_2)\cup(\mathbbm{i}_2\cap\mathbbm{j}_{(3)})
\subseteq\mathbbm{k}_{(3)}$. The desired lemma thus follows.
\end{proof}
We are now ready to deduce the remaining main results of this section.
\begin{thm}\label{T;Theorem7.14}
Assume that $\mathbf{g}, \mathbf{h}, \mathbf{i}\!\in\! \mathbb{E}$. Assume that $\mathfrak{j}\!\!=\!\!(\mathbbm{j}_{(1)}, \mathbbm{j}_{(2)}, \mathbbm{j}_{(3)})\!\in\!\mathbb{U}_{\mathbf{g}, \mathbf{h}}$ and $\mathfrak{k}\!=\!(\mathbbm{k}_{(1)}, \mathbbm{k}_{(2)}, \mathbbm{k}_{(3)})\!\in\!\mathbb{U}_{\mathbf{l}, \mathbf{i}}$ for some $\mathbf{l}\in\mathbb{E}$. Assume that $p\nmid k_{[\mathbf{g}, \mathbf{h}, \mathbf{g}]}k_{[\mathbf{l}, \mathbf{i}, \mathbf{l}]}k_\mathfrak{j}k_\mathfrak{k}$. Then the inequality $(D_{\mathbf{g}, \mathbf{h}, \mathfrak{j}}+\mathrm{Rad}(\mathbb{T}))(D_{\mathbf{l}, \mathbf{i}, \mathfrak{k}}+\mathrm{Rad}(\mathbb{T}))\neq O+\mathrm{Rad}(\mathbb{T})$ holds only if $\mathbf{h}=\mathbf{l}$,
$(\mathbbm{g}_1\!\cap\!\mathbbm{h}_1)^\circ\!\setminus\!\mathbbm{j}_{(1)}\!=\!
(\mathbbm{h}_1\!\cap\!\mathbbm{i}_1)^\circ\!\setminus\!\mathbbm{k}_{(1)}$,
$(\mathbbm{g}_2\!\cap\!\mathbbm{h}_2)^\bullet\!\setminus\!\mathbbm{j}_{(2)}\!=\!
(\mathbbm{h}_2\!\cap\!\mathbbm{i}_2)^\bullet\!\setminus\!\mathbbm{k}_{(2)}$,
$(\mathbbm{g}_2\!\cap\!\mathbbm{h}_2)\!\setminus\!\mathbbm{j}_{(3)}\!=\!
(\mathbbm{h}_2\!\cap\!\mathbbm{i}_2)\!\setminus\!\mathbbm{k}_{(3)}$.
\end{thm}
\begin{proof}
As $\mathbf{h}=\mathbf{l}$ by Equation \eqref{Eq;3}, the displayed inequality gives $D_{\mathbf{g}, \mathbf{h}, \mathfrak{j}}D_{\mathbf{h}, \mathbf{i}, \mathfrak{k}}\neq O$. The desired theorem thus follows from combining Lemmas \ref{L;Lemma5.14}, \ref{L;Lemma7.1}, and \ref{L;Lemma7.13}.
\end{proof}
\begin{thm}\label{T;Theorem7.15}
Assume that $\mathbf{g}, \mathbf{h}, \mathbf{i}\!\in\! \mathbb{E}$. Assume that $\mathfrak{j}=(\mathbbm{j}_{(1)}, \mathbbm{j}_{(2)}, \mathbbm{j}_{(3)})\!\in\!\mathbb{U}_{\mathbf{g}, \mathbf{h}}$, $\mathfrak{k}\!\!=\!\!(\mathbbm{k}_{(1)}, \mathbbm{k}_{(2)}, \mathbbm{k}_{(3)})\!\in\!\mathbb{U}_{\mathbf{h}, \mathbf{i}}$, $p\nmid k_{[\mathbf{g}, \mathbf{h}, \mathbf{g}]}k_{[\mathbf{h}, \mathbf{g}, \mathbf{h}]}k_{[\mathbf{h}, \mathbf{i}, \mathbf{h}]}k_{[\mathbf{i}, \mathbf{h}, \mathbf{i}]}k_\mathfrak{j}k_\mathfrak{k}$, $(\mathbbm{g}_1\!\cap\!\mathbbm{h}_1)^\circ\!\setminus\!\mathbbm{j}_{(1)}\!=\!
(\mathbbm{h}_1\!\cap\!\mathbbm{i}_1)^\circ\!\setminus\!\mathbbm{k}_{(1)}$,
$(\mathbbm{g}_2\!\cap\!\mathbbm{h}_2)^\bullet\!\setminus\!\mathbbm{j}_{(2)}\!=\!
(\mathbbm{h}_2\!\cap\!\mathbbm{i}_2)^\bullet\!\setminus\!\mathbbm{k}_{(2)}$,
$(\mathbbm{g}_2\!\cap\!\mathbbm{h}_2)\!\setminus\!\mathbbm{j}_{(3)}\!=\!
(\mathbbm{h}_2\!\cap\!\mathbbm{i}_2)\!\setminus\!\mathbbm{k}_{(3)}$. Then $p\nmid k_{[\mathbf{g}, \mathbf{i}, \mathbf{g}]}k_{[\mathbf{i}, \mathbf{g}, \mathbf{i}]}k_{(\mathbf{g}, \mathbf{h}, \mathbf{i}, \mathfrak{j}, \mathfrak{k})}$ and $(D_{\mathbf{g}, \mathbf{h}, \mathfrak{j}}+\mathrm{Rad}(\mathbb{T}))(D_{\mathbf{h}, \mathbf{i}, \mathfrak{k}}+\mathrm{Rad}(\mathbb{T}))=D_{\mathbf{g}, \mathbf{i}, (\mathbf{g}, \mathbf{h}, \mathbf{i}, \mathfrak{j}, \mathfrak{k})}+\mathrm{Rad}(\mathbb{T})$.
\end{thm}
\begin{proof}
Since $p\nmid k_{[\mathbf{g}, \mathbf{h}, \mathbf{g}]}k_{[\mathbf{h}, \mathbf{g}, \mathbf{h}]}k_{[\mathbf{h}, \mathbf{i}, \mathbf{h}]}k_{[\mathbf{i}, \mathbf{h}, \mathbf{i}]}k_\mathfrak{j}k_\mathfrak{k}$, Lemmas
\ref{L;Lemma7.6} and \ref{L;Lemma7.11} thus imply that $p\nmid k_{[\mathbf{g}, \mathbf{i}, \mathbf{g}]}k_{[\mathbf{i}, \mathbf{g}, \mathbf{i}]}k_{(\mathbf{g}, \mathbf{h}, \mathbf{i}, \mathfrak{j}, \mathfrak{k})}$. Assume that $\mathfrak{l}=(\mathbbm{l}_{(1)}, \mathbbm{l}_{(2)}, \mathbbm{l}_{(3)})\in\mathbb{U}_{\mathbf{g}, \mathbf{h}}$, $\mathfrak{j}\preceq\mathfrak{l}\preceq(\mathbf{g}, \mathbf{h}; \mathfrak{j})$, and the containments
$((\mathbbm{h}_1\!\cap\!\mathbbm{i}_1)^\circ\setminus\mathbbm{g}_1)\!\cup\!(\mathbbm{i}_1\cap\mathbbm{l}_{(1)})
\!\!\subseteq\!\!\mathbbm{k}_{(1)}$, $({\mathbbm{k}_{(3)}}^\bullet\!\setminus\! \mathbbm{g}_2)\!\cup\!(\mathbbm{k}_{(3)}\cap\mathbbm{l}_{(2)})\!\subseteq\!\mathbbm{k}_{(2)}$,
$((\mathbbm{h}_2\cap\mathbbm{i}_2)\setminus\mathbbm{g}_2)\cup(\mathbbm{i}_2\cap\mathbbm{l}_{(3)})
\subseteq\mathbbm{k}_{(3)}$ hold together. Notice that $\mathbbm{i}_1\cap\mathbbm{l}_{(1)}\!=\!\mathbbm{g}_1\cap\mathbbm{k}_{(1)}\!=\!\mathbbm{i}_1\cap\mathbbm{j}_{(1)}$
and $\mathbbm{l}_{(1)}\setminus\mathbbm{i}_1=\mathbbm{j}_{(1)}\setminus\mathbbm{i}_1$ as $((\mathbbm{g}_1\cap\mathbbm{h}_1)^\circ\!\setminus\!\mathbbm{i}_1)\cup(\mathbbm{g}_1\cap\mathbbm{k}_{(1)})
\!\subseteq\!\mathbbm{j}_{(1)}$ by Lemma \ref{L;Lemma7.13}. Notice that $\mathbbm{i}_2\cap\mathbbm{l}_{(2)}=\mathbbm{g}_2\cap\mathbbm{i}_2 \cap\mathbbm{k}_{(3)}\cap\mathbbm{l}_{(2)}
=\mathbbm{g}_2\cap\mathbbm{i}_2\cap\mathbbm{j}_{(3)}\cap\mathbbm{k}_{(2)}=\mathbbm{i}_2\cap\mathbbm{j}_{(2)}$ and $\mathbbm{l}_{(2)}\setminus\mathbbm{i}_2\!=\!\mathbbm{j}_{(2)}\setminus\mathbbm{i}_2$ as $({\mathbbm{j}_{(3)}}^\bullet\!\setminus\! \mathbbm{i}_2)\cup(\mathbbm{j}_{(3)}\cap\mathbbm{k}_{(2)})\!\subseteq\!\mathbbm{j}_{(2)}$ and $\mathbbm{g}_2\cap\mathbbm{k}_{(3)}\subseteq\mathbbm{j}_{(3)}$ by Lemma \ref{L;Lemma7.13}. Notice that
$\mathbbm{i}_2\cap\mathbbm{l}_{(3)}=\mathbbm{g}_2\cap\mathbbm{k}_{(3)}=\mathbbm{i}_2\cap\mathbbm{j}_{(3)}$ and $\mathbbm{l}_{(3)}\setminus\mathbbm{i}_2=\mathbbm{j}_{(3)}\setminus\mathbbm{i}_2$ as $((\mathbbm{g}_2\!\cap\!\mathbbm{h}_2)\!\setminus\!\mathbbm{i}_2)\!\cup\!(\mathbbm{g}_2\!\cap\!\mathbbm{k}_{(3)})
\!\!\subseteq\!\!\mathbbm{j}_{(3)}$ by Lemma \ref{L;Lemma7.13}. The above discussion thus implies that $\mathfrak{j}=\mathfrak{l}$. As $p\nmid k_{[\mathbf{g}, \mathbf{h}, \mathbf{g}]}k_\mathfrak{j}$, notice that $D_{\mathbf{g}, \mathbf{h}, \mathfrak{j}}\neq O$. The desired theorem thus follows from Lemmas \ref{L;Lemma7.12} and \ref{L;Lemma5.14}.
\end{proof}
We conclude this section by presenting a remark of Theorems \ref{T;Theorem7.14} and \ref{T;Theorem7.15}.
\begin{rem}\label{R;Remark7.16}
The structure constants of the $\F$-basis of $\mathbb{T}/\mathrm{Rad}(\mathbb{T})$ in Theorem \ref{T;Theorem7.4} are read off from Theorems \ref{T;Theorem7.14} and \ref{T;Theorem7.15}. Notice that they are contained in $\{\overline{0}, \overline{1}\}$.
\end{rem}
\section{Algebraic structure of $\mathbb{T}$: Wedderburn-Artin decomposition}
In this section, we apply Theorems \ref{T;Theorem7.14} and \ref{T;Theorem7.15} to present the Wedderburn-Artin decomposition of $\mathbb{T}$. This means that we determine the algebraic structure of the semisimple $\F$-algebra $\mathbb{T}/\mathrm{Rad}(\mathbb{T})$ up to $\F$-algebra isomorphism. For this purpose, we recall Notations \ref{N;Notation3.7}, \ref{N;Notation3.8}, \ref{N;Notation3.14}, \ref{N;Notation3.15}, \ref{N;Notation3.18}, \ref{N;Notation5.6}, \ref{N;Notation5.7} and display a preliminary lemma.
\begin{lem}\label{L;Lemma8.1}
Assume that $\mathbf{g}, \mathbf{h}, \mathbf{i}\in\mathbb{E}$. Assume that $\mathfrak{j}=(\mathbbm{j}_{(1)}, \mathbbm{j}_{(2)}, \mathbbm{j}_{(3)})\in \mathbb{U}_{\mathbf{g}, \mathbf{h}}$ and $\mathfrak{k}=(\mathbbm{k}_{(1)}, \mathbbm{k}_{(2)}, \mathbbm{k}_{(3)})\in\mathbb{U}_{\mathbf{h}, \mathbf{i}}$. Assume that $\mathbbm{l}_{(1)}=((\mathbbm{g}_1\cap\mathbbm{i}_1)^\circ\setminus\mathbbm{h}_1)\cup(\mathbbm{g}_1\cap\mathbbm{i}_1\cap(\mathbbm{j}_{(1)}\cup\mathbbm{k}_{(1)}))$,
$\mathbbm{l}_{(2)}\!=\!((\mathbbm{g}_2\cap\mathbbm{i}_2)^\bullet\!\setminus\!\mathbbm{h}_2)\cup(\mathbbm{g}_2\cap\mathbbm{i}_2\cap(\mathbbm{j}_{(2)}\cup\mathbbm{k}_{(2)}))$,
$\mathbbm{l}_{(3)}\!=\!((\mathbbm{g}_2\cap\mathbbm{i}_2)\!\setminus\!\mathbbm{h}_2)\cup(\mathbbm{g}_2\cap\mathbbm{i}_2\cap(\mathbbm{j}_{(3)}\cup\mathbbm{k}_{(3)}))$,
$(\mathbbm{g}_1\!\cap\!\mathbbm{h}_1)^\circ\!\setminus\!\mathbbm{j}_{(1)}\!=\!
(\mathbbm{h}_1\!\cap\!\mathbbm{i}_1)^\circ\!\setminus\!\mathbbm{k}_{(1)}$,
$(\mathbbm{g}_2\!\cap\!\mathbbm{h}_2)^\bullet\!\setminus\!\mathbbm{j}_{(2)}\!=\!
(\mathbbm{h}_2\!\cap\!\mathbbm{i}_2)^\bullet\!\setminus\!\mathbbm{k}_{(2)}$,
$(\mathbbm{g}_2\!\cap\!\mathbbm{h}_2)\!\setminus\!\mathbbm{j}_{(3)}=(\mathbbm{h}_2\!\cap\!\mathbbm{i}_2)\!\setminus\!\mathbbm{k}_{(3)}$.
Then $(\mathbbm{g}_1\!\cap\!\mathbbm{h}_1)^\circ\!\setminus\!\mathbbm{j}_{(1)}\!\!=\!\!
(\mathbbm{h}_1\!\cap\!\mathbbm{i}_1)^\circ\!\setminus\!\mathbbm{k}_{(1)}\!\!=\!\!(\mathbbm{g}_1\!\cap\!\mathbbm{i}_1)^\circ\!\setminus\!\mathbbm{l}_{(1)}$,
$(\mathbbm{g}_2\!\cap\!\mathbbm{h}_2)^\bullet\!\setminus\!\mathbbm{j}_{(2)}\!\!=\!\!
(\mathbbm{h}_2\!\cap\!\mathbbm{i}_2)^\bullet\!\setminus\!\mathbbm{k}_{(2)}\!\!=\!\!(\mathbbm{g}_2\!\cap\!\mathbbm{i}_2)^\bullet\!\setminus\!\mathbbm{l}_{(2)}$, and
$(\mathbbm{g}_2\!\cap\!\mathbbm{h}_2)\!\setminus\!\mathbbm{j}_{(3)}\!\!=\!\!
(\mathbbm{h}_2\!\cap\!\mathbbm{i}_2)\!\setminus\!\mathbbm{k}_{(3)}\!\!=\!\!(\mathbbm{g}_2\!\cap\!\mathbbm{i}_2)\!\setminus\!\mathbbm{l}_{(3)}$.
\end{lem}
\begin{proof}
By a direct computation, notice that $(\mathbbm{g}_1\!\cap\!\mathbbm{i}_1)^\circ\!\setminus\!\mathbbm{l}_{(1)}=
(\mathbbm{g}_1\cap\mathbbm{h}_1\cap\mathbbm{i}_1)^\circ\setminus(\mathbbm{j}_{(1)}\cup\mathbbm{k}_{(1)})$, $(\mathbbm{g}_2\!\cap\!\mathbbm{i}_2)^\bullet\!\setminus\!\mathbbm{l}_{(2)}=
(\mathbbm{g}_2\cap\mathbbm{h}_2\cap\mathbbm{i}_2)^\bullet\setminus(\mathbbm{j}_{(2)}\cup\mathbbm{k}_{(2)})$, and
$(\mathbbm{g}_2\!\cap\!\mathbbm{i}_2)\!\setminus\!\mathbbm{l}_{(3)}=
(\mathbbm{g}_2\cap\mathbbm{h}_2\cap\mathbbm{i}_2)\setminus(\mathbbm{j}_{(3)}\cup\mathbbm{k}_{(3)})$. Notice that $(\mathbbm{g}_1\!\cap\!\mathbbm{h}_1)^\circ\!\setminus\!\mathbbm{j}_{(1)}=
(\mathbbm{g}_1\cap\mathbbm{h}_1\cap\mathbbm{i}_1)^\circ\setminus(\mathbbm{j}_{(1)}\cup\mathbbm{k}_{(1)})$ as $((\mathbbm{g}_1\cap\mathbbm{h}_1)^\circ\!\setminus\!\mathbbm{i}_1)\cup(\mathbbm{g}_1\cap\mathbbm{k}_{(1)})
\!\subseteq\!\mathbbm{j}_{(1)}$ by Lemma \ref{L;Lemma7.13}. Notice that $(\mathbbm{g}_2\!\cap\!\mathbbm{h}_2)^\bullet\!\setminus\!\mathbbm{i}_2\subseteq\mathbbm{j}_{(2)}$ as  $(\mathbbm{g}_2\!\cap\!\mathbbm{h}_2)^\bullet\!\setminus\!\mathbbm{j}_{(2)}\!\!=\!\!
(\mathbbm{h}_2\!\cap\!\mathbbm{i}_2)^\bullet\!\setminus\!\mathbbm{k}_{(2)}$. As $\mathbbm{g}_2\!\cap\!\mathbbm{k}_{(3)}\!\subseteq\!\mathbbm{j}_{(3)}$ and
$\mathbbm{j}_{(3)}\!\cap\!\mathbbm{k}_{(2)}\!\subseteq\!\mathbbm{j}_{(2)}$ by Lemma \ref{L;Lemma7.13}, $(\mathbbm{g}_2\!\cap\!\mathbbm{h}_2)^\bullet\!\setminus\!\mathbbm{j}_{(2)}\!\!=\!\!
(\mathbbm{g}_2\cap\mathbbm{h}_2\cap\mathbbm{i}_2)^\bullet\!\setminus\!(\mathbbm{j}_{(2)}\!\cup\!\mathbbm{k}_{(2)})$. Notice that
$(\mathbbm{g}_2\!\cap\!\mathbbm{h}_2)\setminus\mathbbm{j}_{(3)}\!=\!
(\mathbbm{g}_2\cap\mathbbm{h}_2\cap\mathbbm{i}_2)\setminus(\mathbbm{j}_{(3)}\!\cup\!\mathbbm{k}_{(3)})$ as $((\mathbbm{g}_2\cap\mathbbm{h}_2)\setminus\mathbbm{i}_2)\cup(\mathbbm{g}_2\cap\mathbbm{k}_{(3)})\subseteq\mathbbm{j}_{(3)}$ by Lemma \ref{L;Lemma7.13}. The desired lemma thus follows from the above discussion.
\end{proof}
Lemma \ref{L;Lemma8.1} motivates us to present the following notations and another lemma.
\begin{nota}\label{N;Notation8.2}
Define $\mathbb{D}=\{(\mathbf{a}, \mathbf{b}, \mathfrak{c}): \mathbf{a}, \mathbf{b}\in \mathbb{E}, \mathfrak{c}\in\mathbb{U}_{\mathbf{a}, \mathbf{b}}, p\nmid k_{[\mathbf{a}, \mathbf{b}, \mathbf{a}]}k_{[\mathbf{b}, \mathbf{a}, \mathbf{b}]}k_\mathfrak{c}\}$. Notice that $\mathbb{D}\!\neq\!\varnothing$ as $(0_\mathfrak{S}, 0_\mathfrak{S}, \mathfrak{o})\!\in\!\mathbb{D}$. Assume that $(\mathbf{g}, \mathbf{h}, \mathfrak{i})$, $(\mathbf{j}, \mathbf{k}, \mathfrak{l})\!\in\!\mathbb{D}$, $\mathfrak{i}\!=\!(\mathbbm{i}_{(1)}, \mathbbm{i}_{(2)}, \mathbbm{i}_{(3)})$, and $\mathfrak{l}=(\mathbbm{l}_{(1)}, \mathbbm{l}_{(2)}, \mathbbm{l}_{(3)})$. Then $D_{\mathbf{g}, \mathbf{h}, \mathfrak{i}} $ is a defined nonzero matrix. Write $(\mathbf{g}, \mathbf{h}, \mathfrak{i})\sim(\mathbf{j}, \mathbf{k}, \mathfrak{l})$ if and only if $(\mathbbm{g}_1\cap\mathbbm{h}_1)^\circ\setminus\mathbbm{i}_{(1)}=(\mathbbm{j}_1\cap\mathbbm{k}_1)^\circ\setminus\mathbbm{l}_{(1)}$,
$(\mathbbm{g}_2\cap\mathbbm{h}_2)^\bullet\setminus\mathbbm{i}_{(2)}=(\mathbbm{j}_2\cap\mathbbm{k}_2)^\bullet\setminus\mathbbm{l}_{(2)}$,
$(\mathbbm{g}_2\cap\mathbbm{h}_2)\setminus\mathbbm{i}_{(3)}\!=\!(\mathbbm{j}_2\cap\mathbbm{k}_2)\setminus\mathbbm{l}_{(3)}$.
Therefore $\sim$ is an equivalence relation on $\mathbb{D}$. There is $n_\sim\in\mathbb{N}$ such that $\mathbb{D}_1, \mathbb{D}_2,\ldots, \mathbb{D}_{n_\sim}$ are exactly all equivalence classes of $\mathbb{D}$ with respect to $\sim$. Assume that $m\!\in\! [1, n_\sim]$. Write $\mathbb{D}(m)\!=\!\{\mathbf{a}: \mathbf{a}\in \mathbb{E}, \exists\ \mathfrak{b}\in \mathbb{U}_{\mathbf{a}, \mathbf{a}}, (\mathbf{a}, \mathbf{a}, \mathfrak{b})\in\mathbb{D}_m\}$.
\end{nota}
\begin{nota}\label{N;Notation8.3}
Assume that $g\in[1,n_\sim]$. $\langle\{D_{\mathbf{a}, \mathbf{b}, \mathfrak{c}}+\mathrm{Rad}(\mathbb{T}): (\mathbf{a}, \mathbf{b}, \mathfrak{c})\in\mathbb{D}_g\}\rangle_{\mathbb{T}/\mathrm{Rad}(\mathbb{T})}$ is denoted by $\mathbb{I}(g)$.
According to Theorem \ref{T;Theorem7.4}, the $\F$-dimension of $\mathbb{I}(g)$ equals $|\mathbb{D}_g|$.
\end{nota}
\begin{lem}\label{L;Lemma8.4}
Assume that $g\in[1, n_\sim]$. Then $\mathbb{I}(g)$ is a two-sided ideal of $\mathbb{T}/\mathrm{Rad}(\mathbb{T})$. Moreover,  $\mathbb{T}/\mathrm{Rad}(\mathbb{T})$ is a direct sum of the $\F$-linear subspaces $\mathbb{I}(1), \mathbb{I}(2),\ldots, \mathbb{I}(n_\sim)$.
\end{lem}
\begin{proof}
The first statement is from combining Theorems \ref{T;Theorem7.4}, \ref{T;Theorem7.14}, \ref{T;Theorem7.15}, Lemma \ref{L;Lemma8.1}.
As $\{\mathbb{D}_1, \mathbb{D}_2, \ldots, \mathbb{D}_{n_\sim}\}$ is a partition of $\mathbb{D}$, the desired lemma is from Theorem \ref{T;Theorem7.4}.
\end{proof}
The following four lemmas focus on the investigation of objects in Notation \ref{N;Notation8.2}.
\begin{lem}\label{L;Lemma8.5}
Assume that $g\in [1, n_\sim]$ and $(\mathbf{h}, \mathbf{i}, \mathfrak{j}), (\mathbf{h}, \mathbf{i}, \mathfrak{k})\in\mathbb{D}_g$. Then $\mathfrak{j}=\mathfrak{k}$.
\end{lem}
\begin{proof}
Assume that $\mathfrak{j}\!=\!(\mathbbm{j}_{(1)}, \mathbbm{j}_{(2)}, \mathbbm{j}_{(3)})$ and $\mathfrak{k}\!\!=\!\!(\mathbbm{k}_{(1)}, \mathbbm{k}_{(2)}, \mathbbm{k}_{(3)})$. As $(\mathbf{h}, \mathbf{i}, \mathfrak{j}), (\mathbf{h}, \mathbf{i}, \mathfrak{k})\!\in\!\mathbb{D}_g$, notice that $\mathfrak{j}, \mathfrak{k}\in\mathbb{U}_{\mathbf{h}, \mathbf{i}}$, $(\mathbbm{h}_1\cap\mathbbm{i}_1)^\circ\setminus \mathbbm{j}_{(1)}\!=\!(\mathbbm{h}_1\cap\mathbbm{i}_1)^\circ\setminus \mathbbm{k}_{(1)}$,
$(\mathbbm{h}_2\cap\mathbbm{i}_2)^\bullet\!\setminus \!\mathbbm{j}_{(2)}\!=\!(\mathbbm{h}_2\cap\mathbbm{i}_2)^\bullet\!\setminus\! \mathbbm{k}_{(2)}$, and $(\mathbbm{h}_2\cap\mathbbm{i}_2)\setminus \mathbbm{j}_{(3)}=(\mathbbm{h}_2\cap\mathbbm{i}_2)\setminus \mathbbm{k}_{(3)}$. So $\mathbbm{j}_{(1)}=(\mathbbm{h}_1\cap\mathbbm{i}_1)^\circ\cap\mathbbm{j}_{(1)}=
(\mathbbm{h}_1\cap\mathbbm{i}_1)^\circ\cap\mathbbm{k}_{(1)}=\mathbbm{k}_{(1)}$, $\mathbbm{j}_{(2)}\!=\!(\mathbbm{h}_2\cap\mathbbm{i}_2)^\bullet\cap\mathbbm{j}_{(2)}\!=\!
(\mathbbm{h}_2\cap\mathbbm{i}_2)^\bullet\cap\mathbbm{k}_{(2)}\!=\!\mathbbm{k}_{(2)}$,
$\mathbbm{j}_{(3)}\!=\!(\mathbbm{h}_2\cap\mathbbm{i}_2)\cap\mathbbm{j}_{(3)}\!=\!
(\mathbbm{h}_2\cap\mathbbm{i}_2)\cap\mathbbm{k}_{(3)}\!=\!\mathbbm{k}_{(3)}$ since $\mathfrak{j}, \mathfrak{k}\in\mathbb{U}_{\mathbf{h}, \mathbf{i}}$.
The desired lemma thus follows from the above discussion.
\end{proof}
\begin{lem}\label{L;Lemma8.6}
Assume that $g\in [1, n_\sim]$ and $(\mathbf{h}, \mathbf{i}, \mathfrak{j})\in \mathbb{D}_g$. Then there are $\mathfrak{k}\in\mathbb{U}_{\mathbf{h}, \mathbf{h}}$ and
$\mathfrak{l}\in\mathbb{U}_{\mathbf{i},\mathbf{i}}$ such that $(\mathbf{h}, \mathbf{h}, \mathfrak{k}), (\mathbf{i}, \mathbf{i}, \mathfrak{l})\in\mathbb{D}_g$. In particular, $\mathbf{h}, \mathbf{i}\in\mathbb{D}(g)$.
\end{lem}
\begin{proof}
Assume that $\mathfrak{j}=(\mathbbm{j}_{(1)}, \mathbbm{j}_{(2)}, \mathbbm{j}_{(3)})$, $\mathfrak{k}=(\mathbbm{k}_{(1)}, \mathbbm{k}_{(2)}, \mathbbm{k}_{(3)})$, $\mathfrak{l}=(\mathbbm{l}_{(1)}, \mathbbm{l}_{(2)}, \mathbbm{l}_{(3)})$, where
$\mathbbm{k}_{(1)}=({\mathbbm{h}_1}^\circ\setminus\mathbbm{i}_1)\cup\mathbbm{j}_{(1)}$, $\mathbbm{k}_{(2)}=({\mathbbm{h}_2}^\bullet\setminus\mathbbm{i}_2)\cup\mathbbm{j}_{(2)}$, $\mathbbm{k}_{(3)}=(\mathbbm{h}_2\setminus\mathbbm{i}_2)\cup\mathbbm{j}_{(3)}$, $\mathbbm{l}_{(1)}=({\mathbbm{i}_1}^\circ\setminus\mathbbm{h}_1)\cup\mathbbm{j}_{(1)}$,
$\mathbbm{l}_{(2)}=({\mathbbm{i}_2}^\bullet\setminus\mathbbm{h}_2)\cup\mathbbm{j}_{(2)}$, and $\mathbbm{l}_{(3)}=({\mathbbm{i}_2}\setminus\mathbbm{h}_2)\cup\mathbbm{j}_{(3)}$. As $(\mathbf{h}, \mathbf{i},\mathfrak{j})\in\mathbb{D}_g$, notice that $\mathfrak{j}\in \mathbb{U}_{\mathbf{h}, \mathbf{i}}$ and $p\nmid k_{[\mathbf{h}, \mathbf{i}, \mathbf{h}]}k_{[\mathbf{i}, \mathbf{h}, \mathbf{i}]}k_\mathfrak{j}$. Hence $\mathfrak{k}\in\mathbb{U}_{\mathbf{h}, \mathbf{h}}$, $\mathfrak{l}\in\mathbb{U}_{\mathbf{i},\mathbf{i}}$, and $p\nmid k_\mathfrak{k}k_\mathfrak{l}$ by a direct computation.

Notice that $(\mathbbm{h}_1\cap\mathbbm{i}_1)^\circ\setminus\mathbbm{j}_{(1)}
\!=\!{\mathbbm{h}_1}^\circ\setminus\mathbbm{k}_{(1)}\!=\!{\mathbbm{i}_1}^\circ\setminus\mathbbm{l}_{(1)}$,
$(\mathbbm{h}_2\cap\mathbbm{i}_2)^\bullet\setminus\mathbbm{j}_{(2)}
={\mathbbm{h}_2}^\bullet\setminus\mathbbm{k}_{(2)}={\mathbbm{i}_2}^\bullet\setminus\mathbbm{l}_{(2)}$, and
$(\mathbbm{h}_2\cap\mathbbm{i}_2)\setminus\mathbbm{j}_{(3)}
={\mathbbm{h}_2}\setminus\mathbbm{k}_{(3)}={\mathbbm{i}_2}\setminus\mathbbm{l}_{(3)}$. The desired lemma thus follows.
\end{proof}
\begin{lem}\label{L;Lemma8.7}
Assume that $g\in [1, n_\sim]$ and $(\mathbf{h}, \mathbf{h}, \mathfrak{i}), (\mathbf{j}, \mathbf{j}, \mathfrak{k})\in\mathbb{D}_g$. Then there exists $\mathfrak{l}\in \mathbb{U}_{\mathbf{h}, \mathbf{j}}$ such that $(\mathbf{h}, \mathbf{j}, \mathfrak{l})\in\mathbb{D}_g$.
\end{lem}
\begin{proof}
Assume that $\mathfrak{i}\!=\!(\mathbbm{i}_{(1)}, \mathbbm{i}_{(2)}, \mathbbm{i}_{(3)})$ and $\mathfrak{k}\!\!=\!\!(\mathbbm{k}_{(1)}, \mathbbm{k}_{(2)}, \mathbbm{k}_{(3)})$. As $(\mathbf{h}, \mathbf{h}, \mathfrak{i}), (\mathbf{j}, \mathbf{j}, \mathfrak{k})\!\in\!\mathbb{D}_g$, notice that $p\nmid k_\mathfrak{i}k_\mathfrak{k}$, ${\mathbbm{h}_1}^\circ\setminus\mathbbm{i}_{(1)}={\mathbbm{j}_1}^\circ\setminus\mathbbm{k}_{(1)}$,
${\mathbbm{h}_2}^\bullet\setminus\mathbbm{i}_{(2)}={\mathbbm{j}_2}^\bullet\setminus\mathbbm{k}_{(2)}$, and ${\mathbbm{h}_2}\setminus\mathbbm{i}_{(3)}={\mathbbm{j}_2}\setminus\mathbbm{k}_{(3)}$. Therefore  ${\mathbbm{h}_1}^\circ\setminus\mathbbm{j}_1\subseteq\mathbbm{i}_{(1)}$, ${{\mathbbm{h}_2}}^\bullet\setminus\mathbbm{j}_2\subseteq
\mathbbm{i}_{(2)}$, $\mathbbm{h}_2\setminus({\mathbbm{h}_2}^\bullet\cup\mathbbm{j}_2)\subseteq\mathbbm{i}_{(3)}\setminus\mathbbm{i}_{(2)}$,
${\mathbbm{j}_1}^\circ\setminus\mathbbm{h}_1\subseteq\mathbbm{k}_{(1)}$, ${\mathbbm{j}_2}^\bullet\setminus\mathbbm{h}_2\subseteq\mathbbm{k}_{(2)}$, and $\mathbbm{j}_2\setminus(\mathbbm{h}_2\cup{\mathbbm{j}_2}^\bullet)\subseteq\mathbbm{k}_{(3)}\setminus\mathbbm{k}_{(2)}$.
Hence $p\nmid k_{[\mathbf{h}, \mathbf{j}, \mathbf{h}]}k_{[\mathbf{j}, \mathbf{h}, \mathbf{j}]}$ as $p\nmid k_\mathfrak{i}k_\mathfrak{k}$ and Lemma \ref{L;Lemma3.19} holds.
Define $\mathfrak{l}\!=\!\mathbf{h}\cap\mathfrak{k}$. So $\mathfrak{l}\!\in\!\mathbb{U}_{\mathbf{h}, \mathbf{j}}$ and $p\nmid k_{[\mathbf{h}, \mathbf{j}, \mathbf{h}]}k_{[\mathbf{j}, \mathbf{h}, \mathbf{j}]}k_\mathfrak{l}$ by Lemma \ref{L;Lemma3.20}.

Assume that $\mathfrak{l}=(\mathbbm{l}_{(1)}, \mathbbm{l}_{(2)}, \mathbbm{l}_{(3)})$. Notice that ${\mathbbm{h}_1}^\circ\setminus\mathbbm{i}_{(1)}={\mathbbm{j}_1}^\circ\setminus\mathbbm{k}_{(1)}=
(\mathbbm{h}_1\cap\mathbbm{j}_1)^\circ\setminus\mathbbm{l}_{(1)}$, ${\mathbbm{h}_2}^\bullet\setminus\mathbbm{i}_{(2)}={\mathbbm{j}_2}^\bullet\setminus\mathbbm{k}_{(2)}=
(\mathbbm{h}_2\cap\mathbbm{j}_2)^\bullet\setminus\mathbbm{l}_{(2)}$, and ${\mathbbm{h}_2}\setminus\mathbbm{i}_{(3)}={\mathbbm{j}_2}\setminus\mathbbm{k}_{(3)}=
(\mathbbm{h}_2\cap\mathbbm{j}_2)\setminus\mathbbm{l}_{(3)}$ by a direct computation. The desired lemma thus follows from the above discussion.
\end{proof}
Lemmas \ref{L;Lemma8.5}, \ref{L;Lemma8.6}, \ref{L;Lemma8.7} allow us to present the following lemma and another notation.
\begin{lem}\label{L;Lemma8.8}
Assume that $g\in [1, n_\sim]$. Then the cartesian product $\mathbb{D}(g)\times\mathbb{D}(g)\neq\varnothing$. Moreover, the map that sends $(\mathbf{h}, \mathbf{i}, \mathfrak{j})$ to $(\mathbf{h}, \mathbf{i})$ is bijective from $\mathbb{D}_g$ to $\mathbb{D}(g)\times\mathbb{D}(g)$. In particular, the $\F$-dimension of $\mathbb{I}(g)$ is equal to $|\mathbb{D}(g)|^2$.
\end{lem}
\begin{proof}
As $\mathbb{D}_g\neq\varnothing$ and Lemma \ref{L;Lemma8.6} holds, it is clear to see that $\mathbb{D}(g)\times\mathbb{D}(g)\neq\varnothing$. The desired lemma thus follows from combining Lemmas \ref{L;Lemma8.5}, \ref{L;Lemma8.6}, and \ref{L;Lemma8.7}.
\end{proof}
\begin{nota}\label{N;Notation8.9}
Assume that $g\in [1, n_\sim]$. According to Lemma \ref{L;Lemma8.8} and Theorem \ref{T;Theorem7.4}, there is a unique $D_{\mathbf{h}, \mathbf{i},\mathfrak{j}}\!+\!\mathrm{Rad}(\mathbb{T})\!\in\!\{D_{\mathbf{a}, \mathbf{b}, \mathfrak{c}}\!+\!\mathrm{Rad}(\mathbb{T}): (\mathbf{a}, \mathbf{b}, \mathfrak{c})\in \mathbb{D}_g\}$ for any $\mathbf{h}, \mathbf{i}\in\mathbb{D}(g)$. This nonzero element in $\mathbb{T}/\mathrm{Rad}(\mathbb{T})$ is denoted by $D_{\mathbf{h}, \mathbf{i}}(g)$ for any $\mathbf{h}, \mathbf{i}\in\mathbb{D}(g)$. For completeness, define $D_{\mathbf{h}, \mathbf{i}}(g)=O+\mathrm{Rad}(\mathbb{T})$ for any $\mathbf{h}, \mathbf{i}\in \mathbb{E}$ and $\{\mathbf{h}, \mathbf{i}\}\not\subseteq\mathbb{D}(g)$. Hence $\mathbb{I}(g)$ has an $\F$-basis $\{D_{\mathbf{a}, \mathbf{b}}(g): \mathbf{a}, \mathbf{b}\in\mathbb{D}(g)\}$ by Lemma \ref{L;Lemma8.8} and Theorem \ref{T;Theorem7.4}.
\end{nota}
The following lemma gives the computational rule of the objects in Notation \ref{N;Notation8.9}.
\begin{lem}\label{L;Lemma8.10}
$\mathbb{T}/\mathrm{Rad}(\mathbb{T})$ has an $\F$-basis $\{D_{\mathbf{b}, \mathbf{c}}(a): a\in[1, n_\sim], \mathbf{b}, \mathbf{c}\in \mathbb{D}(a)\}$. If $g, h\in [1, n_\sim]$, $\mathbf{i}, \mathbf{j}\in\mathbb{D}(g)$, and $\mathbf{k}, \mathbf{l}\in\mathbb{D}(h)$, then $D_{\mathbf{i}, \mathbf{j}}(g)D_{\mathbf{k}, \mathbf{l}}(h)=\delta_{g, h}\delta_{\mathbf{j}, \mathbf{k}}D_{\mathbf{i},\mathbf{l}}(g)$.
\end{lem}
\begin{proof}
The first statement is from Lemma \ref{L;Lemma8.4}. By Lemma \ref{L;Lemma8.4}, there is no loss to assume that $g=h$.
Then there are unique $\mathfrak{m}\in\mathbb{U}_{\mathbf{i}, \mathbf{j}}$, $\mathfrak{q}\in\mathbb{U}_{\mathbf{k}, \mathbf{l}}$, $\mathfrak{r}\in\mathbb{U}_{\mathbf{i}, \mathbf{l}}$ such
that $(\mathbf{i}, \mathbf{j}, \mathfrak{m}), (\mathbf{k}, \mathbf{l}, \mathfrak{q}), (\mathbf{i}, \mathbf{l}, \mathfrak{r})\!\in\!\mathbb{D}_g$, $D_{\mathbf{i}, \mathbf{j}}(g)\!\!=\!\!D_{\mathbf{i}, \mathbf{j}, \mathfrak{m}}\!+\!\mathrm{Rad}(\mathbb{T})$, $D_{\mathbf{k}, \mathbf{l}}(g)\!\!=\!\!D_{\mathbf{k}, \mathbf{l}, \mathfrak{q}}\!\!+\!\!\mathrm{Rad}(\mathbb{T})$, and $D_{\mathbf{i}, \mathbf{l}}(g)\!=\!D_{\mathbf{i}, \mathbf{l}, \mathfrak{r}}+\mathrm{Rad}(\mathbb{T})$. By Theorem \ref{T;Theorem7.14}, there is no loss to assume that $\mathbf{j}=\mathbf{k}$. As $(\mathbf{i}, \mathbf{j}, \mathfrak{m}), (\mathbf{j}, \mathbf{l}, \mathfrak{q}), (\mathbf{i}, \mathbf{l}, \mathfrak{r})\in\mathbb{D}_g$, notice that all conditions in Theorem \ref{T;Theorem7.15} are satisfied. The desired lemma follows from an application of Theorem \ref{T;Theorem7.15} and Lemma \ref{L;Lemma8.5}.
\end{proof}
The computational rule in Lemma \ref{L;Lemma8.10} allows us to prove the following lemmas.
\begin{lem}\label{L;Lemma8.11}
Assume that $g\in[1, n_\sim]$. Then $\langle\{D_{\mathbf{a}, \mathbf{h}}(g): \mathbf{a}\in\mathbb{D}(g)\}\rangle_{\mathbb{T}/\mathrm{Rad}(\mathbb{T})}$ is an irreducible $\mathbb{T}$-module for any
$\mathbf{h}\in\mathbb{D}(g)$. If $\mathbf{h}, \mathbf{i}\in\mathbb{D}(g)$, then there is a $\mathbb{T}$-module isomorphism from $\langle\{D_{\mathbf{a}, \mathbf{h}}(g): \mathbf{a}\!\in\!\mathbb{D}(g)\}\rangle_{\mathbb{T}/\mathrm{Rad}(\mathbb{T})}$ to
$\langle\{D_{\mathbf{a}, \mathbf{i}}(g): \mathbf{a}\in\mathbb{D}(g)\}\rangle_{\mathbb{T}/\mathrm{Rad}(\mathbb{T})}$.
\end{lem}
\begin{proof}
The first statement is from Lemma \ref{L;Lemma8.10}. According to the first statement and Lemma \ref{L;Lemma8.10}, the map that sends $D_{\mathbf{j}, \mathbf{h}}(g)$ to $D_{\mathbf{j}, \mathbf{i}}(g)$ for any $\mathbf{j}\in\mathbb{D}(g)$ is a desired $\mathbb{T}$-module isomorphism from the $\mathbb{T}$-module $\langle\{D_{\mathbf{a}, \mathbf{h}}(g): \mathbf{a}\in\mathbb{D}(g)\}\rangle_{\mathbb{T}/\mathrm{Rad}(\mathbb{T})}$ to the $\mathbb{T}$-module $\langle\{D_{\mathbf{a}, \mathbf{i}}(g): \mathbf{a}\in\mathbb{D}(g)\}\rangle_{\mathbb{T}/\mathrm{Rad}(\mathbb{T})}$. The desired lemma thus follows.
\end{proof}
\begin{lem}\label{L;Lemma8.12}
Assume that $g\in [1, n_\sim]$. Then $\mathbb{I}(g)\cong \mathrm{M}_{|\mathbb{D}(g)|}(\F)$ as $\F$-algebras.
\end{lem}
\begin{proof}
It suffices to check that $\mathbb{I}(g)\cong\mathrm{M}_{\mathbb{D}(g)}(\F)$ as $\F$-algebras. If $\mathbf{h}, \mathbf{i}\in\mathbb{D}(g)$, let $E_{(\mathbf{h}, \mathbf{i})}$ be the $\{\overline{0}, \overline{1}\}$-matrix in $\mathrm{M}_{\mathbb{D}(g)}(\F)$ whose unique nonzero entry is the $(\mathbf{h}, \mathbf{i})$-entry. In particular,
$E_{(\mathbf{h}, \mathbf{i})}E_{(\mathbf{j}, \mathbf{k})}=\delta_{\mathbf{i},\mathbf{j}}E_{(\mathbf{h}, \mathbf{k})}$ for any $\mathbf{h}, \mathbf{i}, \mathbf{j}, \mathbf{k}\in \mathbb{D}(g)$. By Lemma \ref{L;Lemma8.10}, the map that
sends $D_{\mathbf{h}, \mathbf{i}}(g)$ to $E_{(\mathbf{h}, \mathbf{i})}$ for any $\mathbf{h}, \mathbf{i}\in \mathbb{D}(g)$ is
an $\F$-algebra isomorphism from $\mathbb{I}(g)$ to $\mathrm{M}_{\mathbb{D}(g)}(\F)$. The desired lemma thus follows from the above discussion.
\end{proof}
We are now ready to deduce the final main result of this paper and two corollaries.
\begin{thm}\label{T;Decomposition}
The number of all pairwise nonisomorphic irreducible $\mathbb{T}$-modules equals $n_\sim$. Moreover, $$\mathbb{T}/\mathrm{Rad}(\mathbb{T})\!\cong\!\bigoplus_{g=1}^{n_\sim}\mathrm{M}_{|\mathbb{D}(g)|}(\F)\ \text{as $\F$-algebras.}$$
\end{thm}
\begin{proof}
The desired theorem follows from combining Lemmas \ref{L;Lemma8.4}, \ref{L;Lemma8.12}, and \ref{L;Lemma2.6}.
\end{proof}
\begin{cor}\label{C;Corollary8.14}
Assume that $g, h\!\in\! [1, n_\sim]$ and $\mathbbm{Irr}(g)$ denotes the unique irreducible $\mathbb{T}$-module up to $\mathbb{T}$-module isomorphism whose definition is in Lemma \ref{L;Lemma8.11}. Then $\mathbbm{Irr}(g)\!\cong\!\mathbbm{Irr}(h)$ as $\mathbb{T}$-modules if and only if $g=h$. Moreover, $\{\mathbbm{Irr}(a): a\in [1, n_\sim]\}$ comprises a complete system of distinct representatives of all isomorphism classes of irreducible $\mathbb{T}$-modules, and the irreducible $\mathbb{T}$-modules are precisely the absolutely irreducible $\mathbb{T}$-modules.
\end{cor}
\begin{proof}
The first statement follows from Lemmas \ref{L;Lemma8.11} and \ref{L;Lemma8.10}. The desired corollary thus follows from combining the first statement, Theorem \ref{T;Decomposition}, and Lemma \ref{L;Lemma2.6}.
\end{proof}
\begin{cor}\label{C;Corollary8.15}
Assume that $\mathfrak{S}$ is a $p'$-valenced scheme. Then $$\mathbb{T}\cong\bigoplus_{g=1}^{n_\sim}\mathrm{M}_{|\mathbb{D}(g)|}(\F)\
\text{as $\F$-algebras.}$$
\end{cor}
\begin{proof}
The desired corollary is a direct application of Theorems \ref{T;Semisimplicity} and \ref{T;Decomposition}.
\end{proof}
We close the whole paper by some examples of Theorem \ref{T;Decomposition} and Corollary \ref{C;Corollary8.15}.
\begin{eg}\label{E;Example8.16}
Assume that $n=1$ and $\ell_1=m_1=2$. If $p=2$, Theorem \ref{T;Decomposition} implies that $\mathbb{T}/\mathrm{Rad}(\mathbb{T})\cong\mathrm{M}_2(\F)\oplus\mathrm{M}_1(\F)$ as $\F$-algebras. If $p\neq 2$, Lemma \ref{L;Lemma3.19} and Corollary \ref{C;Corollary8.15} imply that $\mathfrak{S}$ is a $p'$-valenced scheme and $\mathbb{T}\cong\mathrm{M}_3(\F)\oplus\mathrm{M}_1(\F)$ as $\F$-algebras. Then $n_\sim=2$ by Theorem \ref{T;Decomposition}. So $n_\sim$ is independent of the choice of $p$.
\end{eg}
\begin{eg}\label{E;Example8.17}
Assume that $n=1$ and $\ell_1>m_1=2$. If $p=2$, Theorem \ref{T;Decomposition} implies that $\mathbb{T}/\mathrm{Rad}(\mathbb{T})\cong\mathrm{M}_2(\F)\oplus\mathrm{M}_1(\F)$ as $\F$-algebras. If $p>2$ and $p\mid \ell_1-1$,
Theorem \ref{T;Decomposition} implies that $\mathbb{T}/\mathrm{Rad}(\mathbb{T})\cong\mathrm{M}_2(\F)\oplus2\mathrm{M}_1(\F)$ as $\F$-algebras. For the remaining case $p\nmid 2(\ell_1-1)$, Lemma \ref{L;Lemma3.19} and Corollary \ref{C;Corollary8.15} imply that $\mathfrak{S}$ is a $p'$-valenced scheme and $\mathbb{T}\cong\mathrm{M}_3(\F)\oplus2\mathrm{M}_1(\F)$ as $\F$-algebras. Theorem \ref{T;Decomposition} also gives
\[n_\sim=\begin{cases} 2, & \text{if $p=2$},\\
3, &\text{if $p\neq2$}.
\end{cases}\]
\end{eg}
\begin{eg}\label{E;Example8.18}
Assume that $n=1$ and $m_1>\ell_1=2$. If $p\mid (m_1-1)m_1$, Theorem \ref{T;Decomposition} implies that $\mathbb{T}/\mathrm{Rad}(\mathbb{T})\cong\mathrm{M}_2(\F)\oplus2\mathrm{M}_1(\F)$ as $\F$-algebras. For the other case $p\nmid (m_1-1)m_1$, Lemma \ref{L;Lemma3.19} and Corollary \ref{C;Corollary8.15} imply that $\mathfrak{S}$ is a $p'$-valenced scheme and
$\mathbb{T}\cong\mathrm{M}_3(\F)\oplus2\mathrm{M}_1(\F)$ as $\F$-algebras. According to Theorem \ref{T;Decomposition}, it is obvious to see that $n_\sim\!=\!3$. This statement implies that $n_\sim$ is independent of the choice of $p$.
\end{eg}
\begin{eg}\label{E;Example8.19}
Assume that $n=1$ and $\min\{\ell_1, m_1\}>2$. If $p\mid m_1$, Theorem \ref{T;Decomposition} implies that $\mathbb{T}/\mathrm{Rad}(\mathbb{T})\!\!\cong\!\!\mathrm{M}_2(\F)\!\oplus\!2\mathrm{M}_1(\F)$ as $\F$-algebras. If $p\!\mid\! \ell_1-1$ and $p\mid m_1-1$, Theorem \ref{T;Decomposition} implies that $\mathbb{T}/\mathrm{Rad}(\mathbb{T})\!\cong\!4\mathrm{M}_1(\F)$ as $\F$-algebras. If $p\nmid (\ell_1\!-\!1)(m_1\!-\!1)m_1$, Theorem \ref{T;Semisimplicity} and Corollary \ref{C;Corollary8.15} thus imply that $\mathfrak{S}$ is a $p'$-valenced scheme and $\mathbb{T}\cong\mathrm{M}_3(\F)\oplus3\mathrm{M}_1(\F)$ as $\F$-algebras. For the remaining case, Theorem \ref{T;Decomposition} implies that $\mathbb{T}/\mathrm{Rad}(\mathbb{T})\cong\mathrm{M}_2(\F)\oplus3\mathrm{M}_1(\F)$ as $\F$-algebras. Moreover, Theorem \ref{T;Decomposition} also gives
\[n_\sim=\begin{cases} 3, & \text{if $p\mid m_1$},\\
4, &\text{if $p\nmid m_1$}.
\end{cases}\]
\end{eg}
\begin{eg}\label{E;Example8.20}
Assume that $n=\ell_1=2$ and $\ell_2=m_1=m_2=3$. If $p=2$, Theorem \ref{T;Decomposition} implies that $\mathbb{T}/\mathrm{Rad}(\mathbb{T})\cong4\mathrm{M}_2(\F)\oplus 8\mathrm{M}_1(\F)$ as $\F$-algebras. If $p=3$, Theorem \ref{T;Decomposition} implies that $\mathbb{T}/\mathrm{Rad}(\mathbb{T})\cong\mathrm{M}_4(\F)\oplus 4\mathrm{M}_2(\F)\oplus 4\mathrm{M}_1(\F)$ as $\F$-algebras. If $p\notin[2,3]$, the combination of Example \ref{E;Example5.23}, Theorem \ref{T;Semisimplicity},
Corollary \ref{C;Corollary8.15} imply that $\mathfrak{S}$ is a $p'$-valenced scheme and $\mathbb{T}\!\cong\!\mathrm{M}_9(\F)\!\oplus\! 5\mathrm{M}_3(\F)\!\oplus\! 6\mathrm{M}_1(\F)$ as $\F$-algebras. By Theorem \ref{T;Decomposition},
\[n_\sim=\begin{cases} 9, & \text{if $p=3$},\\
12, &\text{if $p\neq 3$}.
\end{cases}\]
\end{eg}
\subsection*{Disclosure statement} No relevant financial or nonfinancial interests are reported.
\subsection*{Data availability} All data used in this present study are included in this paper.
\subsection*{Acknowledgements}
This present research was supported by the National Natural Science Foundation of China (Youth Program, No. 12301017) and Anhui Provincial Natural Science Foundation (Youth Program, No. \!2308085QA01). The author would like to gratefully thank Prof. Tatsuro Ito and Prof. Gang Chen for encouragements. The author would like to gratefully thank a referee for careful reading and many valuable comments. These comments largely improve the readability of this paper.



\begin{thebibliography}{99}
\bibitem{B} R. A. Bailey, Association Schemes: Designed Experiments, Algebra and Combinatorics, Cambridge Stud. Adv. Math., vol $\mathbf{84}$, Cambridge University Press, Cambridge, 2004.
\bibitem{BST} G. Bhattacharyya, S. Y. Song, R. Tanaka, Terwilliger algebras of wreath products of one-class association schemes, J. Algebraic Combin. $\mathbf{31}$ (2010), 455-466.
\bibitem{CX} Z. Chen, C. Xi, Structure of Terwilliger algebras of quasi-thin association schemes, J. Combin. Theory Ser. A $\mathbf{213}$ (2025), Paper No. 106024.
\bibitem{CD} B. Curtin, I. Daqqa, The subconstituent algebra of a Latin square, European J. Combin. $\mathbf{30}$ (2009), 447-457.
\bibitem{DK} Y.A. Drozd, V.V. Kirichenko, Finite Dimensional Algebras, Springer-Verlag, Berlin, 1994.
\bibitem{Han} A. Hanaki,  Modular Terwilliger algebras of association schemes, Graphs Combin. $\mathbf{37}$ (2021), 1521-1529.
\bibitem{Her} A. Herman, A survey of semisimple algebras in algebraic combinatorics, Indian J. Pure Appl. Math. $\mathbf{52}$ (2021), 631-642.
\bibitem{J1} Y. Jiang, A note on modular Terwilliger algebras of association schemes, Beitr. Algebra Geom. $\mathbf{63}$ (2022), 829-851.
\bibitem{J2} Y. Jiang, On Terwilliger $\F$-algebras of quasi-thin association schemes, J. Algebraic Combin. $\mathbf{57}$ (2023), 1219-1251.
\bibitem{J3} Y. Jiang, On Terwilliger $\F$-algebras of factorial association schemes, Preprint (2024), \url{
https://doi.org/10.48550/arXiv.2403.08537
}.
\bibitem{K} G. Karpilovsky, The Jacobson Radical of Group Algebras, North-Holland Publishing
 Co., Amsterdam, 1987.
\bibitem{KM}S. Kageyama, Y. Miao, Two classes of $q$-ary codes based on group divisible association schemes, Discrete Math. $\mathbf{195}$ (1999), 269-276.
\bibitem{LMP} F. Levstein, C. Maldonado, D. Penazzi, The Terwilliger algebra of a Hamming scheme $H(d, q)$, European J. Combin. $\mathbf{27}$ (2006), 1-10.
\bibitem{LM} F. Levstein, C. Maldonado, The Terwilliger algebra of the Johnson schemes, Discrete Math. $\mathbf{307}$ (2007), 1621-1635.
\bibitem{LMW} B. Lv, C. Maldonado, K. Wang, More on the Terwilliger algebra of Johnson schemes,  Discrete Math. $\mathbf{328}$ (2014), 54-62.
\bibitem{MA} A. Munemasa, An application of Terwilliger's algebra, Unpublished preprint (1993), \url{http://www.math.is.tohoku.ac.jp/~munemasa/unpublished.html}.
\bibitem{T1} P. Terwilliger, The subconstituent algebra of an association scheme. I, J. Algebraic Combin. $\mathbf{1}$ (1992), 363-388.
\bibitem{T2}P. Terwilliger, The subconstituent algebra of an association scheme. II, J. Algebraic Combin. $\mathbf{2}$ (1993), 73-103.
\bibitem{T3}P. Terwilliger, The subconstituent algebra of an association scheme. III, J. Algebraic Combin. $\mathbf{2}$ (1993), 177-210.
\bibitem{TY}M. Tomiyama, N. Yamazaki, The subconstituent algebra of a strongly regular graph, Kyushu J. Math. $\mathbf{48}$ (1994), 323-334.
\bibitem{W} Y. Watanabe, The generalized wreath product of triply-regular association schemes, The 35th Symposium on Algebraic Combinatorics (Proceedings) (2019), 109-114.
\bibitem{Z} P.-H. Zieschang, An Algebraic Approach to Association Schemes, Lecture Notes in Math., vol. $\mathbf{1628}$, Springer-Verlag, Berlin, 1996.
\end{thebibliography}
\end{document}